\documentclass[12pt]{amsart}
\usepackage[utf8]{inputenc}
\usepackage {amssymb}
\usepackage {amsmath}
\usepackage{amsthm}
\usepackage{graphicx}
\usepackage {amscd}
\usepackage {epic}
\usepackage{xypic}
\usepackage[colorlinks, citecolor=blue]{hyperref}
\usepackage[left=1.5in,right=1.5in,top=1.5in,bottom=1.5in]{geometry}
\usepackage{soul}
\usepackage{comment}
\usepackage{enumitem}
\usepackage{mathrsfs}

\theoremstyle{plain}
\newtheorem{theorem}{Theorem}[section]
\newtheorem{corollary}[theorem]{Corollary}
\newtheorem{lemma}[theorem]{Lemma}

\newtheorem{proposition}[theorem]{Proposition}
\newtheorem{definition-lemma}[theorem]{Definition-Lemma}
\theoremstyle{remark}
\newtheorem{remark}[theorem]{Remark}

\theoremstyle{definition}

\newtheorem{convention}[theorem]{Convention}

\newcommand{\Pic}[0]{\operatorname{Pic}}

\newcommand{\wh}{\widehat}
\newcommand{\wt}{\widetilde}
\def\ep{\varepsilon}
\def\ddbar{\partial\bar\partial}
\def\ve{\varepsilon}

\def\NE{\overline{\operatorname{NE}}}

\def\BC{\operatorname{BC}}
\def\NA{\overline{\operatorname{NA}}}
\def\Null{\operatorname{Null}}
\newcommand{\<}{\leq}
\def\>{\geq}
\newcommand{\mbQ}{\mathbb{Q}}
\newcommand{\mbR}{\mathbb{R}}
\newcommand{\CC}{\mathbb{C}}
\newcommand{\cC}{{\mathcal{C}}}

\def\mcO{\mathcal{O}}
\newcommand{\num}{\equiv}

\newcommand{\dbr}{\bar{\partial}}

\newcommand{\OO}{{\mathcal{O}}}
\newcommand{\Q}{{\mathbb{Q}}}

\newcommand{\R}{{\mathbb{R}}}

\newcommand{\mult}{{\rm mult}}
\newcommand{\Supp}{{\rm Supp}}

\newcommand{\mbC}{\mathbb{C}}
\newcommand{\mbZ}{\mathbb{Z}}

\newcommand{\bir}{\dashrightarrow}
\newcommand{\vphi}{\varphi}
\def\injective{\hookrightarrow}
\def\surjective{\twoheadrightarrow}

\def\lrd{\lfloor}
\def\rrd{\rfloor}

\def\mbN{\mathbb{N}}
\def\mbP{\mathbb{P}}
\def\>{\geq}
\def\ve{\varepsilon}
\def\mcA{\mathcal{A}}
\def\mcO{\mathcal{O}}

\def\mcC{\mathcal{C}}

\def\mcE{\mathcal{E}}

\def\mcH{\mathcal{H}}
\def\mcK{\mathcal{K}}
\def\mcL{\mathcal{L}}
\def\mcM{\mathcal{M}}

\def\mcP{\mathcal{P}}

\def\msI{\mathscr{I}}
\def\msJ{\mathscr{J}}
\def\msL{\mathscr{L}}

\def\eps{\epsilon}
\def\bfP{\textbf{P}}
\def\bfQ{\textbf{Q}}

\def\lrd{\lfloor}
\def\rrd{\rfloor}

\def\pt{\operatorname{pt}}
\def\Ex{\operatorname{Ex}}
\def\dim{\operatorname{dim}}

\def\NA{\operatorname{\overline{NA}}}

\def\NE{\operatorname{\overline{NE}}}
\def\Nef{\operatorname{Nef}}

\def\Div{\operatorname{Div}}
\def\lct{\operatorname{lct}}
\def\LCT{\operatorname{LCT}}
\def\Bs{\operatorname{Bs}}
\def\Zero{\operatorname{Zero}}

\theoremstyle{definition}
\newtheorem{definition}[theorem]{Definition}
\theoremstyle{definition}

\numberwithin{equation}{section}

\theoremstyle{remark}
\newtheorem{claim}[theorem]{Claim}

\author{Omprokash Das}
\address{School of Mathematics\\
Tata Institute of Fundamental Research\\
Homi Bhabha Road, Navy Nagar\\
Colaba, Mumbai 400005}
\email{omdas@math.tifr.res.in}
\email{omprokash@gmail.com}
\thanks{Omprokash Das was supported by the Start--Up Research Grant(SRG), Grant No. \# SRG/2020/000348 of the Science and Engineering Research Board (SERB), Govt. Of India.}

\author{Christopher Hacon}
\address{Department of Mathematics\\
University of Utah\\
155 S 1400 E\\
Salt Lake City, Utah 84112, USA}
\email{hacon@math.utah.edu}
\thanks{Christopher Hacon was partially supported by the NSF research grants no: DMS-1952522, DMS-1801851 and by a grant from the Simons Foundation; Award Number: 256202.}

\author{Mihai P\u aun}
\address{Institut für Mathematik\\ Universität Bayreuth \\ 95440 Bayreuth, Germany
}
\email{mihai.paun@uni-bayreuth.de}
\thanks{Mihai P\u aun gratefully acknowledges support from the DFG}

\title[$4$-dimensional K\"ahler MMP]{On the $4$-dimensional minimal model program for K\"ahler varieties}

\begin{document}
\maketitle
\begin{abstract}
 	In this article we establish the following results: Let $(X, B)$ be a dlt pair, where $X$ is a $\mbQ$-factorial K\"ahler $4$-fold -- (i) if $X$ is compact and $K_X+B\sim_{\mbQ} D\>0$ for some effective $\mbQ$-divisor, then $(X, B)$ has a log minimal model, (ii) if $(X/T, B)$ is a semi-stable klt pair, $W\subset T$ a compact subset and $K_X+B$ is effective over $W$ (resp. not effective over $W$), then we can run a $(K_X+B)$-MMP over $T$ (in a neighborhood of $W$) which ends with a minimal model over $T$ (resp. a Mori fiber space over $T$). We also give a proof of the existence of flips for analytic varieties in all dimensions and the relative MMP for projective morphisms between analytic varieties.

 \end{abstract}

\tableofcontents

\date{\today}

\section{Introduction}
In recent years there has been substantial progress towards the minimal model program for complex projective varieties of arbitrary dimension \cite{BCHM10}.
Unluckily, much less is known about the minimal model program for K\"ahler varieties.
In dimension 3, the situation is now well understood, including the cone theorem, the base point free theorem, the existence of flips and divisorial contractions and the termination of flips (see \cite{HP16}, \cite{CHP16}, \cite{DO23}, \cite{DH20} and references therein).
In higher dimension, however the situation is less clear.
Recently, however, Fujino proved the minimal model program for projective morphisms between complex analytic spaces (of arbitrary dimension) \cite{Fuj22}.

In this paper we take the first steps towards proving that the minimal model program holds for K\"ahler 4-folds. In particular we show that it holds for effective dlt pairs, and for (strongly) semistable families of $3$-folds over curves.

\begin{theorem}\label{thm:effective-dlt-mm}
Let $(X, B)$ be a $\mbQ$-factorial compact K\"ahler $4$-fold dlt pair such that $K_X+B\sim_{\mbQ} M\>0$. Then $(X, B)$ has a log minimal model.
\end{theorem}

\begin{theorem}\label{thm:ss-mmp}
Let $f:(X,B)\to T$ be a $\mbQ$-factorial  
semi-stable klt pair of dimension $4$ and $W\subset T$ a compact subset (see Definition \ref{def:klt-semi-stable-pair}). If $K_X+B$ is effective  (resp. not effective) over $W$ (see Lemma \ref{l-psef}), then we can run the $(K_X+B)$-MMP over a neighborhood of $W$ in $T$ which ends with a minimal model over $W$ (resp. with a Mori fiber space over $W$).
\end{theorem}

The main idea for the proof of Theorem \ref{thm:effective-dlt-mm} is as follows.  If $K_X+B\sim _{\mathbb Q}M\geq 0$, then running the minimal model program for $K_X+B$ is equivalent to running the minimal model program for $K_X+B+\lambda M$ for any $\lambda >0$. Suppose for simplicity that  
$(X,{\rm Supp}(B+M))$ has simple normal crossings and $(X,B+\lambda M)$ is dlt for some $\lambda >0$ such that the support of $\lfloor B+\lambda M\rfloor$ is equal to the support of $M$. It then follows that $K_X+B$ is nef if and only if $K_X+B+\lambda M$ is nef. If this is not the case, then we show that there is a $K_X+B$ negative extremal ray $R$ spanned by a rational curve $C$ such that $M\cdot C<0$ and hence $S\cdot C<0$ for a component $S$ of $M$ and hence of $\lfloor B+\lambda M\rfloor$. By adjunction $K_S+B_S:=(K_X+B+\lambda M)|_S$ is a divisorially log terminal 3-fold. We can now apply the 3-dimensional minimal model program to the pair $(S,B_S)$ and in particular we have a contraction $S\to T$ corresponding to the $K_S+B_S$ negative extremal face $F$ spanned by the curves of the ray $R$ contained in $S$.
Since $S\cdot C<0$, we are able to extend this to a contraction $X\to Y$ of the ray $R$. If this is a divisorial contraction, we replace $X$ by $Y$ and repeat the procedure. Otherwise we have a flipping contraction, which is in particular a projective morphism and hence its flip $X\dasharrow X^+$ exists by \cite{Fuj22} (see also Theorem \ref{t-mmpscale} below). We then replace $X$ by $X^+$ and repeat the procedure. In order to conclude it is necessary to show the termination of the corresponding sequences of flips. This follows along the usual approach by using special termination, the acc for log canonical thresholds, and termination of flips in dimension 3.
Some of the ideas in this approach are inspired by the approach for projective varieties \cite{BCHM10}, \cite{Bir07}, and \cite{Bir10}, but not surprisingly many new technical issues arise in the context of K\"ahler varieties. Regarding Theorem \ref{thm:ss-mmp}, we simply remark that according to our definition of a  
semi-stable klt pair $f:(X,B)\to T$, for any $t\in W$, $(X,X_t+B)$ is a plt pair, thus $K_{X_t}+B_t=(K_X+X_t+B)|_{X_t}$ is a klt 3-fold and so we can reduce questions on the existence of the relative $(X,B)$ minimal model program to known results about the 3-fold minimal model program for $(X_t,B_t)$. Termination of flips when $K_X+B$ is not effective over $W$ is the most challenging part of this proof as the usual approach does not immediately apply here.\\

We will also use the results of \cite{Nak87} and recent advances in the minimal model program to prove the following results conjectured in \cite{Nak87}.

\begin{theorem}[Finite generation  conjecture]\label{c-fg} Let $f: X \to Y$ be a proper surjective morphism of analytic varieties where $X$ is in Fujiki's class $\mathcal C$. Suppose that $(X,B)$ is a klt pair. Then the relative canonical $\mathcal O _Y$-algebra \[R(X/Y,K_X+B):=\oplus_{m\>0} f_*\mathcal O _X(m(K_X+B))\] is locally finitely generated. 
\end{theorem}

\begin{theorem}\label{t-mmpscale}
Let $\pi :X\to U$ be a projective morphism of normal varieties and $B\geq 0$ a $\Q$-divisor such that $(X,B)$ is klt. Let $W\subset U$ be a compact subset such that $\pi :X\to U$ satisfies property $\mathbf P$ or $\mathbf Q$ (see Definition \ref{def:property-pq}) and $X$ is $\mbQ$-factorial near $W$ (cf. \ref{def:Q-factorial}), then after shrinking $U$ in a neighborhood of $W$, 
\begin{enumerate}
    \item we can run the $K_X+B$ MMP over $U$,
    \item if $K_X+B$ is pseudo-effective, and either $B$ or $K_X+B$ is big over $U$, then any MMP with scaling of a relatively ample divisor terminates with a minimal model, and
    \item if $K_X+B$ is not pseudo-effective over $U$, then any MMP with scaling of a relatively ample divisor terminates with a Mori fiber space.
\end{enumerate}
\end{theorem}

\begin{remark}
After completing the proofs of Theorems \ref{c-fg} and \ref{t-mmpscale}, we were informed that Fujino has also proved these results see \cite{Fuj22}. We note that Fujino's approach is based on \cite{BCHM10} whereas our approach is inspired by \cite{CL10}. Another possible approach can be found in \cite{Pau12}, which is particularly suited to the analytic context.
\end{remark}
 
 This article is organized in the following manner:  In Part 1, we collect and prove various preliminary results. In Subsection 2.4 we prove two important results, namely Theorem \ref{thm:nef-big-to-kahler} and \ref{thm:nef-restricts-to-pseff}. These two results work as our main tools for testing whether a $(1, 1)$ class $\alpha$ is nef or not, see Remark \ref{rmk:nef-criteria} for more details. Part 2 of the article is devoted to proving finite generation as in \cite{CL10}. We prove Theorem \ref{c-fg} and \ref{t-mmpscale} in Section 4 of this part. In Part 3, we prove Theorem \ref{thm:effective-dlt-mm} (in Section 7) and Theorem \ref{thm:ss-mmp} (in Section 8).\\  
 
{\bf Acknowledgment.} O. Das would like to thank Cristian Martinez for many useful discussions. We would also like to sincerely thank the referees for their careful reading of our paper and many detailed suggestions for improvements.

\part{Preliminaries}
\section{Preliminaries}
A \textit{complex analytic variety} or simply an \textit{analytic variety} is a reduced and irreducible complex space. All complex spaces in this article are assumed to be \textit{second countable} spaces. A holomorphic map $f:X\to Y$ between complex spaces is called a \textit{morphism}. An open subset $U\subset X$ is called a Zariski open set if the complement $Z=X\setminus U$ is a closed analytic subset of $X$, i.e. there is a sheaf of ideals $\msI_Z\subset\mcO_X$ such that $Z=\Supp(\mcO_X/\msI_Z)$. Let $\mcP$ be a property. We say that  \textit{general} points of $X$ satisfy $\mcP$ if there is a dense Zariski open subset $U\subset X$ such that $\mcP$ is satisfied for all $x\in U$. We say that \textit{very general} points of $X$ satisfy $\mcP$ if there is a countable collection of dense Zariski open subsets $\{U_i\}_{i\in I}$ of $X$ such that $x\in X$ satisfies $\mcP$ for all $x\in \cap_{i\in I} U_i$. Similarly, if $f:X\to Y$ is a morphism between complex spaces, we say that \textit{general fibers} of $f$ satisfy $\mcP$ if there is a dense Zariski open subset $U\subset Y$ such that $X_y:=f^{-1}(y)$ satisfies $\mcP$ for all $y\in Y$; very general fibers are defined analogously.

Let $S\subset X$, then we say that $S$ is \textit{uncountably Zariski dense} in $X$ if $S$ is not contained in any countable union of closed analytic subsets of $X$. Note that, if $S\subset X$ is uncountably Zariski dense, then for any non-empty Zariski open subset $U\subset X$, $S\cap U\neq\emptyset$.\\

\begin{definition}\label{def:kahler-variety}
Let $X$ be an analytic variety. Then $X$ is called a K\"ahler variety if the there is a K\"ahler form on $X$, i.e. a positive closed real $(1, 1)$ form $\omega\in \mcA_{\mbR}^{1,1}(X)$ such that the following holds: for every $x\in X$, there is an open neighborhood $x\in U\subset X$ and a closed embedding $\iota: U\to V$ into an open subset of $\mbC^N$, and a strictly plurisubharmonic $C^\infty$ function $f:V\to \mbR$ such that $\omega|_{U\cap X_\textsubscript{sm}}=(i\partial\bar\partial f)|_{U\cap X_\textsubscript{sm}}$.  
\end{definition}

\noindent
\begin{enumerate}
    \item For a compact analytic variety $X$, $N^1(X)$ is defined to be the Bott-Chern cohomology group $H^{1,1}_{\BC}(X)$ (which is also an $\mbR$-vector space), see \cite[Definition 3.1]{HP16}. $N_1(X)$ is defined in \cite[Definition 3.8]{HP16}. When $X$ is a normal compact analytic variety with rational singularities and belongs to Fujiki's class $\mcC$, the duality of $N^1(X)$ and $N_1(X)$ is established in \cite[Proposition 3.9]{HP16}. 
    
    \item Let $X$ be a compact analytic variety. Let $u\in H^{1,1}_{\BC}(X)$ be a class represented by a form $\alpha$ with local potentials. Then $u$ is called nef if for some positive $(1,1)$ smooth form $\alpha$ and for every $\eps>0$, there exists a smooth function $f_\eps\in\mcA^0(X)$ such that
    \[
    \alpha+i\partial\bar{\partial}f_\eps\>-\eps\omega.
    \]
If $X$ is in Fujiki's class $\mcC$, then we denote by $\Nef(X)\subset N^1(X)$ the cone of nef cohomology classes.

    \item For the definitions of  big and pseudo-effective classes and the corresponding cones, see \cite{HP16}, \cite[Definition 2.2]{DH20} and also Subsection \ref{subs:analytic-classes}. 
    
    \item Let $D=\sum a_i D_i$ and $D'=\sum a'_iD_i$ be two $\mbR$-divisors on a normal analytic variety $X$. Then we define $D\wedge D'$ as 
\[
D\wedge D':=\sum_i\min\{a_i, a'_i\} D_i.
\]
\end{enumerate}

\begin{remark}\label{rmk:nef-criteria}
    Let $X$ be a normal compact K\"ahler variety, and $B\>0$ be an effective divisor such that $K_X+B$ is $\mbQ$-Cartier. Under mild singularity assumptions on the pair $(X, B)$, the Minimal Model Program asks whether $K_X+B$ is nef or not. When $X$ is a projective variety, nefness of a $\mbQ$-Cartier divisor $D$ can simply be tested by checking whether $D\cdot C$ is non-negative (or not) for all curves $C\subset X$. However, in general for compact K\"ahler varieties this criteria is not equivalent to Definition \ref{def:kahler-variety}(2), for a counterexample see \cite[Page 5]{HP17}. When $\dim X=3$, using Boucksom's divisorial Zariski decomposition \cite{Bou04} it is shown in \cite{HP16, CHP16} (also see \cite[Lemma 2.7]{DH20}) that $K_X+B$ is nef if and only if $(K_X+B)\cdot C\>0$ for all curves $C\subset X$. This result is expected to be true  in $\dim X\>4$, but a proof is not yet known, the proof in dimension $3$ does not automatically extend in higher dimensions; for a partial result in higher dimensions see \cite{CH20}.
   
    In absence of such a nefness criteria we use our Theorem \ref{thm:nef-restricts-to-pseff} to test whether a class $\alpha\in H^{1,1}_{\BC}(X)$ is nef or not; it says that $\alpha$ is nef if and only if $\alpha|_V$ is a pseudo-effective class for all analytic varieties $V\subset X$.    
\end{remark}

The following results about nefness will be used throughout the article.

\begin{lemma}\cite[Remark 3.12]{HP16}\label{lem:nef-cone}
Let $X$ be a normal compact K\"ahler variety, $\Nef(X)$ is the cone of nef classes in $H^{1,1}_{\BC}(X)$ and $\mcK(X)$ is the (open) cone of K\"ahler classes. Then $\overline{\mcK(X)}=\Nef(X)$.
\end{lemma}

\begin{proposition}\label{pro:nef-mori-cone-duality}
Let $X$ be a normal compact K\"ahler variety with rational singularities. Then $\Nef(X)$ and $\NA(X)$ are dual to each other via the natural isomorphism $N^1(X)\to N_1(X)^*$ induced by their usual perfect pairing.
\end{proposition}

\begin{proof}
A similar proof as in \cite[Proposition 3.15]{HP16} holds here. Note that the main ingredient of the proof of \cite[Proposition 3.15]{HP16} is Lemma 3.13 in \cite{HP16}, for which we use Lemma \ref{lem:nef-pullback}.
\end{proof}

\begin{definition}\label{def:finite-generation}
Let $X$ be a complex space and $\mathcal R$ a graded sheaf of $\mcO_X$-algebras. We say that $\mathcal R$ is \textit{locally finitely generated}, if for every $x\in X$ there is an open neighborhood $x\in U$ such that $\mathcal R(U)$ is a finitely generated $\mcO_X(U)$-algebra. We say $\mathcal R$ is \textit{finitely generated}, if there exists an integer $m\>0$ such that for every $x\in X$, there is an open neighborhood $x\in U$ such that $\mathcal R(U)$ generated by elements of degree $\<m$.
\end{definition}

\begin{remark}\label{rmk:finite-generation}
    Note that finite generation is a necessary condition for the existence of $\mbox{Projan\ }\mathcal R\to X$. If $W\subset X$ is a compact subset, then for any locally finitely generated graded algebra $\mathcal R$, there is an open neighborhood $U\supset W$ of $W$ such that $\mathcal R|_U$ is a finitely generated $\mcO_U$-algebra. Indeed, since $W$ is compact, it can be covered by finitely many open sets $\{U_i\}_{1\<i\<k}$ such that $\mathcal R(U_i)$ is a finitely generated $\mcO_X(U_i)$-algebra. Now let $m_i\>0$ be an integer such that each $\mathcal R(U_i)$ is generated by degree $\<m_i$ monomials. Then $m:=\{m_i\;:\; i=1,2,\ldots, k\}$  does the job.\\

\end{remark}

\begin{definition}\label{def:Q-factorial}
Let $X$ be a normal analytic variety. The canonical sheaf $\omega_X$ is defined as $\omega_X:=(\wedge^{\dim X} \Omega^1_X)^{**}$. Note that unlike the case of algebraic varieties, $\omega_X$ here does not necessarily correspond to a Weil divisor $K_X$ such that $\omega_X\cong \mcO_X(K_X)$. However, by abuse of notation we will say that $K_X$ is a canonical divisor when we actually mean the canonical sheaf $\omega_X$. This doesn't create any problem in general as running the minimal model program involves intersecting subvarieties with $\omega_X$. 
\begin{enumerate}
    \item A $\mbQ$-divisor $D$ on $X$ is called $\mbQ$-Cartier if $mD$ is Cartier for some $m\in\mbN$.  We say $X$ is $\mbQ$-factorial, if every prime Weil divisor $D$ on $X$ is $\mbQ$-Cartier and there is a positive integer $m>0$ such that $(\omega_X^{\otimes m})^{**}$ is a line bundle. Note that if $X$ is $\mbQ$-factorial and $U\subset X$ is an open subset, then $U$ is not necessarily $\mbQ$-factorial.
    
    \item A $\mbQ$-divisor $D$ is called $\mbQ$-Cartier at a point $x\in X$, if there is an open neighborhood $x\in U\subset X$ such that $D|_U$ is $\mbQ$-Cartier.
    
    \item Let $\pi :X\to T$ be a projective morphism of complex varieties, $X$ is normal and $W\subset T$ a compact subset. We say that $X$ is $\mathbb Q$-factorial over $W$, if every divisor $D$ defined on a neighborhood of $\pi ^{-1}(W)$ is $\mathbb Q$-Cartier at every point $x\in \pi ^{-1}(W)$ and $\omega _X$ is also $\mathbb Q$-Cartier at every point $x\in \pi ^{-1}(W)$.
    
    \item A pair $(X, \Delta)$ consists of a normal variety $X$ and a $\mbQ$-divisor $\Delta$ such that $K_X+\Delta$ is $\mbQ$-Cartier. The singularities of $(X, \Delta)$ are defined exactly the same way as in \cite[Chapter 2]{KM98}. Note that in this article when we say that a pair $(X, \Delta)$ is klt, we assume that $\Delta$ is an \textit{effective} divisor. If $\Delta$ is not necessarily effective, then we will call $(X, \Delta)$ a \textit{sub-klt pair}. Similar conventions are made for other classes of singularities. If $E$ is a divisor over $X$, the \textit{discrepancy} of $E$ with respect to $(X, \Delta)$ will be denoted by $a(E, X, \Delta)$. 
    
    \item We will often abuse notation and simply say that $K_X+\Delta$ is klt (instead of the pair $(X, \Delta)$ is klt).
\end{enumerate}

\end{definition}

\begin{definition}\label{def:log-terminal-and-log-minimal-model}
Let $(X, B)$ be a log canonical pair and $\phi:X\bir Y$ a bimeromorphic map. Let $B_Y$ be the push-forward of $B$ under $\phi$ and $E_Y=\sum E_j$ the sum of all prime Weil divisors on $Y$ which are contracted by $\phi^{-1}:Y\bir X$.
\begin{enumerate}
    \item We say that $(Y, B_Y+E_Y)$ is a \textit{nef model} if $(Y, B_Y+E_Y)$ is a $\mbQ$-factorial dlt pair and $K_Y+B_Y+E_Y$ is nef.\\
    
    \item  We say that $(Y, B_Y+E_Y)$ is a \textit{log minimal model} if it is a nef model and for any prime Weil divisor $E\subset X$ which is contracted by $\phi$, $a(E, X, B)<a(E, Y, B_Y+E_Y)$ holds.\\

    \item We say $(Y, B_Y)$ is a \textit{log terminal model} of $(X, B)$ if the following hold:
   \begin{enumerate}
       \item[(i)] $(X, B)$ is a $\mbQ$-factorial dlt pair,
      \item[(ii)]  $K_Y+B_Y$ is nef, 
      \item[(iii)] $\phi$ does not extract any divisor, i.e. $\phi^{-1}:Y\bir X$ does not contract any divisor, and
      \item[(iv)] for any prime Weil divisor $E\subset X$ which is contracted by $\phi$, $a(E, X, B)<a(E, Y, B_Y)$ holds.\\
   \end{enumerate}
\end{enumerate}

Clearly, every log terminal model is a log minimal model, and every log minimal model is a nef model. 

\end{definition}

The following result shows that if $(X, B)$ is a plt pair, then every log minimal model of $(X, B)$ is a log terminal model.
\begin{lemma}\label{lem:lmm-to-ltm}
Let $(X, B)$ be a log canonical pair and $(Y, B_Y+E_Y)$ be a log minimal model of $(X, B)$ as in Definition \ref{def:log-terminal-and-log-minimal-model} above. Let $\phi:X\bir Y$ be the induced bimeromorphic map. Then the following holds:
\begin{enumerate}
    \item For any prime Weil divisor $E$ over $X$, $a(E, X, B)\<a(E, Y, B_Y+E_Y)$.
    \item If $(X, B)$ is a plt pair, then $E_Y=0$, i.e. $\phi^{-1}$ does not contract any divisor; in particular, $(Y, B_Y)$ is a log terminal model of $(X, B)$.
\end{enumerate}
\end{lemma}

\begin{proof}
(1) Let $W$ be the normalization of the graph of $\phi$, and $p:W\to X$ and $q:W\to Y$ be the induced bimeromorphic morphisms. Then we can write $K_W=p^*(K_X+B)+G$ and $K_W=q^*(K_Y+B_Y+E_Y)+H$. Note that $p_*G=-B$ and $q_*H=-(B_Y+E_Y)$. Thus we have
\[
p^*(K_X+B)=q^*(K_Y+B_Y+E_Y)+H-G.
\]
Therefore $-(H-G)\num_p q^*(K_Y+B_Y+E_Y)$ is $p$-nef, and $p_*(H-G)=p_*H+B$. Let $D$ be a component of $H$. If $q_*D$ is a component of $B_Y$, then $p_*D\neq 0$, and the coefficient of $p_*D$ in $p_*H+B$ is $0$. If $q_*D=0$ and $p_*D\neq 0$, then from the definition of log minimal model it follows that $a(D, Y, B_Y+E_Y)>a(D, X, B)$. In particular, the coefficient of $p_*D$ in $p_*H+B$ is positive. Thus $p_*(H-G)$ is an effective divisor, and hence from the negativity lemma it follows that $H-G$ is an effective divisor. Thus for any prime Weil divisor $E$ over $X$ we have $a(E, X, B)\<a(E, Y, B_Y+E_Y)$.\\

\noindent
(2) Let $E_i$ be a component of $E_Y$. Then $E_i$ is an exceptional divisor over $X$, in particular, $a(E_i, X, B)>-1$, since $(X, B)$ is plt. But from part (1) it follows that $-1<a(E_i, X, B)\<a(E_i, Y, B_Y+E_Y)=-1$. This is a contradiction, and hence $E_i=0$ for all $i$, i.e. $E_Y=0$, i.e. $\phi^{-1}$ does not contract any divisor.\\

\end{proof}

\begin{convention}\label{con:relatively-compact}
We say that a complex space $X$ is \textit{relatively compact} if there is another complex space $Y$ such that $X$ is an open subspace of $Y$ and the closure $\overline{X}\subset Y$ is compact.
We will say that $f:X\to U$ is a morphism from a complex space $X$ to a \textit{relatively compact} space $U$, if there exists a morphism $f':X'\to U'$ of complex spaces such that $U\subset U'$ is a relatively compact open subset of $U'$, $X=X'\times_{U'} U$ and $f$ is the induced morphism to $U$. 
\end{convention}

\subsection{Projective morphisms} \cite[Chapter II, Page 24]{Nak04}\label{s-pm}
Let $f:X\to Y$ be a proper morphism of complex spaces. A line bundle $\msL$ on $X$ is called \textit{$f$-free} or \textit{$f$-generated} if the natural morphism $f^*f_*\msL\to \msL$ is surjective. We say $\msL$ is \textit{$f$-very ample} or \textit{very ample over $Y$} if $\msL$ is $f$-free and $X\to \mbP_Y(f_*\msL)$ is a closed embedding. We say that $\msL$ is \textit{$f$-ample} or \textit{ample over $Y$} if for every $y\in Y$, there is an open neighborhood $y\in V$ and a positive integer $m>0$ such that $\msL^m|_{f^{-1}V}$ is very ample over $V$. A proper morphism $f:X\to Y$ of complex spaces is called \textit{projective}, if there exists a $f$-ample line bundle $\msL$ on $X$. The morphism $f:X\to Y$ is called \textit{locally projective} if $Y$ has a open cover $\{U_i\}$ such that $f|_{X_{U_i}}:X_{U_i}\to U_i$ is projective for all $i$, where $X_{U_i}:=f^{-1}U_i$.  

\begin{remark}\label{rmk:projective-morphism}
Note that the composition of two projective morphisms are not necessarily projective, see \cite[Page 557]{Nak87} for a counterexample. However, the composition of two locally projective morphisms is locally projective. On the other hand, if $f:X\to Y$ and $g:Y\to Z$ are two projective morphisms of complex spaces and $K\subset Z$ is a compact subset, then over a neighborhood of $K$, $g\circ f$ is projective. 
\end{remark}

We have the following properties of $f$-ample line bundles.
\begin{theorem}\label{t-rel-ample}
Let $f:X\to Y$ be a projective morphism of complex spaces, $L$ an $f$-ample line bundle,  $F$ a coherent sheaf and for any integer $m$ let $F(m)=F\otimes L^m$. Then, for any compact subset $K\subset Y$ there exists an integer $m_0=m_0(K,F)$ such that
\begin{enumerate}
 \item $f^*f_*(F(m))\to F(m)$ is surjective for any point $x\in X_K:=f^{-1}K$ and any $m\geq m_0$,
\item $R^if_*(F(m))=0$ on a neighborhood of $K$ for any $i\geq 1$ and $m\geq m_0$,
\item if $U\subset Y$ is a relatively compact Stein open subset, then $F(m)|_{f^{-1}(U)}$ is globally generated and $H^i(f^{-1}(U),F(m))=0$ for any $i\geq 1$ and $m\geq m_0$,
\item if $F$ is invertible, then $F(m)$ is ample (resp. very ample) over a neighborhood of $K$ for all $m\geq m_0$,
\item if $U\subset Y$ is a relatively compact Stein open subset, $S$ is a normal subvariety of $X$ and $D$ is Cartier on $X$, then $|(D+mL)_{|{S_U}}|=|D+mL|_{S_U}$ for all $m\geq m_0$, where ${S_U}=S\cap f^{-1}(U)$.
\end{enumerate}
\end{theorem}
\begin{proof}
(1-2) are standard results due to Grauert and Remmert, for example, see \cite[IV Theorem 2.1]{BS76}.

For (3) recall that if $G$ is a coherent sheaf on a Stein space $U$, then by Cartan's theorem, $G$ is globally generated and $H^p(U,G)=0$ for every $p>0$. 
Since $U$ is relatively compact, then by (2) we have $R^if_*(F(m))|_U=0$  for any $i\geq 1$ and $m\geq m_0$ and so by a spectral sequence argument $H^i(f^{-1}(U),F(m))=H^i(U,f_*(F(m)))=0$ for $i>0$, since $f_*F(m)$ is coherent on $U$.
By (1) we have $f^*f_*(F(m))\to F(m)$ is surjective over $U$, and since $U$ is Stein, $f_*F(m)|_U$ is globally generated and hence so is $f^*f_*(F(m))|_{f^{-1}(U)}$. In particular, $F(m)|_{f^{-1}(U)}$ is globally generated.\\

(4) If $F$ is invertible, then by (1) we may assume that $F(m)$ is $f$-generated and hence $f$-nef over a neighborhood of $K$ for $m\geq m_0$. But then $F(m+1)$ is $f$-ample (as it is the tensor product of an $f$-nef and an $f$-ample line bundle). The very ampleness statement follows similarly.\\ 

(5) The inclusion $|(D+mL)_{|{S_U}}|\supset |D+mL|_{S_U}$ is immediate from the definitions.
Consider the short exact sequence
\[ 0\to \OO _{X_U}(D-S)\to \OO _{X_U}(D)\to \OO _{S_U}(D|_{S_U})\to 0.\]
Twisting this sequence by $\OO _{X_U}(mL)$ and then pushing forward by $f$ we obtain the following surjectivity from (1):
\[
\xymatrixcolsep{3pc}\xymatrix{H^0(X_U, \mcO_{X_U}(D+mL))\ar@{->>}[r] & H^0(S_U, \mcO_{S_U}((D+mL)|_{S_U})). }
\]

Thus the reverse inclusion holds.

\end{proof}~\\

\subsubsection{Resolutions of singularities}
If $X$ is a complex manifold and $D\subset X$ is a divisor, then $D$ has simple normal crossing support if for any point $x\in X$ we may choose local parameters $z_1,\ldots ,z_n$ such that ${\rm Supp }(D)$ is locally defined by the vanishing of $z_1\cdot \ldots \cdot z_r$.
In this case we say that $(X,D)$ is log smooth.
If $(X,D)$ is a pair, then we let  $SNC(X, D)$ be the open subset of points $x\in X$ such that $(X,D)$ is log smooth on a neighborhood of $x\in X$.

\begin{theorem}[Log Resolution]\cite[Thm. 13.2, 1.10 and 1.6]{BM97}\cite[Thm. 2.16]{DH20}\label{thm:log-resolution}
	Let $X\subset W$ be a relatively compact open subset of an analytic variety $W$ and $D$ a $\mbQ$-Cartier divisor on $X$. Then there exists a projective bimeromorphic morphism $f:Y\to X$ from a smooth variety $Y$ satisfying the following properties:
	\begin{enumerate}
		\item $f$ is a successive blow up of smooth centers contained in $X\setminus SNC(X, D)$,
		\item $f^{-1}(SNC(X, D))\cong SNC(X, D)$, and
		\item $\Ex(f)$ is a pure codimension $1$ subset of $Y$ such that $\Ex(f)\cup(f^{-1}_*D)$ has SNC support.
	\end{enumerate}
	
\end{theorem}

\begin{remark}\label{r-log-resolution}
 { Note that if $\mathcal J\subset \mathcal O _X$ is a sheaf of ideals, then there exists a projective bimeromorphic morphism $f:Y\to X$ from a smooth variety $Y$ such that $\mathcal J\cdot \mathcal O _Y= \mathcal O _Y(-G)$ where $(Y,G+{\rm Ex}(f))$ is log smooth.  To see this, simply blow up $\mathcal J$ to get $f_1:X_1\to X$ such that $\mathcal J\cdot \mathcal O _{X_1}= \mathcal O _{X_1}(-D)$ (note that by \cite[Theorem 1.10]{BM97} we can also achieve this step by a finite sequence of blow ups along smooth centers). Then apply Theorem \ref{thm:log-resolution} to obtain $g:Y\to X_1$ so that $(Y,f^{-1}_*D+{\rm Ex}(f))$ is log smooth.}
 
\end{remark}
\begin{lemma}\label{l-log-resolution}
      Let $\pi :X\to U$ be a projective morphism from a smooth complex variety to a relatively compact Stein variety. Let $V\subset |L|$ be a non-empty linear series, then there exists a projective birational morphism $f:X'\to X$ such that ${\rm Fix}f^*V$ is a divisor with simple normal crossings and ${\rm Mob}(f^*V)$ is $\pi \circ f$-free.
\end{lemma}
\begin{proof}
Let $\mathcal V\subset H^0(X,L)$ be the vector space corresponding to the linear series $V$.
Let $\mathfrak b$ be the base ideal of $\mathcal V$ so that $\mathcal V\cdot \mathcal O_X\to L\otimes \mathfrak b$ is surjective.
Let $f:Y\to X$ be a resolution of $\mathfrak b$ so that $\mathfrak b\cdot \mathcal O _{X'}=\mathcal O _{X'}(-F)$, where $F$ is a divisor with simple normal crossings. Then $f^*\mathcal V\otimes \mathcal O _{X'}\to f^*L\otimes \mathcal O _{X'}(-F)$ is surjective, and thus $f^*L\otimes \mathcal O _{X'}(-F)$ is globally generated. In particular, ${\rm Fix}f^*\mathcal V=F$ is a divisor with simple normal crossings and ${\rm Mob}(f^*\mathcal V)$ is $\pi \circ f$-free.
\end{proof}

 \subsection{Bertini's theorem}

In this subsection we will prove a analytic version of Bertini's theorem which will be useful in what follows. First we need the following definitions.
\begin{definition}
Let $X$ be a complex space. A subset $W\subset X$ is called \emph{analytically meager}, if there exist countably many locally analytic subsets $\{Z_i\}_{i\in \mbN}$ of $X$ of codimension $\>1$ such that $W\subset \cup_{i=1}^\infty Z_i$. Clearly, a countable union of analytically meager sets is analytically meager.

\end{definition}

\begin{definition}
Let $X$ be a complete metric space. A subset $M\subset X$ is called \textit{fat}, if there are countably many dense open subsets $\{U_i\}$ of $X$ such that $\cap_i U_i\subset M$. Clearly, countable intersections of fat sets are fat. Let $\mcP$ be a property. We say that  \textit{sufficiently general} points of $X$ satisfy $\mcP$, if there exists a fat subset $M\subset X$ such that $x\in X$ satisfies $\mcP$ for all $x\in M$.
Note that, since $X$ is a complete metric space, by Baire's theorem any fat set is dense in $X$. From \cite[Remark II.3, Page 276]{Man82} we know also that if $M$ is a fat subset of $X$, then $X\setminus M$ is analytically meager. 
\end{definition}
\begin{remark}\label{rmk:sifficiently-general}
    Let $X$ be a complex space and $\msL$ a line bundle on $X$. Let $V\subset H^0(X, \msL)$ be a finite dimensional $\mbC$-subspace. By abuse of terminology we will say that a \textit{sufficiently general member} $D$ of the linear system $|V|$ satisfies property $\mcP$, if for a sufficiently general member $s\in V$, $D=\Zero(s)\subset X$ satisfies property $\mcP$.
\end{remark}

\begin{remark}\label{rmk:meager-set}
    Note that, if $W\subset X$ is an analytically meager set, then $W$ is nowhere dense in $X$, i.e. the interior of the closure $\overline{W}$ is an empty set. Consequently, $X\setminus W$ is dense in $X$. Moreover, if $f:X\to Y$ is a surjective morphism between complex spaces and $W\subset Y$ is an analytically meager set, then $f^{-1}W$ is an analytically meager subset of $X$. Let $g:X\to Y$ be a surjective continuous map between complete metric spaces and $M\subset Y$ is a fat subset. By definition, $M$ contains a countable intersection of dense open subsets, say $\cap U_i$ of $Y$. Then $Y\setminus \cap U_i$ is an analytically meager set, and thus $g^{-1}(Y\setminus \cap U_i)$ is also analytically meager in $X$. In particular, $X\setminus g^{-1}(Y\setminus \cap U_i)$ is a dense subset of $X$, and hence $\cap g^{-1}U_i$ is dense in $X$. Therefore $g^{-1}M$ is a fat subset of $X$.\\
\end{remark}

\begin{theorem}\label{t-bertini+}
Let $\pi :X\to U$ be a projective morphism from a smooth complex variety to a relatively compact Stein space. Let $D$ be a simple normal crossings divisor on $X$ and $L$ a $\pi$-generated line bundle on $X$. Then the following hold:
\begin{enumerate}
    \item  If $\dim X=n$, then there exist sections $s_0,\ldots , s_n\in H^0(X, L)$ generating $L$.
    \item Let $V\subset H^0(X,L)$ be a finite dimensional $\mathbb C$-subspace  such that $V$ generates $L$, i.e. $V\otimes _{\mathbb C}\OO_X\to L$ is surjective. Then $(X, D+G)$ is log smooth for all sufficiently general members $G\in |V|$.  
\end{enumerate}
\end{theorem}

\begin{proof} 
The proof of (1) is well known and hence we will just prove (2). 
 Let $Z\subset X$ be a positive dimensional strata of $D$. Then $V|_Z:=\{s|_Z\; |\; s\in V\}\subset H^0(Z, L|_Z)$ generates $L|_Z$ globally and there is a surjection $\vphi:V\surjective V|_Z$ of vector spaces. Since $V$ (and hence also $V|_Z$) is a finite dimensional $\mbC$-vector space, fixing some norms on $V$ and $V|_Z$ we may assume that $\vphi$ is a surjective continuous linear transformation between two Banach spaces. Then by \cite[Theorem II.5]{Man82}, there exists a fat set $M\subset V|_Z$ such that the zero set $\Zero(s|_Z)\subset Z$ is smooth for all $s|_Z\in M$. Then from Remark \ref{rmk:meager-set} it follows that $\vphi^{-1}M$ is a fat subset of $V$. Let $K:=\ker(\vphi)$; since $V\setminus K\to V|_Z\setminus\{0\}$ is surjective, it follows that $\Zero(s)|_Z$ is smooth for all $s\in \vphi^{-1}M\setminus K$. Note that $V\setminus K$ is a dense open subset of $V$, since $K$ a proper closed subspace of $V$; in particular, $\vphi^{-1}M\setminus K$ is fat subset of $V$. Since there are only finitely many strata of $D$, by induction on the number of positive dimensional strata of $D$, it follows that there is a fat subset $N\subset V$ such that $(X, D+G)$ is log smooth for all $s\in N$ with $\Zero(s)=G$.

\end{proof}

\begin{lemma}\label{l-klt}
Let $f:X\to Y$ be a projective morphism of complex spaces such that $Y$ is a relatively compact Stein space, and $L$ is a $f$-generated line bundle. Let $V\subset H^0(X, L)$ be a finite dimensional $\mbC$-subspace such that $L$ is globally generated by the sections of $V$. If $(X, B)$ is klt, then 
for sufficiently general member $D\in |V|$, $(X,B+tD)$ is klt for any $t<1$.
\end{lemma}
\begin{proof}
Passing to a log resolution, we may assume that $X$ is smooth and $B$ has simple normal crossings support. Then by Theorem \ref{t-bertini+}, for sufficiently general $D\in |V|$, $(X, B+tD)$ is log smooth, and the lemma follows.
\end{proof}
\subsection{Linear series}
Let $\pi:X\to U$ be a projective surjective morphism of normal analytic varieties such that $X$ is smooth and $D$ a $\mbR$-divisor on $X$. 
Note that by definition $\OO _X(D)=\OO _X(\lfloor D\rfloor )$.
If $\pi _* \OO _X(D)\ne 0$, then let $B$ be a prime Weil divisor on $X$ and $m_B(D)$ the largest integer $m$ such that $\pi _* \OO _X(D-mB)\to \pi _* \OO _X(D)$ is an isomorphism (this can be computed on any open subset $V\subset U$ such that $V\cap f(B)\ne \emptyset$ cf. \cite[pg 97]{Nak04}).
We define \[|D/U|=\{D'\sim _U D|D'\geq 0\}\qquad {\rm and}\qquad |D|=\{D'\sim D|D'\geq 0\}.\] Here $D'\sim _U D$  if $D-D'$ is a $\mbZ$-linear combination of principal divisors and Cartier divisors pulled back from $U$. Similarly, we say that $D'\sim _{\mbR,U}D$ if $D-D'$ is an $\mbR$-linear combination of principal divisors and Cartier divisors pulled back from $U$. We let $|D/U|_{\mbR}:=\{D'\geq 0|D'\sim _{\mbR,U}D\}$.

Assume now that $U$ is Stein.
\begin{lemma}\label{l-mD} Let $\pi :X\to U$ be a projective morphism from a normal variety $X$ to a Stein variety $U$, and $D$ an $\mbR$-divisor on $X$. If $\pi _* \OO _X(D)\ne 0$, then $|D|\ne \emptyset $ and $|D/U|\ne \emptyset $. Moreover, for a prime Weil divisor $B$ on $X$ define $m_B|D|:={\rm max}\{t \geq 0\;|\;D'\geq tB\ {\rm for\ all}\ D'\in |D|\}$ and $m_B|D/U|:={\rm max}\{t \geq 0\;|\;D'\geq tB\ {\rm for\ all}\ D'\in |D/U|\}$. Then
\[ m_B(D)=m_B|D|=m_B|D/U|.\]

\end{lemma}

 \begin{proof}
Since $U$ is Stein, $H^0(X, \OO_X(D))\cong H^0(U, \pi _* \OO _X(D))\ne 0$ and so $|D|\ne \emptyset $ and $|D/U|\ne \emptyset $.
 
  Let $m_B=m_B(D)$. Since $\pi _* \OO _X(D-m_BB)\hookrightarrow  \pi _* \OO _X(D)$ is an isomorphism, $H^0(X, \OO _X(D))=H^0(X, \OO _X(D-m_BB))$
 and hence $m_B|D|\geq m_B$.
 Since $\pi _* \OO _X(D-(m_B+1)B)\hookrightarrow  \pi _* \OO _X(D)$ is not surjective, and $U$ is Stein,
 this map is also not surjective on global sections, i.e. $H^0(\OO _X(D-(m_B+1)B))\to H^0(\OO _X(D))$
 is not surjective so that $m_B|D|<m_B+1$, and hence $m_B|D|= m_B$.
 
 Clearly $m_B|D|\geq m_B|D/U|$. If this inequality is strict, then there is a divisor $G\sim _U D$ such that ${\rm mult }_B(G)<m:=m_B|D|$.
 We can then pick an open subset $V\subset U$ such that $V\cap \pi (B)\ne \emptyset$ and $G|_{X_V}\sim D|_{X_V}$. But then
 $G$ is not in the image of $\phi:\pi _* \OO _X(D-mB)|_V\hookrightarrow  \pi _* \OO _X(D)|_V$. On the other hand, we have already seen that $m=m_B|D|=m_B$ and hence $\phi$ is an isomorphism. This is impossible and so $m_B|D|= m_B|D/U|$.
 \end{proof}
We let \[{\rm Fix}|D/U|=\sum m_B|D/U|\cdot B,\qquad  {\rm Mob}|D/U|=D-{\rm Fix}|D/U|.\]
Note that by what we have seen above, we have 
\[{\rm Fix}|D/U|={\rm Fix}|D|, \mbox{ where }{\rm Fix}|D|=\sum m_B|D|\cdot B.\]
\begin{lemma} Let $\pi :X\to U$ be a projective morphism to a Stein variety, $X$ smooth and $D$ a divisor on $X$ such that $\pi _*\OO _X(D)\ne 0$.
       If  $F={\rm Fix}|D/U|$ and $M={\rm Mob}|D/U|$, then ${\rm Fix}|M/U|=0$ and \[|D/U|=F+|M/U|,\qquad |D|=F+|M|.\] 
\end{lemma}
\begin{proof}
Immediate consequence of Lemma \ref{l-mD}.
\end{proof}
It is easy to see that ${\rm Fix}|kmD/U|\leq k{\rm Fix}|mD/U|$ for any integers $k,m>0$. If $D\sim _{\Q,U}D'\geq 0$, then we let \[{\mathbf {Fix}}(D/U):={\rm lim inf}\frac 1 k {\rm Fix}|kD/U|\] for all $k>0$ sufficiently divisible. Clearly  \[{\mathbf {Fix}}(D):={\rm lim inf}\frac 1 k {\rm Fix}|kD|={\mathbf {Fix}}(D/U).\]
If $S\subset X$ is a smooth divisor, then we let $|D/U|_S\subset |D_{|S}/U|$ be the sub-linear series consisting of all divisors $D'|_S$, where $D'\in |D/U|$ and ${\rm Supp} D'$ does not contain $S$.
If $|D/U|_S\ne \emptyset$, we let ${\rm Fix}_S|D/U|:={\rm Fix}(|D/U|_S)$ and if $|kD/U|_S\ne \emptyset$ for some integer $k>0$, then we let \[{\mathbf {Fix}}_S(D/U):={\rm lim inf}\frac 1 k {\rm Fix}|kD/U|_S\] for all $k>0$ sufficiently divisible. 
Similarly to what we have seen above, one can show that $|D/U|_S\ne \emptyset$ if and only if the homomorphism $\pi _* \OO _X(D)\to \pi _* \OO _S(D|_S)$ is non-zero. Since $U$ is Stein, this is in turn equivalent to the fact that 
$|D|_S\ne \emptyset$. It then follows that ${\rm Fix}_S|D/U|={\rm Fix}_S|D|$, and
\[{\mathbf {Fix}}_S(D):={\rm lim inf}\frac 1 k {\rm Fix}|kD|_S={\mathbf {Fix}}_S(D/U).\] 

If $f:X'\to X$ is a proper birational morphisms of smooth varieties, $D$ an $\R$-Cartier divisor on $X$ and $E\geq 0$ an $f$-exceptional divisor, then $|f^*D/X|_{\R}+E=|f^*D+E/X|_{\R}$, and $|f^*D|_\mbR+E=|f^*D+E|_\mbR$.

If $|D/X|_{\R}\ne \emptyset$, 
then define ${\mathbf B}(D/X)=\cap _{D'\in |D/X|_{\R}}{\rm Supp}(D')$.
 
\begin{lemma}\label{l-2.3} Let $\pi :X\to U$ be a projective morphism to a Stein variety, $X$ smooth and $D$ a divisor on $X$.
      If $D$ is a $\Q$-divisor such that $|D/U|_{\R}\ne \emptyset$, then $|D/U|_{\Q}\ne \emptyset$, $|D|_{\Q}\ne \emptyset$ and 
      \[{\mathbf B}(D/U)=\cap _{D'\in |D/U|_{\Q}}{\rm Supp}(D')=\cap _{D'\in |D|_{\Q}}{\rm Supp}(D').\]
\end{lemma}
\begin{proof} It is easy to see that ${\mathbf B}(D/U)\subset \cap _{D'\in |D/U|_{\Q}}{\rm Supp}(D')\subset \cap _{D'\in |D|_{\Q}}{\rm Supp}(D')$.
By \cite[Lemma 2.3]{CL10}, it follows that ${\mathbf B}(D/U)= \cap _{D'\in |D/U|_{\Q}}{\rm Supp}(D')$.  Finally, let $x\in X$ and $\nu :X'\to X$ be the blow up of $x$ and $E$ the corresponding exceptional divisor. Then $x\in \cap _{D'\in |D|_{\Q}}{\rm Supp}(D')$ if and only if $m_E|\nu ^* mD|>0$ for any $m>0$. Assume that $x\in \cap _{D'\in |D|_{\Q}}{\rm Supp}(D')$. Then by Lemma \ref{l-mD}, $m_E|\nu ^* mD/U|>0$ for any $m>0$, and hence $x\in \cap _{D'\in |D/U|_{\Q}}{\rm Supp}(D')$. This shows that $\cap _{D'\in |D/U|_{\Q}}{\rm Supp}(D')\supset \cap _{D'\in |D|_{\Q}}{\rm Supp}(D')$ and the claim follows.
\end{proof}

\begin{lemma}\label{l-lu} Let $\pi :X\to U$ be a projective morphism from a smooth connected complex variety to a Stein space, $\mathcal L=\OO _X(L)$ a line bundle on $X$ and $S$ a smooth divisor on $X$.
       Then the following are equivalent.
       \begin{enumerate}
           \item $\pi _*\mathcal L\to \pi _*(\mathcal L|_S)$ is surjective,
           \item $H^0(X,\mathcal L)\to H^0(S,\mathcal L|_S)$ is surjective or equivalently $|L|_S=|L_{|S}|$,
           \item  $|L/U|_S=|L_{|S}/U|$.
       \end{enumerate}
\end{lemma}
\begin{proof}
(1) implies (2). Since $U$ is Stein, $H^0(U,\pi _*\mathcal L)\to H^0(U,\pi _*(\mathcal L|_S))$ is surjective, and hence so is
 $H^0(X,\mathcal L)\to H^0(S,\mathcal L|_S)$.
 
 (2) implies (1). Since $U$ is Stein, $\pi _*(\mathcal L|_S)$ is globally generated by sections of $H^0(U,\pi _*(\mathcal L|_S))\cong H^0(S,\mathcal L|_S)$. By assumption these sections lift to $H^0(U,\pi _*\mathcal L)\cong H^0(X,\mathcal L)$.
 Thus $\pi _*\mathcal L\to \pi _*(\mathcal L|_S)$ is surjective.

 (3) implies (1). Since $U$ is Stein, $\pi _*(\mathcal L|_S)$ is globally generated by sections of $H^0(U,\pi _*(\mathcal L|_S))\cong H^0(S,\mathcal L|_S)$. Fix $u\in U$ and $g^1_u,\ldots , g^k_u\in H^0(S,\mathcal L|_S)$
 local generators of $\pi _*(\mathcal L|_S)$ at $u$. If $G^1_u,\ldots ,G^k_u\in |L_{|S}|$ are the corresponding  divisors, then by assumption there are divisors $G^i\in |L+\pi ^*C^i|$ such that $G^i|_S=G^i_u$ and $C^i$ is Cartier on $U$.
 Since the $C^i$ are Cartier, there is an open subset $u\in V\subset U$ such that $C^i|_V$ is principal, and hence $\pi ^*C^i|_{X_V}\sim 0$, where $X_V:=\pi^{-1}V$. But then $G^i|_{X_V}\sim L|_{X_V}$ and $(G^i|_{X_V})|_{S_V}=G^i_u|_{S_V}$, where $S_V=S\cap X_V$. This means that \[g^1_u|_{S_V},\ldots , g^k_u|_{S_V}\in {\rm im}\left( H^0(L|_{X_V})\to H^0(L|_{S_V})\right).\]
 Since  $g^1_u,\ldots , g^k_u\in H^0(S,\mathcal L|_S)$ are 
 local generators of $\pi _*(\mathcal L|_S)$ at $u$, then $\pi _*\mathcal L\to \pi _*(\mathcal L|_S)$ is surjective at $u$. Since $u\in U$ is arbitrary, (1) holds.
 
 (1) implies (3). It is clear that $|L/U|_S\subset |L_{|S}/U|$. Suppose that $G_S\in |L_{|S}/U|$, then we must show that $G_S=G|_S$ for some $G\in |L/U|$. 
 By definition, there is a Cartier divisor $C$ on $U$ such that $G_S\sim L_{|S}+(\pi ^* C)_{|S}$. 
 By our assumption, $\pi _*\mathcal L\to \pi _*(\mathcal L|_S)$ is surjective, and hence so is 
 $\pi _*\mathcal L(\pi ^* C)\to \pi _*(\mathcal L(\pi ^* C)|_S)$ (here we use the projection formula and the fact that $\OO _U(C)$ is invertible).
 Since $U$ is Stein, this induces a surjection on  global sections and hence
 $H^0(X,\mathcal L(\pi ^* C))\to H^0(S, \mathcal L(\pi ^* C)|_S)$ is surjective, i.e. $G_S=G|_S$ for some $G\in |\mathcal L(\pi ^* C)|$. Thus $G\sim _U  L$ concluding the proof.
 
\end{proof}
\begin{lemma}\label{l-gu} Let $\pi :X\to U$ be a projective morphism from a smooth complex variety to a Stein space, $\mathcal L=\OO _X(L)$ a line bundle on $X$. For any point $x\in X$, we have that $\Bs(|L|)$ does not contain $x$ if and only if $\Bs(|L/U|)$ does not contain $x$, if and only if $\pi _* \mathcal L\to \pi _*(\mathcal L/\mathfrak m _x)$ is surjective.
   
\end{lemma}
\begin{proof}
Since $|L|\subset |L/U|$, it is clear that if $\Bs(|L|)$ does not contain $x$, then $\Bs(|L/U|)$ does not contain $x$.

Suppose now that $\Bs(|L/U|)$ does not contain $x$. So there is a divisor $0\<G\in |L/U|$ such that $x\not\in {\rm Supp}(G)$.
Since $G\sim L+\pi ^*C$, where $C$ is a Cartier divisor on $U$, we may find an open subset $\pi (x)\in V\subset U$ such that $C|_V$ is a principal divisor, i.e. $C|_V\sim 0$. But then $G|_{X_V}\sim L|_{X_V}$ and it follows that $\mcL|_{X_V}$ is globally generated at $x$.  
$ \mathcal L\to \mathcal L/\mathfrak m _x\cong \mathbb C _x$ is surjective, and hence so is $\pi _* \mathcal L\to \pi _*(\mathcal L/\mathfrak m _x)$, since $U$ is Stein. 

Suppose now that $\pi _* \mathcal L\to \pi _*(\mathcal L/\mathfrak m _x)$ is surjective.
Since $U$ is Stein, \[H^0(X, \mathcal L)\cong H^0(U, \pi _* \mathcal L)\to H^0(U, \pi _*(\mathcal L/\mathfrak m _x))\cong H^0\left(\{x\}, \mathcal L/\mathfrak m _x\right)\cong \mbC_x\] is surjective, and hence $\Bs(|L|)$ does not contain $x$.
\end{proof}

\begin{lemma}\label{l-2.28}
Let $\pi :X\to U$ be a projective morphism from a smooth variety to a Stein space and let $D_1,\ldots , D_\ell \in {\rm Div}_\mbQ(X)$
be such that $|D_i|_\mbQ \ne \emptyset$ for each $i$. Let $V \subset {\rm Div}_\R(X)$ be the subspace spanned by the
components of $D_1,\ldots , D_\ell$, and let $\mathcal P \subset V$ be the convex hull of $D_1,\ldots , D_\ell$. Assume
that the ring \[R(X; D_1,\ldots , D_\ell):=\bigoplus_{(m_1,\ldots, m_\ell)\in \mathbb N ^k}H^0\left(X, \OO _X\left(\sum m_iD_i\right)\right)\] is finitely generated. Then:
\begin{enumerate}
    \item ${\rm Fix}$ extends to a rational piecewise affine function on $\mathcal P$;
    \item there exists a positive integer $k$ such that for every $D \in \mathcal P$ and every $m \in \mathbb N$,
if $\frac m
k D \in {\rm Div}(X)$, then $\mathbf{Fix}(D) = \frac 1
m {\rm Fix} |mD|$.
\end{enumerate} 
\end{lemma}
\begin{proof}
See the proof of \cite[Lemma 2.28]{CL10}.
\end{proof}

\subsection{K\"ahler classes}\label{subs:analytic-classes}

In this section we recall a well known to the experts, characterizations of  K\"ahler classes.
Since we were unable to find complete references in the literature, we include a detailed proof  below.

We consider the following set-up. Let $X$ be a compact, normal complex space. The objects we will work with in this subsection are introduced below.
\smallskip

\begin{definition} Let
$\displaystyle (A_i)_{i\in I}$ be an open finite covering of $X$ such that each subset $A_i$ is a local analytic subset of some open subset $\Omega_i\subset \CC^{N_i}$. The space of forms of type $(p,q)$, denoted by $\cC^{k}_{p,q}(X)$, is defined by local restrictions of forms of type $(p,q)$ which are $k$ times differentiable on the sets $\Omega_i$ above. Here $k$ is a positive integer or $\infty$. The definition of the space of currents on $X$ is then
completely parallel to the smooth case.
\end{definition}

\noindent A Hermitian metric $g$ on $X$ corresponds to a smooth, positive definite $(1,1)$-form. By this we mean that locally $g|_{A_i}$ is given by the restriction of a 
Hermitian form on $\Omega_i$. For a more complete presentation we refer the reader to the first part of the article \cite{Dem85}.
\medskip

\noindent Let $\alpha\in \cC^{\infty}_{1,1}(X)$ be a smooth $(1,1)$-form on $X$.
We assume that $\alpha$ is $\partial$ and $\dbr$ closed, such that moreover the following properties hold true.

\begin{enumerate}
\smallskip

\item[(1)] The class $\{\alpha\}$ is nef, i.e, we have $(f_\ep)_{\ep> 0}\subset \cC^{\infty}(X)$ such that
$$\alpha+ i\ddbar f_\ep\geq -\ep g$$
on $X$.
\smallskip

\item[(2)] The class $\{\alpha\}$ is big, i.e. there exists
a function $\tau\in L^1(X)$ such that 
$$\alpha+ i\ddbar \tau\geq \ep_0 g$$
 as currents on $X$, where $\ep_0> 0$ is a positive constant.
\smallskip

\item[(3)] Let $V\subset X$ be a positive dimensional (compact) reduced analytic subset. Then we have 
$$\int_{V_{\rm reg}}\alpha^{\dim(V)}> 0.$$
\end{enumerate}
\smallskip

\noindent Then we show that the following holds true.

\begin{theorem}\label{thm:nef-big-to-kahler}
Let $X$ be a compact analytic normal variety and $\alpha\in  \cC^{\infty}_{1,1}(X)$
such that $\partial\alpha= 0, \dbr\alpha= 0$. We assume moreover that the
properties (1)-(3) above are satisfied. 
Then $\alpha$ is a K\"ahler class, i.e. there exists a function $\varphi\in \cC^\infty(X)$ such that
\begin{equation}\label{eq1}
\alpha+ i\ddbar \varphi\geq \ep_1 g
\end{equation}
on $X$, where $\ep_1>0$.
\end{theorem}

\noindent In particular, $X$ is a K\"ahler space in the sense adopted in
\cite{BG13} provided that $\alpha$ is locally in the image of the $\ddbar$ operator. But in any case we can construct the function $\varphi$ with the properties of \eqref{eq1}.
\smallskip

\begin{remark} In case $X$ admits an embedding into a K\"ahler manifold (i.e. non-singular, but not assumed to be compact or even complete) Theorem \eqref{thm:nef-big-to-kahler}
is a direct consequence of the main result of the article \cite{CT16} by Collins-Tosatti. This hypothesis is not necessarily satisfied in our case, but nevertheless many of the techniques of 
\cite{CT16} will be used in the next proof.
\end{remark}

\subsubsection{Psh functions on complex spaces} We recall here a few basic facts concerning psh functions defined on normal complex spaces. Our main reference is the first section of the article \cite{Dem85}.
\smallskip

\noindent To start with, a \emph{quasi-psh} function $\phi:X\to [-\infty, \infty[$  
is {by definition} given by the restriction to each $A_i$  
of a quasi-psh function on $\Omega_i$, for all $i\in I$.  A locally integrable function 
$\psi:X\to [-\infty, \infty[$ is called \emph{weakly quasi-psh} if it is locally bounded from above and such that 
\begin{equation}\label{eq8}
i\ddbar \psi\geq - Cg
\end{equation}
for some positive real constant $C> 0$. Note that the local boundedness hypothesis is automatic in the non-singular case, but this is no longer true in our actual context.
\smallskip

\noindent We quote next a result which plays a crucial role in what follows. Its proof 
(cf. \cite{Dem85}, Theorem 1.7) relies on two fundamental facts: the desingularization theorem of Hironaka and the characterisation of psh functions by restrictions to holomorphic disks, due to Forna\!ess-Narasimhan. 

\begin{theorem}\label{wpsh}\cite[Theorem 1.7]{Dem85} Let $\psi$ be a weakly quasi-psh function defined on a normal compact complex space $X$. Then the function 
\begin{equation}\label{eq9}
\psi^\star(x):= \lim\sup_{y\to x}\psi(y)
\end{equation}
is quasi-psh on $X$. Moreover, if $\displaystyle i\ddbar \psi\geq - C\gamma$ for some smooth form 
$\gamma$ on $X$ then the Hessian of the function $\psi^\star$ in \eqref{eq9} has the same property. 
\end{theorem}

\noindent  The following statement is a direct consequence of Theorem \ref{wpsh}.
\begin{corollary}\label{coro1} Let $p: Y\to X$ be a modification, where $X$ and $Y$ are normal complex spaces. We assume that $\displaystyle p^\star\alpha+ i\ddbar \psi_Y\geq Cp^\star g$, where 
$C$ is a real number. Then there exist a quasi-psh function $\psi_X:X\to [-\infty, \infty[$  such that 
\begin{equation}\label{eq10}
\alpha+ i\ddbar \psi_X\geq Cg
\end{equation} 
in the sense of currents on $X$. 
\end{corollary}

\noindent Indeed, $\psi_X$ is obtained by taking the direct image of $\psi_Y$ and then applying the \emph{usc} regularisation procedure \eqref{eq9}. We notice that the direct image of $\psi_Y$ is automatically locally bounded from above. 
\smallskip

\begin{proof}[Proof of Theorem \ref{thm:nef-big-to-kahler}] 
The first step consists in constructing a (new) function $\tau$ with similar properties as in (2) above such that its singularities are concentrated along an analytic subset of $X$. 

\noindent Let $\pi:\wh X\to X$ be a desingularization of $X$. The pull-back of (2) shows that we have 
\begin{equation}\label{eq2}
\Theta:= \pi^\star \alpha+ i\ddbar (\tau\circ \pi)\geq \ep_0 \pi^\star g.
\end{equation}
In other words $\Theta$ is a closed $(1,1)$-current on $\wh X$, greater than $\ep_0\pi^\star g$. This implies that $\{\Theta\}$ contains a so-called K\"ahler current, that is to say a representative which is greater than a positive multiple of a Hermitian metric on $\wh X$.

 By Demailly's regularisation theorem (cf. \cite{Dem92}, main result), we can replace $\Theta$ by a cohomologous current 
say $\Theta_1\in \{\Theta\}$ such that 
\begin{equation}\label{eq3}
\Theta_1:= \pi^\star \alpha+ i\ddbar \varphi_1\geq \ep_1 \wh g, \qquad \Theta_1|_{\wh X\setminus W}
\in \cC^\infty(\wh X\setminus W)\end{equation} 
on $\wh X$, where $W\subset \wh X$ is a proper analytic subset and $\wh g$ is a Hermitian metric on $\wh X$.
\smallskip

 The direct image $\pi_\star \Theta_1$ has the property (2) and it is non-singular on the complement of the analytic set $Y\subset X$
\begin{equation}\label{eq4}
Y:= X_{\rm sing}\cup \pi(W).
\end{equation}
In order to keep the notations as simple as possible, we assume from now on that $\tau$ in (2) is smooth in the complement of an analytic set $Y$.
\medskip

The next step consists in establishing the following simple statement, which will be used to argue by induction.

\begin{lemma}\label{ind}
Let $Z\subset X$ be any normal analytic subspace. Then the restriction $\alpha|_Z$ 
defines a $(1,1)$-form on $Z$ which satisfies properties {\rm (1)-(3)}.
\end{lemma}
\begin{proof}[Proof of Lemma \ref{ind}]
It is clear that $\alpha|_Z$ satisfies the properties (1) and (3). We show next that it is the case for (2) as well. By the existence of the current $\Theta_1$ as in \eqref{eq3} it follows that a further modification of the complex manifold $\wh X$ is K\"ahler, see \cite[Theorem 0.7]{DP04}. We assume that it is the case for $\wh X$ itself. In particular, $X$ is in Fujiki's class $\mcC$.
 \smallskip

Let $p_Z:\wh Z\to Z$ be a desingularization of $Z$. As we have seen that $X$ is in Fujiki's class $\mcC$, by \cite[A, page 235]{Fuj83C}, $Z$ is also in the class $\mcC$. Therefore passing to a higher desingularization we may assume that $\widehat{Z}$ is K\"ahler.

Then the class $p_Z^\star\{\alpha\}$ is nef, and $\displaystyle \int_{\wh Z}p_Z^\star\alpha^d> 0$. By \cite[Theorem 0.5]{DP04} it contains a K\"ahler current, whose direct image combined with Corollary \ref{coro1} allow us to conclude. 
\end{proof}
\medskip

 In this last step we 
remove the singularities of $\pi_\star\Theta_1$ by induction. For this, we are using the gluing techniques as in \cite{Dem90} (the reader may also consult \emph{Complex Analytic and Differential Geometry} by J.-P. Demailly, book available on the author's website, pages 411-414). We have to face two types of difficulties:

 $\bullet$ The space $Y$ may have several components.

 $\bullet$ Even if $Y$ is irreducible, it may not be normal (which will give us troubles, since we intend to use induction).

\smallskip

 In order to understand how it works, we first assume that $Y$ is irreducible and normal.
Lemma \ref{ind} plus the induction hypothesis show the existence of a function $\tau_Y\in \cC^\infty(Y)$ such that 
\begin{equation}\label{eq5}
\alpha|_Y+ i\ddbar \tau_Y\geq \ep_2g|_Y.
\end{equation} 

 By the proof of Theorem 4 in \cite{Dem90} we can assume that there exists an open subset $Y\subset U\subset X$ such that \eqref{eq5} holds true on $U$. That is to say, there exists an extension $\wt \tau_Y\in \cC^\infty(U)$
of $\tau_Y$ such that 
\begin{equation}\label{eq11}
\alpha|_U+ i\ddbar \wt \tau_Y\geq \ep_3g|_U
\end{equation} 
for some strictly positive $\ep_3> 0$. Intuitively the construction of $\wt \tau_Y$ is clear: thanks to \eqref{eq5} the eigenvalues of the Hessian of $\tau_Y$ in the tangent directions of $Y$ are suitable, we simply "correct" the normal directions as indicated in \emph{loc. cit.} In this process there is a loss of positivity involved (since one is using a partition of unity) but since $\ep_2>0$, we can afford that.
\smallskip

 Now we consider the regularized maximum function  
\begin{equation}\label{eq6}
\varphi:= \max_{\rm reg}(\tau, \wt \tau_Y- C)
\end{equation} 
(cf. \cite{Dem90}, part of the proof of Lemma 5)
where $C\gg 0$ is a large enough constant, such that $\varphi= \tau$ near the boundary of $U$. This is possible since $\tau$ is smooth on the complement of $Y$. On the other hand, we clearly have $\varphi= \wt\tau_Y- C$ in a neighborhood of $Y$, since 
$\tau$ equals $-\infty$ when restricted to $Y$. Now the usual properties of the regularised 
maximum of two functions (see especially \emph{loc. cit.}, page 287) show that we have \eqref{eq11}.
\medskip

In order to treat the general case, we formulate the following statement.

\begin{claim}\label{senilita}
Let $Y\subsetneq X$ be an analytic subset of $X$. Then there exists an open subset $U$ such that $Y\subset U\subset X$ and a function $\wt \tau_Y$ for which the property \eqref{eq11} is valid.
\end{claim}

Before explaining the arguments of the claim, a first thing to remark is that it would settle Theorem \ref{thm:nef-big-to-kahler}, by the maximum technique used in the particular case we have just treated above. We proceed in two steps.
\smallskip

 $\bullet$ \emph{It is enough to establish the Claim \ref{senilita} in case of an analytic space $Y$ which is irreducible.} This is done by decomposing
\begin{equation}\label{eq7}
Y= Y_1\cup\dots \cup Y_N
\end{equation} 
the set $Y$ as union of irreducible analytic sets and applying the maximum procedure sketched above combined with induction on $N$. Although standard, we explain next the construction of 
$(U, \wt \tau_Y)$ if $N=2$, i.e. we assume that $Y$ only has two components. For an arbitrary $N$ there are no additional arguments to be invoked. 
\smallskip

 Let $(U_1, \wt \tau_1)$ and $(U_2, \wt \tau_2)$ corresponding to $Y_1$ and $Y_2$, respectively such that \eqref{eq11} holds true.
By \cite[Lemma 5]{Dem90} there exists a quasi-psh function $v$ with log-poles along $Y_1\cap Y_2$
and smooth in the complement of this analytic set. We consider the function 
\begin{equation}\label{eq12}
\wt \tau_1+ \ep v 
\end{equation} 
where $0< \ep\ll 1$ is small enough -fixed- such that 
\begin{equation}\label{eq13}
\alpha|_{U_1}+ i\ddbar \left(\wt \tau_1+ \ep v \right)\geq \frac{1}{2}\ep_3g|_{U_1}. 
\end{equation}
This operation may seem silly --since $\wt \tau_1$ is smooth and by 
adding the small multiple of 
$v$ the resulting function becomes singular along the intersection $Y_1\cap Y_2$. Nevertheless, thanks to it we can conclude. Let $W\subset U_1\cap U_2$ be an open subset of $X$ containing $Y_1\cap Y_2$, and $C\gg 0$ large enough such that we have 
\begin{equation}\label{eq14}
\wt \tau_1+ \ep v\geq \wt \tau_2- C
\end{equation} 
on $\partial W$, the boundary of $W$. We fix such a constant $C$ and remark that the function
\begin{equation}\label{eq15}
\max_{\rm reg}(\wt \tau_1+ \ep v, \wt \tau_2- C)
\end{equation} 
defined on $U_1$ is smooth, its Hessian verifies an inequality similar to \eqref{eq13} and moreover it equals 
$\wt \tau_2- C$ near $Y_1\cap Y_2$. By shrinking $U_1$ and $U_2$ we can combine (smoothly!) the function constructed in \eqref{eq15} with $\wt \tau_2- C$ and therefore obtain $(U, \wt \tau_Y)$.
\smallskip

 $\bullet$ \emph{Induction.} We assume that Theorem \ref{thm:nef-big-to-kahler} is established in case of a normal analytic space of dimension smaller than $\dim(X)$ and that Claim \ref{senilita} is established for analytic sets $Z$ such that $\dim(Z)\leq \dim(Y)-1$.
\smallskip

 Let $Y\subsetneq X$ be an irreducible, proper analytic subset of $X$. Then there exists a modification $f: X_1\to X$ with the following properties.
\begin{enumerate}
\smallskip

\item[(i)] The analytic space $X_1$ is compact and normal.
\smallskip

\item[(ii)] $f$ is an isomorphism over a neighborhood of the general points of $Y$.
\smallskip

\item[(iii)] Let $Y_1\subset X_1$ be the strict transform, then $Y_1\to Y$ is a resolution. 
\end{enumerate}
 To construct $f$ we proceed as follows. Let $Y'\to Y$ be a resolution of singularities given by the following finite sequence of successive blow ups 
\begin{equation}
    \xymatrixcolsep{3pc}\xymatrix{
    Y':=Y^n\ar[r]^{g_n} & Y^{n-1}\ar[r] &\cdots \ar[r] & Y^1\ar[r]^{g_1} & Y^0=:Y
    },
\end{equation}
so that $g_i:Y^i\to Y^{i-1}$ is the blow up of an ideal $\msI^{i-1}\subset \OO _{Y^{i-1}}$. 
We will now define inductively a sequence of closed immersions $Y^i\hookrightarrow X^i$ with $Y^0=Y$ and $X^0=X$.
 proceeding by induction, suppose that $Y^{i-1}\hookrightarrow X^{i-1}$ has been defined.
 Let $\msJ^{i-1}\subset \mcO_{X^{i-1}}$ be the full inverse image of the ideal sheaf $\msI^{i-1}$ under the surjection $\mcO_{X^{i-1}}\surjective \mcO_{Y^{i-1}}$. Let $f_i: X^i\to X^{i-1}$ be the blow up of $X^{i-1}$ along the ideal sheaf $\msJ^{i-1}$. Then we have the following commutative diagram of blow ups and closed immersions (also see \cite[Corollary 7.15, page 165]{Har77})
\begin{equation}
    \xymatrixcolsep{3pc}\xymatrixrowsep{3pc}\xymatrix{
Y^i\ar@{^{(}->}[r]\ar[d]_{g_i} & X^i\ar[d]^{f_i}\\
 Y^{i-1}\ar@{^{(}->}[r] & X^{i-1}
    }
\end{equation}
We let $X'=X^n$ and $X_1\to X'$ be the normalization, then the induced morphism $f:X_1\to X$ satisfies (i) and (ii) and if $Y_1$ is the strict transform of $Y'$ under $f$, then $Y_1\to Y'$ is a proper finite bimeromorphic morphism. Since $Y'$ is smooth (and in particular normal), $Y_1\to Y'$ is an isomorphism (due to Zariski's main theorem), and hence (iii) also holds.

  The restriction $\displaystyle f^\star\alpha|_{Y_1}$ of the $f$-inverse image of 
$\alpha$ to $Y_1$ is still nef and big. Indeed, 
the nefness is clear, since we can pull-back the functions $(f_\ep)$ given in (1) and restrict them to $Y_1$. Moreover, we have 
\[\int_{Y_1}f^\star\alpha^d> 0\]
as consequence of the condition (3) and thus $f^\star\alpha|_{Y_1}$ is big.

Thus it contains a K\"ahler current 
\begin{equation}\label{eq22}
 f^\star\alpha|_{Y_1} + i\ddbar \psi_1\geq g_1|_{Y_1},
 \end{equation} where the function $\psi_1$ can be assumed to have analytic singularities and $g_1$ is a Hermitian metric on $X_1$. In particular $\psi_1$ is smooth in the complement of a proper analytic subset $W_1\subsetneq Y_1$. We can also assume that 
$W_1$ contains the analytic set in the complement of which the restriction of the map
$\displaystyle f|_{Y_1}: Y_1\to Y$ is a biholomorphism.
\smallskip
 
 \noindent By modifying the function $\psi_1$ as in \cite{Dem90}, we infer the following: \emph{there exists an open set $\Lambda$ containing $Y_1\setminus W_1$ such that}
\begin{equation}\label{eq24}
 f^\star\alpha + i\ddbar\psi_1\geq \frac{1}{2}g_1
 \end{equation}  
 \emph{on} $\Lambda$. A more general version of this is established in the article \cite{CT15} pages 1181-1185. 
 \smallskip
 
 \noindent Next, we consider the direct image 
$W:= f(W_1)\subsetneq Y$ and we use the induction hypothesis: there exists an open subset $U_W\subset X$ and a smooth function $\tau_W: U_W\to \mathbb R$ such that
\eqref{eq11} holds true. By taking the inverse image via $f$ we get
\begin{equation}\label{eq23}
 f^\star\alpha + i\ddbar (\tau_W\circ f)\geq \ep_4 f^\star g
 \end{equation} 
pointwise on $f^{-1}(U_W)$. 
\smallskip

\noindent By combining \eqref{eq23} and \eqref{eq24} we obtain a $\cC^\infty$ function $\tau_1$ on an open subset $f^{-1}(U_Y)$, the inverse image of an open subset $U_Y$ containing $Y$. In other words, we have 
\begin{equation}\label{eq25}
 f^\star\alpha + i\ddbar\tau_1\geq \ep_5 f^\star g
 \end{equation}  
 on $f^{-1}(U_Y)$. The ``pinched'' open subset $\Lambda$ is involved in the gluing process, but this makes absolutely no difference as the arguments in 
 \cite{CT15} show (the point is that the function $\tau_W$ used in the gluing procedure is defined on an open subset containing $W_1$). 
\smallskip

\noindent The inequality \eqref{eq25} shows that the smooth function $\tau_1$ is constant on every positive dimensional fiber of $f$ at some point of $U_Y$. It therefore descends to $U_Y$ and Claim \ref{senilita} is established.
\medskip
 
\noindent As we have already mentioned, we can glue the function constructed in our claim 
with $\tau$ obtained in the first part of the proof: in this way we obtain $\varphi$.
\end{proof}
\medskip

\noindent We also include the following result whose proof follows exactly as in the non-singular case, modulo the use of Theorem \ref{wpsh}.

\begin{lemma}\label{lem:arbitrary-nef-and-big}
 Let $X$ be a normal, compact K\"ahler analytic variety and $\alpha\in H^{1, 1}_{\ddbar}(X)$ is a nef class. Then $\alpha$ is a big class on $X$ if and only if $\alpha^n>0$, where $n=\dim X$.
\end{lemma}

\begin{proof}[Proof of Lemma \ref{lem:arbitrary-nef-and-big}]
 Let $f:Y\to X$ be a resolution of singularities of $X$. Given that $X$ is K\"ahler, it will equally be the case for $Y$. Then $f^*\alpha$ is nef and $(f^*\alpha)^n=\alpha^n$, by the projection formula. By a quick direct image argument it follows that $\alpha$ is big if and only if $f^*\alpha$ is big. But the equivalence between the positivity of the top self-intersection and bigness for a nef class on a K\"ahler manifold is well-known \cite{DP04}. 
\end{proof}

\begin{theorem}\label{thm:nef-restricts-to-pseff}
Let $X$ be a compact K\"ahler analytic variety and consider $\alpha\in \cC^{\infty}_{1,1}(X)$ a 
smooth $(1,1)$-form such that $\bar{\partial}\alpha=0,  \partial\alpha= 0$. Then $\alpha$ is a nef class if and only if $\alpha|_Z$ is a pseudo-effective class for all irreducible analytic subvarieties $Z\subset X$. 
\end{theorem}

\begin{proof}[Proof of Theorem \ref{thm:nef-restricts-to-pseff}] If $\alpha$ is nef, then the restriction $\alpha|_Z$ to any irreducible analytic subvariety $Z\subset X$ is also nef. 
If moreover $Z$ is normal, then it follows that $\alpha|_Z$ is pseudo-effective by the usual argument: we construct a closed positive current in the pullback of $\alpha|_Z$ on any non-singular model of $Z$ and then take the direct image. In general, the argument is as follows. Consider 
\[p: Z_\nu\to Z\]
a normalisation of $Z$. This is a proper map with finite fibers and then there exists a quasi-psh function $\rho$ defined on $Z_\nu$ such that 
\[T:= p^\star(\alpha|_Z)+ i\ddbar\rho\geq 0\]
on $Z_\nu$, since $Z_\nu$ is normal. By \cite{Dem85}, Prop. 1.13 we have
\[p_\star i\ddbar\rho= i\ddbar p_\star \rho\]
and moreover the direct image preserves the positivity of $T$. 

\noindent The map $p$ is proper and finite, thus we have 
$p_\star \big(p^\star(\alpha|_Z)\big)= d \alpha|_Z$ where $d$ is the degree of $p$ and therefore the class 
given by $\alpha|_Z$ is pseudo-effective. 

\noindent For the other direction, we proceed as follows: Let $t={\rm inf}\{s\geq 0|\alpha+ s \omega\ {\rm is\  Kahler}\}$, then $\alpha+ t \omega$ is nef but not K\"ahler. Suppose that $t>0$, then $(\alpha+ t \omega)|_Z$ is big (and nef) for every $Z\subset X$ (including $Z=X$), and hence by Lemma \ref{lem:arbitrary-nef-and-big}, $(\alpha+t\omega)^{\dim Z}\cdot Z=((\alpha+ t \omega)|_Z)^{\dim Z}>0$. Then by Theorem \ref{thm:nef-big-to-kahler}, $\alpha+ t \omega$ is K\"ahler, which is easily seen to contradict the definition of $t$. Therefore $t=0$ and so  $\alpha$ is nef.
\end{proof}

\begin{remark}
Theorem \ref{thm:nef-restricts-to-pseff} holds without the assumption that $X$ is K\"ahler,
by using the gluing procedure employed in the proof of Theorem \ref{thm:nef-big-to-kahler}.
\end{remark}
\begin{lemma}\label{lem:nef-pullback}
Let $f :X'\to X$ be a proper surjective morphism of normal compact K\"ahler varieties. Then a class $\alpha \in
H^{1,1}_{\rm BC}(X )$ nef if and only if $f ^*\alpha$ is nef.
\end{lemma}
\begin{proof}
If $\alpha $ is nef then it follows easily that $f ^*\alpha$ is nef. Suppose now that $f ^*\alpha$ is nef and in particular $f^*\alpha$ is pseudo-effective. 
Let $t={\rm inf}\{s\geq 0\;|\;\alpha +s\omega \ \mbox{ is K\"ahler}\}$. Suppose that $t>0$, then $\alpha +t\omega$ is nef but not K\"ahler, and we claim that
$\int _V (\alpha +t\omega) ^{k}>0$ for any subvariety $V$ of codimension $k$.
It follows that conditions (1) and (3) of Theorem \ref{thm:nef-big-to-kahler} are satisfied by $\alpha +t\omega$, and condition (2) is immediate from Lemma \ref{lem:arbitrary-nef-and-big}.
Thus, by Theorem \ref{thm:nef-big-to-kahler}, $\alpha +t\omega$ is K\"ahler, a contradiction. In particular $t=0$ and $\alpha$ is nef.

We now prove the claim. 
For any analytic subvariety $V\subset X$, let $V'$ be the unique irreducible component of $f^{-1}V$ dominating $V$, $F$ a general fiber of $V'\to V$. Assume that $\dim V=k$, $\dim V'=k+j$, $\eta$ K\"ahler on $X'$, and $\lambda=\int _F\eta^j>0$ then, by the projection formula we have
\[\lambda\cdot\int _V (\alpha +t\omega) ^{k}=\int _{V'}f^*(\alpha +t\omega) ^{k}\wedge \eta ^j\geq \int _{V'}f^*(t\omega)^k\wedge \eta ^j=\lambda t^k\int _V \eta ^k >0\]
as $f^*\alpha$ is nef.
\end{proof}
Finally we extend \cite[Corollary 0.3]{DP04} to the singular case.
\begin{corollary}
Let $X$ be a compact normal K\"ahler variety, $\omega $ a K\"ahler form on $X$, and $\alpha \in H^{1,1}_{\rm BC}(X)$, then
\begin{enumerate}
    \item $\alpha $ is nef if and only if $\int _V\alpha ^k\wedge \omega ^{p-k} \geq 0$ for every analytic $p$-dimensional subvariety $V\subset X$ and for all $0<k\leq p$, and
    \item $\alpha $ is K\"ahler if and only if $\int _V\alpha^k \wedge \omega ^{p-k} >0$ for every analytic $p$-dimensional subvariety $V\subset X$ and for all $0<k\leq p$.
\end{enumerate}
\end{corollary}
\begin{proof} The only if part is clear, so assume that $\int _V\alpha ^k \wedge \omega ^{p-k} \geq 0$ for any analytic subvariety $V\subset X$ with $p=\dim V$ and any $0<k\leq p$. Let $V\subset X$ be a proper subvariety,  and $\nu:\tilde V\to V$  the normalization. Suppose that $\tilde \alpha =\nu ^*\alpha $ is the pull-back of $\alpha$, then it follows easily by induction on the dimension that $\tilde\alpha $ is nef. Let $f:X'\to X$ be a resolution of singularities and $V'\subset X'$ a subvariety such that $f(V')=V$. If $\nu ':\tilde V'\to V'$ is the normalization and $\alpha '=f^*\alpha$, then $\tilde \alpha'={\nu'}^* \alpha ' $ is nef as it is the pull-back of $\tilde \alpha$ via the induced map $\tilde V'\to \tilde V$. 

It suffices to show that $\alpha _\epsilon:=\alpha +\epsilon \omega $ is nef for any $0<\epsilon \ll 1$. Clearly \[\int _{X'} f^*(\alpha _\epsilon^k  \wedge \omega ^{n-k})=\int _X\alpha _\epsilon ^k  \wedge \omega ^{n-k}\geq \epsilon^k \int _X \omega ^{n}>0.\]
Let $\omega '$ be a K\"ahler class on $X'$ and $\alpha _\epsilon ' :=f^*\alpha _\epsilon $, then $\omega _\delta :=f^*\omega +\delta \omega '$ is K\"ahler for $\delta>0$ and by continuity $\int _{X'} (\alpha _\epsilon') ^k \wedge \omega_\delta  ^{n-k}>0$ for $0<\delta\ll \epsilon $.
Assume now that $V'\subset X' $ is a proper subvariety of dimension $p<\dim X'$, then since $\tilde \alpha'$ is nef (as observed above), we have
\[\int _{V'}(\alpha '_\epsilon) ^k  \wedge \omega_\delta  ^{p-k}=\int _{\tilde V'}(\tilde \alpha '+\epsilon (f\circ {\nu} ')^*\omega )^k \wedge ( \nu ')^*\omega_\delta  ^{p-k}\geq 0.\]
By \cite[Corollary 0.3]{DP04}, $\alpha _\epsilon '$ is nef, and hence by Lemma \ref{lem:nef-pullback}, $\alpha _\epsilon $ is nef. This proves (1). To see (2), note that if $\alpha$ is K\"ahler then the stated inequalities clearly hold. For the reverse implication, simply observe that the K\"ahler cone coincides with the interior of the nef cone. 
\end{proof}

\subsection{Kawamata-Viehweg vanishing}
The fundamental result in this context is Kawamata-Viehweg vanishing cf. \cite[Theorem 3.7]{Nak87} and \cite[Corollary 1.4]{Fuj13}.

\begin{definition}\label{def:f-nef-big}
	Let $f:X\to Y$ be a proper surjective morphism of analytic varieties and $L$ is a line bundle on $X$. Then $L$ is called $f$-nef-big, if $c_1(L)\cdot C\>0$ for all curves $C\subset X$ such that $f(C)=\pt$, and $\kappa(X/Y, L)=\dim X-\dim Y$ (see \cite[(B), Page 554]{Nak87}).
\end{definition}

The following version of (relative) Kawamata-Viehweg vanishing theorem for a proper morphism between analytic varieties follows from \cite[Theorem 3.7]{Nak87} and \cite[Corollary 1.4]{Fuj13}.

\begin{theorem}\cite[Theorem 2.21]{DH20}\label{thm:relative-kvv}
Let $\pi:X\to S$ be a proper surjective morphism of analytic varieties. Let $\Delta\>0$ be a $\mbQ$-divisor on $X$ such that $(X, \Delta)$ is klt, and $D$ is a $\mbQ$-Cartier integral Weil divisor on $X$ such that $D-(K_X+\Delta)$ is $\pi$-nef-big. Then 
        \[
                R^i\pi_*\mcO_X(D)=0\qquad \text{ for all }\ i>0.
        \]
\end{theorem}

\subsection{Relative cone and contraction theorems for projective morphisms}
Here we collect a cone theorem from \cite{Nak87}. Recall that if $f:X\to Y$ is a projective surjective morphism of analytic varieties and $W\subset Y$ is a compact subset of $Y$, then $Z_1(X/Y;W)$ is generated by curves $C\subset X$ such that $f(C)$ is a point in $W$. We say that two curves $C_1,C_2$ are numerically equivalent over $W$, $C_1\equiv _WC_2$ if $f^*L\cdot (C_1-C_2)=0$ for any Cartier divisor $L$ defined on a neighborhood of $W$. Then $N_1(X/Y;W):=Z_1(X/Y;W)\otimes _{\mathbb Z}\mathbb R/\equiv _W$.\\
We also define $\widetilde{Z}^1(X/Y; W)$ as the group of line bundles defined over some neighborhood of $W$ modulo the following equivalence relation: for $\msL_1\in \Pic(f^{-1}U_1)$ and $\msL_2\in \Pic(f^{-1}U_2)$, where $U_1$ and $U_2$ are open neighborhoods of $W$, $\msL_1\num_W \msL_2$ if $\msL_1\cdot C=\msL_2\cdot C$ for all curves $C\subset X$ such that $f(C)=\pt\in W$. Then $N^1(X/Y; W):=\widetilde{Z}^1(X/Y; W)\otimes_{\mbZ}\mbR$.

\begin{definition}[Property \textbf{P} and \textbf{Q}]\label{def:property-pq}\cite{Fuj22}
Let $f:X\to Y$ be a projective surjective morphism of analytic varieties and $W\subset Y$ is a compact subset of $Y$. We say that $f:X\to Y$ and $W$ satisfy property \bfP\ if the following holds:
\begin{enumerate}
    \item[(P1)] $X$ is a normal analytic variety,
    \item[(P2)] $Y$ is a Stein space,
    \item[(P3)] $W$ is a Stein compact subset of $Y$, i.e. $W$ has a fundamental system of Stein open neighborhoods, and
    \item[(P4)] $W\cap Z$ has finitely many connected components for any analytic set $Z$ defined on an open neighborhood of $W$.
\end{enumerate}
We will simply say that $f:X\to Y$ satisfies property \bfP\ if $W$ is understood.\\

We say that $f:X\to Y$ satisfies property \bfQ\ if 
\begin{enumerate}
    \item[(Q1)] $X$ is normal, and
    \item[(Q2)] $X$ and $Y$ are both compact.
    \end{enumerate}~\\
\begin{itemize}
    \item We will say that a projective morphism $f:X\to Y$ and a compact subset $W\subset Y$ satisfies either property \bfP\ or \bfQ\ if either $f:X\to Y$ and $W\subset Y$ satisfy property \bfP\ or $f:X\to Y$ satisfies property \bfQ. Moreover, if only the property \bfQ\ is satisfied, then we will denote $N^1(X/Y), N_1(X/Y)$, etc. to mean $N^1(X/Y;Y), N_1(X/Y;Y)$, etc.
\end{itemize}    
\end{definition}

\begin{remark}\label{rmk:relative-NS}
    By \cite[Chapter II. 5.19. Lemma]{Nak04}, if $f:X\to Y$ and $W\subset Y$ satisfy either property \bfP\ or \bfQ\, then $N^1(X/Y; W)$ (and hence also $N_1(X/Y;W)$) is finitely dimensional over $\mbR$. {Unluckily, this result is only stated in the case that $X\to Y$ is a} {\bf projective morphism}. By a result of Siu, \cite[Theorem 1]{Siu69}, it is known that property (P4) holds if and only if $\Gamma (W,\mathcal O _W)$ is notherian. In particular, for any $w\in W$ there is a neighborhood $w\in V\subset Y$ such that $V$ satisfies (P3) and (P4).
\end{remark}

\begin{theorem}[Cone Theorem]\cite[Theorem 4.12]{Nak87}\label{thm:general-relative-cone}
Let $f: X \to Y$  be a projective surjective morphsim of analytic varieties and $W\subset Y$ is a compact subset satisfying either property \bfP\ or \bfQ. 
Let $B\>0$ a $\mbQ$-divisor  on  $X$ such that $(X, B)$ is klt. Then the following hold: 
\begin{enumerate}
    \item If $K_X+ B$ is  not $f$-nef over $W$,  then 
\[\overline{NE}(X/Y; W)=\overline{NE}_{K_X+B\geq 0}(X/Y; W)+ \sum \mathbb R^{\geq 0}[l_i]\] where  each 
$l_i$ is an  irreducible  curve in $N_1(X/Y; 
W)$. Furthermore,  $\sum \mathbb R^{\geq 0}[l_i]$  is  locally  finite  and for  any  $R=\mathbb R^{\geq 0}[l_i]$, there 
exists  $L\in   \widetilde{Z}^1(X/Y; W)$  such that  
$R=\{\Gamma \in   \overline{NE}(X/Y; W)\setminus \{0\}|(L\cdot \Gamma )=0\}$ and that $L$ is  $f$-nef over  $W$. Such  an  $L$  is  called a  supporting function of $R$ 
and $R$  is  called an  extremal ray over  $W$  with  respect  to  $K_X+B$. 

\item For  an  extremal ray $R$,  there  exist an  open  neighborhood  $U$  of $W$ 
and a proper surjective morphism $\phi :f^{-1}(U)\to Z$ over  $U$ onto a normal variety 
$Z$  such  that \[\phi ( C) =\pt \qquad \mbox{ if and only if }\qquad [C]\in R\] 
for any irreducible  curve  $C$ of $f^{-1}(U)$ which  is  mapped to  a point of $W$. This $\phi$ is  denoted by ${\rm cont}_R$ and called the  contraction morphism associated with  $R$.

\item $\phi  =  {\rm cont}_R$ has  the following properties:
\begin{enumerate}
    \item  $-(K_X+B)|_{f^{-1}(U)}$ is  $\phi$-ample.
    \item   Let $E$ be an irreducible component of ${\rm Ex}(f)$ of maximal dimension, $n=\dim E-\dim f(E)$ and $p\in f(E)$ a general point, then $E_p=E\cap f^{-1}(p)$ is covered by a family of compact rational curves $\{\Gamma_t\}_{t\in T}$ such that $\phi(\Gamma_t)=\pt$ for all $t\in T$ and $-(K_X+B)\cdot \Gamma_t\<2n$, where $n=\dim X$.
    \item ${\rm Image} (\phi ^*:  {\rm Pic}(Z)\to {\rm Pic}(f^{-1}(U)))$ 
$=\{ D \in {\rm Pic}(f^{-1}(U))\;|\; (D\cdot \Gamma )=0\  \forall\   r  \in R\}$. 
    \item The following  mutually dual sequences are exact. 
 \[0\to N_1(f^{-1}(U)/Z;g^{-1}(W))\to N_1(X/Y; W)\to N_1(Z/U; W)\to 0,\] 
 \[0\leftarrow N^1(f^{-1}(U)/Z;g^{-1}(W))\leftarrow  N^1(X/Y; W)\leftarrow  N^1(Z/U; W)\leftarrow  0.\] 
Here  $g: Z\to U$ is the  structure  morphism. In  particular,  $\rho (X/Y; W)= \rho (Z/U; W)+1$.
\end{enumerate}
\end{enumerate}
\end{theorem}

\begin{proof}
	Everything here is in \cite[Theorem 4.12]{Nak87}, except the claim 3(b). This follows from \cite[Theorem 1.23]{DO23}. 
\end{proof}

Finally we prove a relative dlt cone theorem. 
\begin{theorem}\label{thm:general-relative-dlt-cone}
	Let $(X, \Delta)$ be a dlt pair, $f:X\to Y$ a projective surjective morphism of analytic varieties and $W\subset Y$ is a compact set satisfying either property $\mathbf P$ or $\mathbf Q$. Assume that $X$ is $\mbQ$-factorial over $W$. Then there are countably many rational curves $\{C_i\}_{i\in I}$ such that $f(C_i)=\pt$ for all $i\in I$, $0<-(K_X+\Delta)\cdot C_i\leq 2\dim X$ and 
	\[
		\NE(X/Y;W)=\NE(X/Y;W)_{(K_X+\Delta)\>0}+\sum_{i\in I}\mbR^{\geq 0}\cdot[C_i].
	\] 
\end{theorem}

\begin{proof}
	Fix a $f$-ample divisor $H$ on $X$. Then for any $n\in\mbN$ we can write $K_X+\Delta+\frac 1n H=K_X+(1-\ve)\Delta+(\frac 1n H+\ve \Delta)$ such that $\frac 1n H+\ve \Delta$ is $f$-ample for $\ve\in\mbQ^{\geq 0}$ sufficiently small (depending on $n$). Note that $(X, (1-\ve)\Delta)$ is a klt pair. Thus by Theorem \ref{thm:general-relative-cone}, there are finitely many $(K_X+\Delta+\frac 1n H)$-negative extremal rays generated by rational curves $\{C_i\}_{i\in I_n}$ contained in the fibers of $f$ such that $0<-(K_X+\Delta)\cdot C_i\leq 2\dim X$ and
	\begin{equation}\label{eqn:perturbed-cone}
		\NE(X/Y;W)=\NE(X/Y;W)_{(K_X+\Delta+\frac 1n H)\>0}+\sum_{i\in I_n}\mbR^{\geq 0}\cdot [C_i].
	\end{equation}
Define $I:=\cup_{n\>1} I_n$. Then clearly $\NE(X/Y;W)=\NE(X/Y;W)_{(K_X+\Delta+\frac 1n H)\>0}+\sum_{i\in I}\mbR^{\geq 0}\cdot [C_i]$. Note that we also have 

\[\NE(X/Y;W)_{(K_X+\Delta)\>0}=\cap_{n=1}^\infty \NE(X/Y;W)_{(K_X+\Delta+\frac 1n H)\>0}.\]
 Therefore from \eqref{eqn:perturbed-cone} we have
\begin{align*}
	\NE(X/Y;W) &=\cap_{n=1}^\infty \left(\NE(X/Y;W)_{(K_X+\Delta+\frac 1n H)\>0}+\sum_{i\in I}\mbR^{\geq 0}\cdot [C_i]\right)\\
			 &\supset \cap_{n=1}^\infty \left(\NE(X/Y;W)_{(K_X+\Delta+\frac 1n H)\>0}\right)+\sum_{i\in I}\mbR^{\geq 0}\cdot [C_i]\\
			 &=\NE(X/Y;W)_{(K_X+\Delta)\>0}+\sum_{i\in I}\mbR^{\geq 0}\cdot [C_i].
\end{align*}
	Suppose now that the inclusion is strict and so we have an element $v\in \cap_{n=1}^\infty \left(\NE(X/Y;W)_{(K_X+\Delta+\frac 1n H)\>0}+\sum_{i\in I}\mbR^{\geq 0}\cdot [C_i]\right)$ not contained in 
	\[\NE(X/Y;W)_{(K_X+\Delta)\>0}+\sum_{i\in I}\mbR^{\geq 0}\cdot [C_i].\]
	Intersecting $\NE(X/Y;W)$ with an appropriate affine hyperplane $\mathcal H$ we may assume that $\NE(X/Y;W)\cap \mathcal H$ is compact and convex and $v\in \NE(X/Y;W)\cap \mathcal H$. For each $n\>1$, we can write $v=v_n+w_n$, where $v_n\in \NE(X/Y;W)_{(K_X+\Delta+\frac 1n H)\>0}\cap \mathcal H$ and $w_n\in \sum_{i\in I}\mbR^{\geq 0}\cdot [C_i]\cap \mathcal H$. By compactness, passing to a subsequence, we may assume that limits exist, and  $v_\infty =\lim v_i$ and $w_\infty =\lim w_i$ such that $v=v_\infty +w_\infty$. Since $\NE(X/Y;W)_{(K_X+\Delta)\>0}= \cap_{n=1}^\infty \NE(X/Y;W)_{(K_X+\Delta+\frac 1n H)\>0}$ is closed, $v_\infty\in \NE(X/Y;W)_{(K_X+\Delta)\>0}\cap \mathcal H$.
	Since $\overline {\sum_{i\in I}\mbR^{\geq 0}\cdot [C_i]}\cap \mathcal H$ is compact, $w_\infty\in\overline {\sum_{i\in I}\mbR^{\geq 0}\cdot [C_i]}\cap \mathcal H$. 
	By standard arguments (see the end of the proof of \cite[Theorem III.1.2]{Kol96}) one sees that $\NE(X/Y;W)_{(K_X+\Delta)\>0}+{\sum_{i\in I}\mbR^{\geq 0}\cdot [C_i]}$ is closed and hence \[\overline {\sum_{i\in I}\mbR^{\geq 0}\cdot [C_i]}\subset \NE(X/Y;W)_{(K_X+\Delta)\>0}+{\sum_{i\in I}\mbR^{\geq 0}\cdot [C_i]}.\]
	Thus $w_\infty =v_0+w'_\infty$, where $v_0\in \NE(X/Y;W)_{(K_X+\Delta)\>0}$ and $w'_\infty\in {\sum_{i\in I}\mbR^{\geq 0}\cdot [C_i]}$. Finally, since $v=(v_\infty +v_0)+w'_\infty$, we obtain the required contradiction.
\end{proof}~\\

\part{MMP for Projective Morphisms}

\section{Finite generation following Cascini-Lazi\'c}

In \cite{CL10} it is shown that adjoint rings with big boundaries on   projective varieties are finitely generated.
In this section we will extend this result to the case of a projective morphism of analytic varieties. 
\begin{theorem}\label{t-a} Let
$\pi :X\to U$ be a projective morphism of complex analytic varieties, where $X$ is smooth variety with $\dim X=n$. Let $B_1, \ldots , B_k$
be $\mathbb Q$-divisors on $X$ such that $\lfloor B_i\rfloor  = 0$ for all $i$, and such that the support of
$\sum _{i=1}^k B_i$ has simple normal crossings. Let $A$ be a $\pi$-ample $\mathbb Q$-divisor on $X$, and
denote $D_i = K_X + A + B_i$ for every $i$.
Then the adjoint ring
\begin{equation*} 
R(X/U; D_1, \ldots , D_k) = \bigoplus _{(m_1,\ldots ,m_k)\in \mathbb N_k}
\pi _* \mathcal O _X \left(\lfloor \sum m_iD_i\rfloor \right)
\end{equation*}
is a locally finitely generated $\mathcal O _U$-algebra.
\end{theorem}
Note that if $K_X+B+A$ is klt and relatively nef, then the finite generation of $R(X/U, K_X+A+B)$ follows from the base point free theorem, cf. \cite[Theorem 4.8, Corollary 4.9]{Nak87}.
The proof in \cite{CL10} is an induction on the dimension proving the two statements \cite[Theorem A and  B]{CL10}. 
We will begin by showing that \cite[Theorem B]{CL10} implies a similar result in our setting. Recall that a (rational) \textit{polytope} $\mathcal P \subset \R ^n$ is  the convex hull of finitely many (rational) points in $\R ^n$.

\begin{theorem}\label{t-b}
Let $\left(X, \sum ^p_{i=1} S_i\right)$ be a log smooth pair, where
$S_1, \ldots , S_p$ are distinct prime divisors and $\pi :X\to U$ is a projective morphism to a Stein variety $U$. Let $V = \sum ^p
_{i=1} \mathbb R\cdot S_i \subset {\rm Div}_{ \mathbb R} (X)$, and $A\>0$ be a $\pi$-ample $\mbQ$-divisor on $X$. Define
\[
\mathcal L(V ):=
\{B = \sum_{i=1}^p b_iS_i \in V\; |\; 0 \leq b_i \leq 1\ {\rm for \ all}\ i\}.
\] 
Then
\[\mathcal E_A(V ):= \{B \in \mathcal L(V )\; :\; |K_X + A + B/U|_{\mathbb R} \ne  \emptyset \}\]
is a rational polytope.
\end{theorem}
\begin{proof} By \cite[Theorem B]{CL10}
we know that Theorem  \ref{t-b} holds in the projective case and hence we may assume that $\dim U>0$, so that 
$\dim X_u\leq n-1$ for any $u\in U$.

We will first prove the claim assuming that every $S_i$ dominates $U$.
Let $U'\subset U$ be the biggest open subset such that, 
denoting $\left(X', \sum ^{p}_{i=1} S'_i\right):=\left(X, \sum ^p_{i=1} S_i\right)\times _U U'$ then  $(X_u,\sum_{i=1}^p S_{i,u})$ is log smooth for any $u\in U'$.

Let $W_u$ be the subspace of ${\rm Div}_\R(X_u)$ spanned
by the irreducible components of $S_{i,u}$ and $V_u\subset W_u$ be the image of $V$ under the restriction map $r_u: {\rm Div}_\R(X)\to {\rm Div}_\R(X_u)$. Then $ \mathcal E _{A_u}(W_u)$ is a rational polytope, and hence so is $ \mathcal E _{A_u}(V_u)$ (since it is obtained by intersecting a rational polytope with a rational subspace).
Note that $r_u$ defines an isomorphism of $\mbR$-vector spaces $r_u:V\to V_u$. In what follows we often will identify $V$ and $V_u $.

 For every $u\in U'$, let $B_u^1,\ldots, B_u^{k_u}\in r_u({\rm Div}_\R(X))$ be a set of $\Q$-divisors generating the rational polytope $\mathcal E _{A_u}(V_u)$. 
 Consider the set  $\mathcal C^0=\{\mathbf B\}$ (resp. $\mathcal C^1=\{\mathbf B\}$) of finite subsets $\mathbf B=\{B^1,\ldots , B^k\}$, where $B^i\in  {\rm Div }_\Q (X)$ such that \[U (\mathbf B):=\{u\in U'\;|\;\mathcal E _{A_u}(V_u)=\langle B_u^1,\ldots ,B_u^k\rangle\}\]  is  (resp. is not)  {uncountably Zariski dense}. Here $\langle\mathbf B\rangle:=\langle B^1,\ldots ,B^k\rangle$ denotes the polytope spanned by $B^1,\ldots ,B^k$.
Note that $U'=\cup _{\mathbf B\in \mathcal C ^0\cup \mathcal C ^1}U(\mathbf B)$, where $\mathcal C ^0, \mathcal C ^1$ are countable as their elements are finite subsets of the countable set $V\cap {\rm Div }_\Q (X)$. Since, $\cup _{\mathbf B\in \mathcal C ^1}U(\mathbf B)$ is contained in a countable union of closed analytic subsets, then
 \[U^0:=\cup _{\mathbf B\in \mathcal C ^0}U(\mathbf B)=U'\setminus \cup _{\mathbf B\in \mathcal C ^1}U(\mathbf B)\]
 contains the complement of countably many analytic proper closed subsets in $U'$. In particular, $\mathcal C ^0$ is non-empty and for any 
 $\mathbf B\in \mathcal C ^0$, $U(\mathbf B)$ is {uncountably Zariski dense}.
 
 Fix $\bar {\mathbf B}\in \mathcal C ^0$ and write $\bar {\mathbf B}=\{\bar B ^1,\ldots , \bar B ^k\}$.
 For any $u\in U( \bar {\mathbf B} )$, we have that $\mathcal E _{A_u}(V_u)$ is the rational polytope generated by the $\mbQ$-divisors $\bar B ^1|_{X_u},\ldots , \bar B ^k|_{X_u}$ and there exists an integer $m=m(u)$ such that 
 $|m(K_{X_u}+A|_{X_u}+\bar B^i|_{X_u})|\ne \emptyset$ for all $1\leq i\leq k$. Since $U( \bar {\mathbf B} )$ is uncountably Zariski dense, it must contain an uncountably Zariski dense (in $U'$) set $W\subset U( \bar {\mathbf B} )$ such that $m(u)=\bar m$ is independent of $u\in W$.
 
 Now observe that, by the generic flatness, the upper semi-continuity theorem, and the cohomology and base-change theorem (see Theorem V.4.10 and Theorem III.4.12 in \cite{BS76}) it follows that there is a dense Zariski open subset $U^{\bar{m}}\subset U'$ such that $f$ is flat over $U^{\bar{m}}$ and 
 \[f_*\OO _X(\bar m(K_X+A+\bar B^i))\otimes\mbC(u)\to H^0(\OO _{X_u}(\bar m(K_{X_u}+A|_{X_u}+\bar B^i|_{X_u})))\]
 is an isomorphism for all $1\leq i\leq k$ and for all $u\in U^{\bar{m}}$.\\

Since $W$ is uncountably Zariski dense in $U'$, we have that $ W\cap U^{\bar{m}}\ne\emptyset$ and it follows from the relation above that $f_*\mcO_X(\bar{m}(K_X+A+\bar{B}^i))\otimes\mbC(u_0)\neq 0$ for any $u_0\in W\cap U^{\bar{m}}\subset U$. Then since $U$ is Stein, we have 
\[\Gamma (X,\OO _X(\bar m(K_X+A+\bar B^i)))=\Gamma (U, f_*\OO _X(\bar m(K_X+A+\bar B^i)))\ne 0.\] 
In particular, $|\bar m(K_X+A+\bar B^i)|\ne \emptyset$ and $|\bar{m}(K_{X_u}+A_u+\bar{B}^i_u)|\neq 0$ for all $1\<i\<k$ and for all $u\in U^{\bar{m}}$.\\
This shows that $\bar{B}^i_{u'}\in\mcE_{A_{u'}}(V_{u'})$ for all $u'\in U^{\bar{m}}$. Thus we have that 

    \begin{equation}\label{eqn:cone-comparison}
     r_u^{-1}\left( \mathcal E _{A_u}(V_u)\right) \subset r_{u'}^{-1}\left( \mathcal E _{A_{u'}}(V_{u'})\right)\qquad \text{ for all } u\in W \mbox{ and  for all } u'\in U^{\bar m}. 
\end{equation}
 Since $U^{\bar m}\cap U(  {\mathbf B} )\ne \emptyset $ for {any ${\mathbf B}\in \mathcal C ^0$},  it follows that for $u'\in U^{\bar m}\cap U(  {\mathbf B} )$ we have \[ \langle \bar{\mathbf B}\rangle=r_u^{-1}\left( \mathcal E _{A_u}(V_u)\right) \subset r_{u'}^{-1}\left( \mathcal E _{A_{u'}}(V_{u'})\right)=\langle {\mathbf B}\rangle.\]
 
  By symmetry, we have that $\langle{\mathbf B}\rangle= \langle\bar {\mathbf B}\rangle$ and hence $\mathcal C ^0=\{\bar{\mathbf B}\}$. In particular this shows that $\langle\bar {\mathbf B}\rangle\subset \mathcal E _{A}(V)$
 
 For the reverse inclusion, simply pick $B\in \mathcal E _{A}(V)$, then $\Gamma (X, \OO _X(m(K_X+A+ B)))\ne 0$ for some $m>0$, and so
 $\Gamma (X_u,\OO _{X_u}(m(K_{X_u}+A_u+ B_u)))\ne 0$ for general $u\in U$. This means that $B_u\in \mathcal E _{A_u}(V_u)$ for general $u\in U(\bar {\mathbf B})$,  and hence $B$ is contained in the polytope spanned by $\bar {\mathbf B}$.
 Thus $\mathcal E _{A}(V)=\langle\bar {\mathbf B}\rangle$ is a rational polytope.

To complete the proof, we consider the case when $S_1,\ldots ,S_{p'}$ dominate $U$ and $S_{p'+1},\ldots , S_p$ do not dominate $U$ (and hence $\pi (S_i)\cap U'=\emptyset$ for $i=p'+1,\ldots , p$). It suffices to show that if $B = \sum_{i=1}^p b_iS_i$ and $B' = \sum_{i=1}^{p'} b_iS_i$, then $B\in \mathcal E_A(V )$ if and only if $B'\in \mathcal E_A(V )$.  One direction is clear: if $B'\in \mathcal E_A(V )$ then $K_X + A + B'\sim _{\R,U}D'\geq 0$ so that $K_X + A + B\sim _{\R,U}D'+B-B'\geq 0$, and hence $B\in \mathcal E_A(V )$. Conversely, if $B\in \mathcal E_A(V )$, then $K_X + A + B\sim _{\R,U}D\geq 0$ and so $K_{X_u} + A_u + B_u\sim _{\R}D_u\geq 0$ for all $u\in U''$, where $U''$ is the largest open subset of $U'$ containing the points $u\in U'$ such that $X_u\not \subset {\rm Supp}(D)$. Note that $B_u=B'_u$ for any $u\in U''$ and hence $K_{X_u} + A_u + B'_u\sim _{\R}D_u\geq 0$ for all $u\in U''$. Then arguing as above there is a dense Zariski open subset $\bar U\subset U$ such that $f$ is flat over $\bar U$ and 
 \[f_*\OO _X(\bar m(K_X+A+ B'))\otimes\mbC(u)\to H^0(\OO _{X_u}(\bar m(K_{X_u}+A|_{X_u}+ B'|_{X_u})))\]
 is an isomorphism for all  $u\in \bar U$.
 Since $H^0(\OO _{X_u}(\bar m(K_{X_u}+A|_{X_u}+ B'|_{X_u})))=H^0(\OO _{X_u}(\bar m(K_{X_u}+A|_{X_u}+ B|_{X_u})))\ne 0$ for $u\in U''\cap \bar U$, then
 $f_*\OO _X(\bar m(K_X+A+ B'))\ne 0$.
 Since $U$ is Stein, then $H^0(f_*\OO _X(\bar m(K_X+A+ B')))\ne 0$ and so
it follows that $K_{X} + A + B'\sim _{\R}D^*\geq 0$ for some effective $\mbR$-divisor $D^*$ on $X$.
\end{proof}

The rest of this section is devoted to the proof of Theorem \ref{t-a}.
We will proceed by induction and show that Theorems \ref{t-a}$_{n-1}$ and \ref{t-b}$_{n}$ imply Theorem  \ref{t-a}$_{n}$ (here Theorem  \ref{t-a}$_{n}$ means Theorem  \ref{t-a} for $n$-dimensional varieties $\dim X=n$). Thus Theorem \ref{t-a}  holds in all dimensions.
Unluckily, we are unable to find a direct proof and so we will follow closely the arguments of \cite{CL10}.
We will not repeat the details of each step of the corresponding proof in \cite{CL10}, rather we will emphasize the necessary changes
to the statements, the arguments and the references used.
As remarked above, \cite{CL10} works with $X$ projective.
We will instead assume that $\pi:X\to U$ is a projective morphism of normal analytic varieties where $U$ is Stein  and relatively compact. If $(X,B)$ is a simple normal crossings pair, we will not assume (unless otherwise stated) that 
$(X,B)$ is simple normal crossings over $U$.
There are 3 kinds of results that play a prominent role in the arguments of \cite{CL10}.
The extension theorems from \cite[Section 3]{CL10} rely mainly on Kawamata-Viehweg vanishing and hence generalize easily to our context (cf. Theorem \ref{thm:relative-kvv}).
The results about convex polytopes and diophantine approximation require no changes.
The results about the Zariski decomposition are (with one simple exception discussed below) not used in the proof of Theorem  \ref{t-a}.
\subsection{Extension Theorems}
As an immediate consequence of Kawamata-Viehweg vanishing, one obtains the following basic extension result corresponding to \cite[Lemma 3.1]{CL10}.
   \begin{lemma}\label{l-3.1} Let $(X, B)$ be a log smooth pair of dimension $n$, where $B$ is
a $\mathbb Q$-divisor such that $\lfloor B\rfloor = 0$ and $\pi :X\to U$ is a projective morphism to a Stein variety. Let $A$ be a $\pi$-nef-big $\mathbb Q$-divisor.

(i) Let $S$ be a smooth prime divisor such that $S \not \subset  \Supp B$. If $G \in {\rm Div}(X)$ is
such that $G \sim _{\mathbb Q,U} K_X + S + A + B$, then 
$|G|_S=|G_{|S}|$.

(ii) Let $f : X \to Y$ be a bimeromorphic morphism of varieties projective over $U$, and let
$V \subset  X$ be an open set such that $f|_V$ is an isomorphism. 
{Let $H'$ be a very ample$/U$ divisor on $Y$}
and let $H = f ^*H'$. 
If $F \in {\rm Div}(X)$ is such that $F \sim _{\mathbb Q,U} K_X +(n+1)H +A+B$,
then $|F |$ is basepoint free at every point of $V$.\end{lemma}
    \begin{proof} Consider the short exact sequence
    \[0\to \mathcal O _X(G-S)\to \mathcal O _X(G)\to \mathcal O _S(G|_S)\to 0.\]
    By Kawamata-Viehweg vanishing (Theorem \ref{thm:relative-kvv}), we have $R^i\pi _* \mathcal O _X(G-S)=0 $
    for all $i>0$, and hence a surjection $\pi _*\mathcal O _X(G)\to \pi _*\mathcal O _S(G|_S)$; this is equivalent to (i), by Lemma \ref{l-lu}.

    The proof of (ii) proceeds by induction. Pick a point $x\in V$. Pick elements $T_1,\ldots ,T_{n}\in |H\otimes \mathfrak m_x|$, and let $X_0=X, X_i=T_1\cap \ldots \cap T_i$ for any $1\leq i\leq n$.
    Since $H'$ is very ample over $U$ and $U$ is Stein, $\mathcal O_X(H)\otimes \mathfrak m _x$ is generated over $U$. We may assume that $T_1\cap\ldots \cap T_n$ has a $0$-dimensional 
    component $X_n'$ supported at $x$. 
    For any $0\leq i\leq n-1$, we have short exact sequences
    \[ 0\to \OO _{X_i}((F-T_{i+1})|_{X_i})\to \OO _{X_i}(F|_{X_i})\to \OO _{X_{i+1}}(F|_{X_{i+1}})\to 0 .\]
    Since $R^k\pi_*\OO _X(F-lH)=0$ for $k>0$  and $0\leq l\leq n$ (by Kawamata-Viehweg vanishing, Theorem \ref{thm:relative-kvv} and by induction), {it is easy to see that $R^k\pi _* \OO _{X_i}((F-lH)|_{X_i})=0$ for $k>0$ and $0\leq l\leq n-i$} (cf. proof of \cite[Lemma 2.11]{Kaw99}). Thus the homomorphisms \[ \pi_*\OO _X(F)\to \pi_*\OO _{X_1}(F|_{X_1})\to \ldots \to \pi_*\OO _{X_n}(F|_{X_n})\] are surjective.
    Note that $x\in X$ is an irreducible component of the support of $\OO _{X_n}(F|_{X_n})$ and so there is a surjection $\OO _{X_n}(F|_{X_n})\to \OO _{X_n}(F|_{X_n})/\mathfrak m _x$. It follows that the evaluation map
    $\Gamma (\OO _X(F))\to \OO _X(F) /\mathfrak m _x$ is also surjective, i.e. $\OO _X(F)$ is generated at $x$.
    
    \end{proof}
 
      All results of \cite[Section 3]{CL10} follow similarly assuming that $\pi :X\to U$ is a projective morphism to a (relatively compact) Stein variety. For ease of notation we will say that $\pi :X\to U$ is a
morphism to a relatively compact variety if it is the restriction of a morphism $\pi' :X'\to U'$ over a relatively compact open subset $U\subset U'$ so that $X=X'\times _{U'}U$. 
Note that in this context, we consider $\pi$-ample, $\pi$-nef and $\pi$-big divisors
      instead of ample, nef and big divisors, however we do not require that smooth varieties (resp. log smooth pairs) are relatively smooth, i.e. the corresponding morphism is not assumed to be smooth.
     We  use Theorem \ref{thm:log-resolution} and Lemma \ref{l-log-resolution} for the existence of log resolutions, Theorem \ref{t-rel-ample} for useful facts about relatively ample divisors, Lemma \ref{l-klt} for a result about klt pairs, and Theorem \ref{t-bertini+} for the required Bertini Theorem.

\subsection{Proof of Finite Generation Theorem}
The key step in the proof of Theorem \ref{t-a} is to show that the restricted algebras are finitely generated (locally around every point $u\in U$).
In order to accomplish this we will need the following set up.
Let $(X,S+\sum _{i=1}^p S_i)$ be a log smooth pair, where $S,S_1, \ldots , S_p$ are distinct prime divisors and $\pi :X\to U$ is a projective morphism to a Stein space.
Let $V=\sum _{i=1}^p \mathbb R S_i\subset {\rm Div}_{\mathbb R}(X)$, $A$ be a $\pi$-ample $\mbQ$-divisor and $W\subset {\rm Div}_{\mathbb R}(S)$ is the subspace spanned by the components of $S_1|_S,\ldots ,S_p|_S$.
By Theorem \ref{t-b}, we know that $\mathcal E _{A|_S}(W)$ is a rational polytope. If $E_1,\ldots ,E_d$ are its vertices, then by induction on the dimension,
the ring $R(S/U;K_S+A|_S+E_1,\ldots , K_S+A|_S+E_d)$ is a locally finitely generated $\mathcal O _U$-algebra.
After shrinking $U$, we may assume that this ring is in fact a finitely generated $\mathcal O _U$-algebra.
For any $\Q$-divisor $E\in \mathcal E _{A|_S}(W)$ we let \[\mathbf F(E):={\mathbf{Fix}}(K_S+A|_S+E).\]
Recall that since $U$ is Stein, by Lemma \ref{l-2.3} we have ${\mathbf{Fix}}(K_S+A|_S+E)={\mathbf{Fix}}(K_S+A|_S+E/U)$.
By Lemma \ref{l-2.28}, $\mathbf F(E)$ extends to a rational piece-wise affine  function on $\mathcal E _{A|_S}(W)$,
and there exists an integer $k>0$ such that for any $E\in \mathcal E _{A|_S}(W)$ and any integer $m>0$ such that $(m/k)A|_S$ and $(m/k)E$ are integral, then \[ \mathbf F(E)=\frac 1m{\rm{Fix}}|m(K_S+A|_S+E)|.\]
The subset \[\mathcal F =\{ E\in \mathcal E _{A|_S}(W)\;|\; E\wedge \mathbf{F}(E)=0\}\subset \mathcal E _{A|_S}(W)\]
is defined by finitely many rational linear equalities and inequalities, and hence is a finite union of rational polytopes $\mathcal F=\cup _{i}\mathcal F _i$. For any $\mbQ$-divisor $B\in \mathcal B _A^S(V):=\{ B\in \mathcal L (V)\;|\; S\not \subset \mathbf {B}(K_X+S+A+B)\}$, we let
\[ {\mathbf F}_S(B):={\mathbf {Fix}}_S(K_X+S+A+B).\]
For any integer $m>0$ such that $mA$ and $mB$ are integral and $S\not \subset {\rm Bs}|m(K_X+S+A+B)|$, we let 
\[ \Phi_m (B):=B|_S -B|_S \wedge \frac 1m {\rm Fix}|m(K_X+S+A+B)|_S,\]
\[ {\mathbf {\Phi}} (B):=B|_S -B|_S \wedge {\mathbf F}_S(B)=\limsup \Phi_m (B).\]
With this notation and assumptions, we have the following analog of \cite[Lemma 4.2]{CL10}.
\begin{lemma}\label{l-4.2} If $B\in \mathcal B _A^S(V)$, then $\Phi _m(B)\in \mathcal E _{A|_S}(W)$ and $\Phi _m(B)\wedge {\mathbf F}(\Phi _m(B))=0$.
Thus if $\mathcal B _A^S(V)\ne \emptyset$, then $\mathcal F \ne \emptyset$.
\end{lemma}
\begin{proof}
This follows by the proof of \cite[Lemma 4.2]{CL10}.
\end{proof}
The next result is the analog of \cite[Theorem 4.3]{CL10}.
\begin{theorem}\label{t-4.3} Let $\mathcal G$ be a rational polytope contained in the interior of $\mathcal L (V)$, and assume that $(S,G_{|S})$ is terminal for every $G\in \mathcal G$. If $\mathcal P =\mathcal G \cap \mathcal B _A^S(V)$, then
\begin{enumerate}
    \item $\mathcal P$ is a rational polytope, and
    \item $\mathbf \Phi$ extends to a piece-wise affine function on $\mathcal P$, and there exists a
positive integer $\ell$ with the property that $\mathbf \Phi( P)=\Phi_m( P)$ for every 
$P\in \mathcal P$ and
every positive integer $m$ such that $mP/\ell$ is integral.
\end{enumerate}

\end{theorem}
\begin{proof}This follows by the proof of \cite[Theorem 4.3]{CL10}.

\end{proof}
\begin{theorem}\label{t-ra}  Assume Theorem \ref{t-a} in dimension $n-1$. Let $\pi :X\to U$ be a projective morphism of complex analytic varieties, where $X$ is smooth variety with $\dim X=n$. Let $S,S_1, \ldots , S_p$
be  distinct prime divisors on $X$ such that $S+\sum _{i=1}^pS_i$ has simple normal crossings. Let $A$ be a $\pi$-ample $\mathbb Q$-divisor on $X$, $V=\sum_{i=1}^p \mathbb R S_i\subset {\rm Div}_{\mathbb R}(X)$, $B_1,\ldots , B_m \in \mathcal E_{S+A}(V)$ be $\Q$-divisors and
denote $D_i = K_X + S+A + B_i$ for every $i$.
Then the ring ${\rm res}_S R(X/U; D_1, \ldots , D_m)$
is a locally finitely generated $\mathcal O _U$-module.
\end{theorem}
\begin{proof} 
This follows along the lines of the proof of \cite[Lemma 6.2]{CL10}.
\end{proof}~\\

\begin{proof}[Proof of Theorem \ref{t-a}]
Let $\mathcal P=\mbox{conv}(B_1,\ldots , B_k)\subset {\rm Div}_\R (X)$ be the polytope spanned by $B_i$ and $\mathcal R=\R^{\geq 0}(K_X+A+\mathcal P)$. We may assume that $U$ is a relatively compact Stein space. It suffices to show that $R(X/U, \mathcal R) $ is locally finitely generated (cf. \cite[Lemma 2.27]{CL10}). 
By Theorem \ref{t-b},  $\mathcal P_{\mathcal E}=\mathcal P\cap \mathcal E_A(V)$ is a rational polytope, where $V\subset {\rm Div}_\R (X)$ is the vector space spanned by the components of $B_1,\ldots , B_k$. Since $H^0(X,\OO _X(K_X+A+D))=0$ for any divisor $D\in \mathcal P\setminus \mathcal P_{\mathcal E}$, it suffices to show that
$R(X/U, \mathcal R_{\mathcal E}) $ is locally finitely generated, where $\mathcal R_{\mathcal E}=\R^{\geq 0}(K_X+A+\mathcal P_{\mathcal E})$. By Gordan's lemma (cf. \cite[Lemma 2.11]{CL10}) the monoid 
$\mathcal R_{\mathcal E}\cap \Div(X)$ is finitely generated, so there are $\Q$-divisors $R_i=p_i(K_X+A+P_i)$, where $p_i\in \Q^{\geq 0}$ and $P_i$ are $\Q$-divisors with simple normal crossings support such that $\lfloor P_i\rfloor=0$ for $1\leq i\leq \ell$. Since $P_i\in \mathcal E_A(V)$, we have $K_X+A+P_i\sim _{\Q,U}G_i\geq 0$.
Replacing $B_1,\ldots , B_k$ by $P_1,\ldots , P_\ell$, we may assume that $K_X+A+B_i\sim _{\Q,U}F_i\geq 0$ for all $i$.
Replacing $X$ by a log resolution (see Theorem \ref{thm:log-resolution}), we may assume that $\left(X,\sum (B_i+F_i)\right)$ is a simple normal crossings pair. 

Consider now $W\subset {\rm Div }_\R (X)$ the subspace spanned by the components $S_1,\ldots , S_p$ of $\sum (B_i+F_i)$. Let $\mathcal T=\{(t_1,\ldots , t_k)\;|\;t_i\geq 0,\ \sum t_i=1\}$ and for any $\tau =(t_1,\ldots , t_k)\in \mathcal T$, we let
\[ B_\tau=\sum t_iB_i,\qquad F_\tau=\sum t_iF_i\sim _{\Q,U}K_X+A+B_\tau.\]
Consider the following rational polytopes for $1\leq i\leq p$,
\[\mathcal B =\{F_\tau +B\;|\;\tau \in \mathcal T,\ 0\leq B\in W,\ B_\tau+B\in \mathcal L (W)\}\subset W ,\]
\[\mathcal B_i=\{F_\tau+B \in \mathcal B
\;|\; S_i\subset \lfloor B_\tau+B\rfloor \}\subset W .\]
We also have rational polyhedral cones $\mathcal C=\mathbb R^{\geq 0}\mathcal B$,  $\mathcal C_i=\mathbb R ^{\geq 0}\mathcal B_i$ and monoids $\mathcal S =\mathcal C\cap {\rm Div}(X)$, $\mathcal S _i =\mathcal C_i\cap {\rm Div}(X)$.
Following the proof of \cite[Theorem 6.3]{CL10}, it suffices to show that
\begin{enumerate}
    \item $\mathcal C =\cup _{i=1}^p\mathcal C _i$,
    \item there exists an integer $M>0$ such that if $\sum \alpha _i S_i\in \mathcal C _j$ for some $j$ and some $\alpha _i\in \mathbb N$ with $\sum \alpha _i \geq M$, then $\sum \alpha _i S_i-S_j\in \mathcal C$, and
    \item the rings ${\rm res} _{S_j}R(X/U,\mathcal S _j)$ are locally finitely generated for $1\leq j\leq p$.
\end{enumerate}
(1) Pick $0\ne G\in \mathcal C$. Then there exists $\tau \in \mathcal T$, $0\<B\in W$ and $r>0$ such that $B_\tau +B\in\mcL(W)$ and $G=r (F_\tau +B)$. Let \[ \lambda ={\rm max}\{t\geq 1|B_\tau+tB+(t-1)F_\tau \in \mathcal L (W)\},\]
and $B'=\lambda B+(\lambda -1)F_\tau $, then \[\lambda G=\lambda r(F_\tau +B)=r(F_\tau +\lambda B+(\lambda -1)F_\tau)=r(F_\tau +B'),\] where $\lfloor B_\tau +B'\rfloor$ is non-empty and hence contains a component $S_{j_0}$ for some $1\<j_0\<p$. Thus $G\in \mathcal C _{j_0}$ as required.

(2) Fix $\epsilon >0 $ such that the coefficients of $B_i$ are $\leq 1-\epsilon $, and hence
for any $\tau \in \mathcal T$ the coefficients of $B_\tau $ are also $\leq 1-\epsilon $. Now let $||\cdot ||$ be the sup norm on the vector space $W$ so that for any $D\in W$, $||D||$ is the largest coefficient of $D$ in the unique decomposition $D=\sum_{i=1}^p a_iS_i$. Since each set $\mathcal B_j$ is compact, there exists a constant $C>0$ such that  for any $\Psi \in \cup _{j=1}^p\mathcal B _j$ we have $||\Psi||\leq C$. Define $M:=pC/\epsilon$. 
Let $G=\sum _{i=1}^p\alpha _i S_i\in \mathcal C _j$, where $\sum _{i=1}^p\alpha _i \geq M$. Then 
\[||G||={\rm max}\{\alpha _i\}\geq \frac{\sum _{i=1}^p\alpha _i}p\geq \frac M p=\frac C \epsilon . \]
Since $G\in \mathcal C _j$, there exists $r>0$ such that  $G=r G'$ for some $G'\in \mathcal B _j$. Thus $||G'||\leq C$, and hence $r=\frac{||G||}{||G'||}\geq \frac 1 \epsilon$.
Since $G'\in \mathcal B _j$, we may write $G'=F_\tau +B$ where $\tau \in \mathcal T$, $0\leq B\in W$, $B_\tau +B\in \mathcal L (W)$ and $S_j\subset \lfloor B_\tau +B\rfloor$.
But then ${\rm mult}_{S_j}(B)=1-{\rm mult}_{S_j}(B_\tau )\geq \epsilon \geq 1/r$, so that
\[ G-S_j=r(F_\tau +B-\frac 1r S_j)\in \mathcal C.\]

(3)  We pick generators $E_1,\ldots , E_l$ of $\mathcal S _j=\mathcal C_j\cap {\rm Div}(X)$. For any $i\in \{ 1, \ldots ,l\}$, there exist
$k_i\in \mathbb Q ^{>0}$, $\tau _i\in \mathcal T \cap \mathbb Q ^k$,
$0\leq B_i\in W$ such that $B_{\tau _i}+B_i\in \mathcal L (W)$,
$S_j\leq \lfloor B_{\tau _i}+B_i\rfloor$ and $E_i=k_i(F_{\tau _i}+B_i)$. If $E'_i:=K_X+A+B_{\tau _i}+B_i$, then $E_i\sim_{\mbQ} k_iE_i'$.
Now, ${\rm res}_{S_j}(X/U;E'_1,\ldots , E'_l)$ is finitely generated by
Theorem \ref{t-ra} and hence ${\rm res}_{S_j}(X/U;E_1,\ldots , E_l)$ is also finitely generated (cf. \cite[Lemma 2.25]{CL10}).
Finally the claim follows from the surjection
\[{\rm res}_{S_j}(X/U;E_1,\ldots , E_l)\to {\rm res}_{S_j}(X/U;\mathcal S _j).\]
\end{proof}~\\

\begin{corollary}\label{c-fg1}
Let $\pi :X\to U$ be a projective morphism of normal varieties and  $(X,B)$ is a klt pair such that $K_X+B$ is $\pi$-big. Then $R(X/U, K_X+B):=\oplus_{m\>0}\pi_*\mcO_X(\lrd m(K_X+B)\rrd)$ is locally finitely generated over $U$. In particular, if $W\subset U$ is a compact subset, then after shrinking $U$ near $W$ suitably, $R(X/U, K_X+B)$ is finitely generated over $U$, and hence the log canonical model ${\rm Projan}R(X/U, K_X+B)\to U$ of $(X,B)$ over $U$ exists.
\end{corollary}

\begin{proof}
Working locally on $U$ we may assume that $U$ is a relatively compact Stein space. Since $K_X+B$ is $\pi$-big, we have $K_X+B\sim _{\Q,U} A+N$, where $A$ is $\pi$-ample $\mbQ$-divisor and $N\geq 0$. Let $f:Y\to X$ be a log resolution of $(X, B+N)$ as in Theorem \ref{thm:log-resolution}. Write  $K_Y+\Gamma=f^*(K_X+B)+E$ such that $\Gamma\>0, E\>0, \Gamma\wedge E=\emptyset, f_*\Gamma=B$ and $f_*E=0$. Let $F\>0$ be a $f$-exceptional $\mbQ$-divisor such that $-F$ is $f$-ample. Then $A'=f^*A-F$ is $(\pi\circ f)$-ample. Choose a rational number $0<\eps\ll 1$ such that $(Y, \Gamma+\eps f^*N+\eps F)$ is klt and
\[
(1+\eps)f^*(K_X+B)+E\sim_{\mbQ} K_Y+\Gamma+\eps f^*N+\eps F+\eps A'.
\]
Thus from Theorem \ref{t-a} and \cite[Corollary 2.26]{CL10} it follows that $R(X/U, K_X+B)$ is locally finitely generated over $U$.\\ 

Moreover, if $W\subset U$ is a compact subset, then there exists a positive integer $m>0$ and finitely many open subsets $U_i\subset U, 1\leq i\leq k$ such that $W\subset \cup_{i=1}^k U_i $ and $R(X_i/U_i, K_{X_i}+B_i)$ is finitely generated in degree $\<m$ for all $i=1,2\ldots, k$; where $X_i=X\times _U U_i$ and $B_i=B|_{X_i}$.
The claim then follows replacing $U$ by $\cup_{i=1}^k U_i$.
\end{proof}

\section{Relative MMP for projective morphisms}
In this section we prove Theorems \ref{c-fg} and \ref{t-mmpscale}.\\

\begin{proof}[Proof of Theorem \ref{c-fg}]
 We follow the ideas of \cite{Fuj15}. By the proof of \cite[Theorem 5.1]{Fuj15} (also see \cite[21.5]{Fuj22}), there is a projective morphism $g:Z\to Y$ from a complex manifold $Z$ such that $X\dasharrow Z$ is bimeromorphic to the Iitaka fibration of $K_X+B$ over $Y$ and $(Z,B_Z\>0)$ is a log smooth klt pair such that $K_Z+B_Z$ is big over $Y$ and 
\[ \oplus_{m\geq 0} f_*\mathcal O _X(me(K_X+B))\cong \oplus_{m\geq 0} g_*\mathcal O _Z(me'(K_Z+B_Z))\] 
for some integers $e,e'>0$. {Then the result follows from Corollary} \ref{c-fg1}.

\end{proof}~\\

\begin{proof}[Proof of Theorem \ref{t-mmpscale}] We are free to replace $U$ by arbitrarily small neighborhoods of $W$, see \cite[1.11]{Fuj22}.
If $K_X+B$ is nef over $W$, then there is nothing to prove.
Otherwise, by the Cone Theorem (cf. Theorem \ref{thm:general-relative-cone}), there is a negative extremal ray $R=\R ^{\geq 0}[l]$
and a divisor $L\in A^1(X/U;W)$ such that $R=\overline{NE}(X/U;W)\cap L^\perp$, where $L$ is nef over $U$.
Let $\phi={\rm cont}_R:X\to Z$ be the corresponding morphism (which is defined after possibly further shrinking $U$). If $\dim Z<\dim X$, this is a Mori fiber space. If $\dim Z=\dim X$ and $\phi$ contracts a divisor, then this is a divisorial contraction. In this case we let $(Z,\phi _* B)=(X_1,B_1)$ and we note that $(X_1,B_1)$ is klt and $\Q$-factorial near $W$. If on the other hand, $\dim Z=\dim X$ and $\phi$ is small, then by {Corollary} \ref{c-fg1}, $R(X/Z, K_X+B)$ is finitely generated and hence we obtain a small birational morphism $\psi :X\dasharrow X_1:={\rm Projan}R(X/Z, K_X+B)$ (note that as $X\to Z$ is birational, the bigness assumption is automatically satisfied). In this case  we note that $(X_1,B_1=\psi _*B)$ is klt and $\Q$-factorial near $W$. We may now replace $(X,B)$ by $(X_1,B_1)$ and repeat the procedure. This proves (1).

Suppose now that we are running the MMP with scaling of a sufficiently $\pi$-ample $\mbQ$-divisor $A$.
This means that we have a sequence of flips and divisorial contractions
\[(X,B)=(X_0, B_0)\dasharrow (X_1,B_1)\dasharrow (X_2,B_2)\dasharrow \ldots,\] a $\pi$-ample divisor $A$ and a sequence of rational numbers $\lambda _0\geq \lambda _1\geq \lambda _2\geq \ldots \geq 0$ such that 
$K_{X_i}+B_i+\lambda A_i$ is nef over $U$ for $\lambda _i\geq \lambda \geq \lambda _{i+1}$.
It suffices to show that this sequence terminates locally over a neighborhood of any point of $W$. 
We claim that there exists a constant $\epsilon>0$ such that we may assume that $B\geq \epsilon A$.
Indeed, in Case (2), first if $B$ is $\pi$-big, then $B\sim _{\Q,U}\delta A +E$, where $\delta >0$ and $E\geq 0$. Then for any rational number $0<\gamma\ll 1$, $(X, B'=(1-\gamma)B+\gamma(\delta A+E))$ is klt, and $K_X+B'\sim_{\mbQ, U} K_X+B$. Thus the above MMP is also a $(K_X+B')$-MMP with the scaling of $A$, and $B'\>\gamma\delta A$. In this case we are done by replacing $B$ by $B'$ and setting $\eps=\gamma\delta$. On the other hand, if $K_X+B$ is $\pi$-big, then we can write $K_X+B\sim _{\Q,U}\delta A +E$, where $\delta >0$ and $E\geq 0$. Then again for any rational number $0<\gamma \ll 1$, $(X,B'=B+\gamma (\delta A +E))$ is klt and $K_X+B'\sim_{\Q,U} (1+\gamma)(K_X+B)$. It follows that the above MMP is a $(K_X+B')$-MMP with scaling of $(1+\gamma \delta)A$. Since $B'\geq \gamma \delta A$, the claim follows letting $\epsilon:=\frac {\gamma \delta}{1+\gamma}$. In Case (3) this holds since $K_X+B$ is not $\pi$-pseudo-effective, and therefore $K_X+B+\epsilon A$ is not $\pi$-pseudo-effective for some $0<\epsilon \ll 1$. In particular, $\lambda _i>\epsilon$ for all $i$. Replacing $B$ by $B+\epsilon A$ and $\lambda _i$ by $\lambda _i-\epsilon$ the claim follows.

We will now show that the corresponding minimal model programs terminate.
Suppose that $U$ is relatively compact and Stein.
Fix $||\cdot ||$ a norm on $N^1 (X/U;W)$. Let $\lambda =\lim \lambda _i$.

We may pick relatively ample $\Q$-divisors $H_1,\ldots ,H_r$
such that 
\begin{enumerate}
\item $H_j\geq \epsilon A$ for some $0<\epsilon \ll 1$ and $1\leq j\leq r$,
\item $(X,H_j)$ is klt for $1\leq j\leq r$,
    \item $||(B+\lambda A)- H_j||\ll 1$ for $1\leq j\leq r$,
    \item if $\mathcal C =\mathbb R^{\geq 0}(K_X+B)+\sum _{j=1}^r\mathbb R^{\geq 0}(K_X+H_j)\subset {\rm Div}_\R (X)$, then $K_X+B+\lambda A$ is in the interior of $\mathcal C$, and the dimension of $\mathcal C$ equals $\dim N^1(X/U;W)$.
\end{enumerate}
By Theorem \ref{t-a}, it then follows that $R(X/U, \mathcal C)$ is a  finitely generated $\mathcal O _U$ algebra.
Arguing as in \cite[Theorem 6.5]{CL13}, the corresponding MMP with scaling terminates.

In the general case, we observe that since $U$ is relatively compact, it suffices to prove termination locally over $U$. This follows from the above argument.
\end{proof}

\part{MMP in dimension $4$}

\section{Cone and Contraction Theorems}
\subsection{Cone and contraction theorems in dimension 3}
We begin by proving a unified cone theorem for $\mbQ$-factorial dlt pairs $(X, B)$ that works both when $K_X+B$ pseudo-effective and non pseudo-effective. We will need the following lemma on  the length of extremal rays.

\begin{lemma}\label{lem:lc-length}
 	Let $(X, \Delta)$ be a compact K\"ahler lc pair of dimension $n$. Let $\Delta_0\>0$ be a $\mbQ$-divisor such that $(X, \Delta_0)$ is klt. Let $R$ be a $(K_X+\Delta)$-negative extremal ray of $\NA(X)$ and $f:X\to Y$ is the projective morphism contracting $R$, i.e. a curve $C\subset X$ is contracted by $f$ if and only if $[C]\in R$. Then there is a rational curve $\Gamma\subset X$ contained in a fiber of $f$ such that $R=\mbR^{\geq 0}\cdot[\Gamma]$ and
	\[
		0<-(K_X+\Delta)\cdot\Gamma\<2n.
	\]
 \end{lemma}
 
 \begin{proof}
 	The same proof as in \cite[Theorem 3.8.1]{BCHM10} works using \cite[Theorem 1.23]{DO23} in place of \cite[Theorem 1]{Kaw91}.
 \end{proof}

 \begin{theorem}\label{thm:unified-cone-thm}
	 Let $(X,B)$ be a $\mbQ$-factorial compact K\"ahler $3$-fold dlt pair. Then there exists a countable collection of rational curves $\{C_i\}_{i\in I}$ such that $0<-(K_X+B )\cdot C_i\leq 6$ and
 \[  \NA(X)=\NA(X)_{(K_X+B )\geq 0}+\sum _{i\in I}\mathbb{R}^{\geq 0}\cdot[C_i].\]
Moreover, if $\omega$ is a K\"ahler class, then there are only finitely many extremal rays $R_i=\mbR^{\geq 0}\cdot[C_i]$ satisfying $(K_X+B +\omega)\cdot C_i<0$ for $i\in I$.
 \end{theorem}
 
 \begin{proof} 
 
 The theorem is known when $K_X+B$ is pseudo-effective (see \cite[Theorem 1.3]{CHP16} and \cite[Theorem 2.26]{DO23}).
 
We give a unified proof of all the cases below.\\

 Let $\{C_i\}_{i\in I}$ be a set of curves spanning the $(K_X+B)$-negative extremal rays such that $0<-(K_X+B )\cdot C_i\leq 6$, and $N\subset N_1(X)$ be the cone defined by
 \[
 	N:=\NA(X)_{(K_X+B)\>0}+\sum_{i\in I} \mbR^{\geq 0}\cdot[\Gamma_i].
 \]
First we will show that $\NA(X)=\overline{N}$. Clearly $\overline{N}\subset \NA(X)$ holds, so assume that the reverse inclusion does not hold. Then there is a nef class $\alpha$ such that $\alpha \cdot \gamma >0$ for all $0\neq\gamma \in \overline{N}$ and $\alpha \cdot \gamma=0$ for some $\gamma\in\NA(X)$. Clearly $\NA(X)_{(K_X+B)\>0}$ is a closed sub-cone of $\NA(X)$. Let $K$ be a compact slice of $\NA(X)_{(K_X+B)\>0}$. Then there exists an $\epsilon >0$ such that $(\alpha -\epsilon(K_X+B))\cdot \gamma >0$ for all $0\neq\gamma \in K$. But then $\alpha-\epsilon(K_X+B)$ is strictly positive on $\NA(X)\setminus\{0\}$, and so $\alpha'=\frac{1}{\epsilon} \alpha=K_X+B+\omega$, where $\omega:=\frac{1}{\epsilon}(\alpha-\epsilon(K_X+B))$ is  strictly positive on $\overline {{\rm NA}}(X)$, and hence a  K\"ahler class by \cite[Corollary 3.16]{HP16}. Replacing $\alpha$ by $\alpha'$ we may assume that $\alpha=K_X+B+\omega$ for some K\"ahler class $\omega$ such that $\alpha\cdot\gamma>0$ for all $0\neq \gamma\in\overline{N}$ and $\alpha\cdot\gamma=0$ for some $\gamma\in\NA(X)$. By \cite[Theorem 1.7]{DH20}, there exists a projective morphism $\psi :X\to Z$ such that $K_X+B+\omega\equiv \psi ^* \alpha _Z $ where $\alpha _Z $ is a K\"ahler class on $Z$. In particular, $\psi$ is not an isomorphism, and thus by Theorem \ref{thm:general-relative-dlt-cone}, there is a rational curve $C\subset X$ contained in a fiber of $\psi$ generating an extremal ray contained in the face $\alpha^\bot\cap\NA(X)$ such that $0>(K_X+B)\cdot C\geq -6$ and $(K_X+B+\omega )\cdot C=0$. This is a contradiction and so
$\NA(X)=\overline{N}$.

By \cite[Lemma 6.1]{HP16} it follows that $N$ is a closed cone and so $\NA(X)={N}$. We note that the proof of \cite[Lemma 6.1]{HP16} works with $K_X$ replaced by $K_X+B$. Finally, for any K\"ahler class $\omega$, if $(K_X+B+\omega)\cdot C_i<0$, then $\omega\cdot C_i<-(K_X+B)\cdot C_i\<6$. Hence by a Douady space argument there are finitely many extremal rays $R_i=\mbR^{\geq 0}\cdot[C_i]$ satisfying $(K_X+B+\omega)\cdot R_i<0$. 

 \end{proof}~\\

We deduce the non $\mbQ$-factorial version of this theorem below which is used throughout the article.
\begin{corollary}\label{cor:nQ-unified-cone}
Let $(X, B)$ be a compact K\"ahler $3$-fold dlt pair. Then there exists a countable collection of rational curves $\{C_i\}_{i\in I}$ such that $0<-(K_X+B)\cdot C_i\<6$ for all $i\in I$ and 
\[ 
\NA(X)=\NA(X)_{(K_X+B)\>0}+\sum_{i\in I} \mbR^{\geq 0}\cdot [C_i].
\]
Moreover the following holds:
\begin{enumerate}
    \item For any K\"ahler class $\omega$, there are only finitely many extremal rays $R_i:=\mbR^{\geq 0}\cdot[C_i]$ such that $(K_X+B+\omega)\cdot R_i<0$.
    
    \item For any $(K_X+B)$-negative extremal ray $R=\mbR^{\geq 0}\cdot[C_{i}]$, there is a nef class $\alpha\in H^{1, 1}_{\BC}(X)$ such that $\alpha^\bot\cap\NA(X)=R$ and $\alpha=K_X+B+\eta$ for some K\"ahler class $\eta$.
\end{enumerate}
\end{corollary}

\begin{proof}
Since $(X, B)$ is a dlt pair, there is a log resolution $\phi:Y\to X$ of $(X, B)$ such that $a(E, X, B)>-1$ for all exceptional divisors $E$ of $\phi$. Define $B_Y:=\phi^{-1}_*B+\Ex(\phi)$. Then running a $(K_Y+B_Y)$-MMP over $X$ as in \cite[Proposition 2.26]{DH20}, we may assume that there is a $\mbQ$-factorial dlt pair $(X', B')$ and a small projective bimeromorphic morphism $f:X'\to X$ such that $K_{X'}+B'=f^*(K_X+B)$. Then by the cone Theorem \ref{thm:unified-cone-thm} for $(X', B')$ there exist countably many rational curves $\{C'_i\}_{i\in I'}$ on $X'$ such that $0<-(K_{X'}+B')\cdot C'_i\<6$ for all $i\in I'$ and
\begin{equation}\label{eqn:cone-on-model}
    \NA(X')=\NA(X')_{(K_{X'}+B')\>0}+\sum_{i\in I'} \mbR^{\geq 0}\cdot [C'_i].
\end{equation}
Now from \cite[Proposition 3.14]{HP16} it follows that $f_*\NA(X')=\NA(X)$. Let $f(C'_i)=C_i\subset X$ for all $i\in I'$ such that $f(C'_i)\neq \pt$ and let $\{C_i\}_{i\in I}$ be the collection of all non contracted curves. Applying $f_*$ on both sides of \eqref{eqn:cone-on-model} we claim $\NA(X)=\NA(X)_{(K_{X}+B)\>0}+\sum_{i\in I} \mbR^{\>0}\cdot [C_i]$. If not, then assume that 
\begin{equation}\label{eqn:cone-equality-on-x}
    \NA(X)\supsetneq \NA(X)_{(K_X+B)\>0}+\sum_{i\in I} \mbR^{\>0}\cdot [C_i]. 
\end{equation}

Then there exists a $(1, 1)$ class $\alpha\in N^1(X)$ such that $\alpha$ is positive on the RHS of \eqref{eqn:cone-equality-on-x} and $\NA(X)\cap \alpha_{\<0}\neq \emptyset$. Let $\omega$ be a K\"ahler class on $X$ and $\lambda\in\mbR^{\geq 0}$ is defined as $\lambda:=\inf \{t\>0: \alpha+t\omega \mbox{ is a K\"ahler class}\}$. Then $\alpha+\lambda\omega$ is a nef class which is not K\"ahler. Consequently, from \cite[Corollary 3.16]{HP16} it follows that $(\alpha+\lambda\omega)^\bot \cap \NA(X)\neq \{0\}$. Let $\beta=\alpha+\lambda\omega$; then $\beta$ is nef and $\beta^\bot\cap \NA(X)\neq \{0\}$. In particular, $\beta^\bot\cap \NA(X)$ is an extremal face of $\NA(X)$; let's denote this face by $F$. Let $F'$ be the extremal face of $\NA(X')$ defined by $f^*\beta$, i.e. $F':=(f^*\beta)^\bot\cap\NA(X')$. Then from Lemma \ref{lem:cone-push-forward} it follows that $F'=(f^*)^{-1}F\cap\NA(X')$. We claim that $K_{X'}+B'$ is negative on $F'\setminus (f^*)^{-1}(\mathbf{0})$, where $\mathbf{0}\in\NA(X)$ is the zero vector. Indeed, if $\gamma'\in F'\setminus (f^*)^{-1}(\mathbf{0})$, then $f_*\gamma'\in F\setminus \{\mathbf{0}\}$ and thus $\beta \cdot f_*\gamma'=0$. In particular, $\alpha\cdot f_*\gamma'<0$, and thus $(K_X+B)\cdot f_*\gamma'<0$ (it follows from the construction of $\alpha$). Therefore by the projection formula we have $(K_{X'}+B')\cdot \gamma'<0$.\\
Now by the cone theorem on $(X', B')$ and \cite[Theorem 1.7]{DH20}, there must be a $(K_{X'}+B')$-negative extremal ray, say $R'=\mbR^{\>0}\cdot[C'_i]\subset F'\setminus (f^*)^{-1}(\mathbf{0})$. Then $C_i=f(C'_i)\neq \pt$ is one of the curves in the collection $\{C_i\}_{i\in I}$ above and $(K_X+B)\cdot C_i<0$. But $[C_i]\in F=\beta ^\bot\cap \NA(X)$ and this is a contradiction, since by our assumption $\alpha\cdot C_i>0$, and hence $\beta \cdot C_i= (\alpha+\lambda\omega)\cdot C_i>0$.\\

Now $(1)$ is proved exactly as in Theorem \ref{thm:unified-cone-thm}. For the second part, from \cite[Lemma 6.1]{HP16} we see that $V=\NA(X)_{(K_X+B)\>0}+\sum_{i\in I, i\neq i_0} \mbR^{\geq 0}[C_i]$ is a closed subcone on $\NA(X)$; note that \cite[Lemma 6.1]{HP16} is only stated for $K_X$, but this was never used in the proof, and the exact same proof works for $K_X+B$. Then by \cite[Lemma 6.7(d)]{Deb01} there is a nef class $\alpha\in H^{1, 1}_{\BC}(X)$ such that $\alpha$ is strictly positive on $V\setminus \{\mathbf{0}\}$ and $\alpha^\bot\cap\NA(X)=R$. Then scaling $\alpha$ appropriately we observe that $\alpha-(K_X+B)$ is strictly positive on $\NA(X)\setminus \{\mathbf{0}\}$, and thus by \cite[Corollary 3.16]{HP16}, $\alpha-(K_X+B)=\eta$ is a K\"ahler class on $X$, i.e. $\alpha=K_X+B+\eta$. 

\end{proof}~\\

The following lemma is taken from \cite{Wal18}.
\begin{lemma}\cite[Lemma 3.1]{Wal18}\label{lem:cone-push-forward}
Let $f:V\to W$ be a surjective linear transformation of finite dimensional vector spaces over $\mbR$. Suppose that $C_V\subset V$ and $C_W\subset W$ are closed convex cones of maximal dimensions and $H\subset W$ is a vector subspace of codimension $1$. Assume that the following hold:
\begin{enumerate}
    \item $f(C_V)=C_W$.
    \item $C_W\cap H\subset \partial C_W.$
\end{enumerate}
Then $f^{-1}H\cap C_V\subset \partial C_V$ and also $f^{-1}H\cap C_V=f^{-1}(H\cap C_W)\cap C_V$.
\end{lemma}~\\

The following contraction theorem is a direct generalization of \cite[Theorem 1.7]{DH20}.
\begin{theorem}[Contraction Theorem]\label{thm:contraction-non-q-factorial}
Let $(X, B)$ be a compact K\"ahler $3$-fold klt pair, and $\alpha\in N^1(X)$ be a nef class such that $\alpha-(K_X+B)$ is nef and big. 
Then there exists a proper morphism $f:X\to Y$ with connected fibers to a normal compact K\"ahler variety $Y$ with rational singularities and a K\"ahler class $\omega_Y\in N^1(Y)$ such that $\alpha=f^*\omega_Y$. In particular, if $\alpha-(K_X+B)$ is a K\"ahler class, then $f$ is projective.
 \end{theorem}

\begin{proof}
Let $g:Z\to X$ be a small $\mbQ$-factorialization of $X$ obtained by running an appropriate relative MMP on a log resolution of $(X, B)$ as in \cite[Proposition 2.26]{DH20}. Set $K_Z+B_Z:=g^*(K_X+B)$; then $g^*\alpha-(K_Z+B_Z)$ is nef and and big. Thus by \cite[Theorem 1.7]{DH20}, there is proper morphism $h:Z\to Y$ with connected fibers to a normal compact K\"ahler variety $Y$ with rational singularities and a K\"ahler class $\omega_Y\in N^1(Y)$ such that $g^*\alpha=h^*\omega_Y$. Now we will apply the rigidity lemma. Note that since $g$ is a proper birational morphism, the positive dimensional fibers of $g$ are covered by projective curves. Let $C\subset Z$ be a curve such that $g(C)=\pt$. Then by the projection formula $g^*\alpha\cdot C=0$. Thus $0=h^*\omega_Y\cdot C=\omega_Y\cdot h_*C$, and hence $C$ is contracted by $h$, as $\omega_Y$ is a K\"ahler class on $Y$. Therefore, by the rigidity lemma \cite[Lemma 4.1.13]{BS95}, there is a proper morphism $f:X\to Y$ such that $f\circ g=h$, and thus by pushing forward by $g$ it follows that $\alpha=f^*\omega_Y$.

Finally,  if $\omega _X:=\alpha-(K_X+B)$ is a K\"ahler class, then $-(K_X+B)\num_f \omega_X$, and hence $-(K_X+B)$ is $f$-ample, thus $f$ is projective.

\end{proof}

The following variant is often useful in applications.
\begin{corollary}\label{cor:contraction-non-q-factorial}
Let $(X, B)$ be a compact K\"ahler $3$-fold dlt pair, and $\alpha\in N^1(X)$ a nef class such that $\alpha-(K_X+B)$ is a K\"ahler class. Moreover, assume that $B=B_0+B'$, where $B_0\>0, B'\>0$ are $\mbQ$-divisors such that $K_X+B_0$ is $\mbQ$-Cartier and $(X, B_0)$ has klt singularities. Then there exists a proper morphism $f:X\to Y$ with connected fibers to a normal compact K\"ahler variety $Y$ with rational singularities and a K\"ahler class $\omega_Y\in N^1(Y)$ such that $\alpha=f^*\omega_Y$. In particular, if $\alpha-(K_X+B)$ is a K\"ahler class, then $f$ is projective.
\end{corollary}

\begin{proof}
Let $\alpha=K_X+B+\omega$, where $\omega$ is a K\"ahler class. Then for a sufficiently small $\ve\in\mbQ^{\geq 0}$ we can write $\alpha=K_X+B_0+(1-\ve)B'+(\omega+\ve B')$ so that $\omega+\ve B'$ is a K\"ahler class and $(X, B_0+(1-\ve)B')$ is klt. In particular, $\alpha-(K_X+B_0+(1-\ve)B')$ is a K\"ahler class, and thus by Theorem \ref{thm:contraction-non-q-factorial} there exists a projective morphism $f:X\to Y$ such that $\alpha=f^*\omega_Y$ for some K\"ahler class $\omega_Y$ on $Y$.
\end{proof}

\section{Termination of flips for effective pairs}
We will prove termination of flips for effective pairs as in \cite{Bir07}. In order to do this first we prove the existence of dlt models (local and global) and the ACC property for log canonical thresholds. Note that the ACC for for log canonical thresholds is also proved in \cite{Fuj22b}.\\

\begin{theorem}[Global dlt model]\label{thm:global-dlt-model}
	Let $(X, B)$ be a compact K\"ahler lc pair of dimension $4$. Then there exists a $\mbQ$-factorial dlt pair $(X', B')$ and a projective bimeromorphic morphism $g:X'\to X$ such that $K_{X'}+B'=g^*(K_X+B)$.
\end{theorem}

\begin{proof}
	Let $f:Y\to X$ be a log resolution of $(X, B)$.
	Define $B_Y:=f^{-1}_*B+\Ex(g)$. Then using the cone Theorem \ref{thm:general-relative-dlt-cone} we will run a $(K_Y+B_Y)$-MMP over $X$. If $R=\mbR^{\geq 0}\cdot[C_i]$ is a $(K_Y+B_Y)$-negative extremal ray of $\NE(Y/X)$, then from a standard argument using the rationality theorem as in \cite[Theorem 4.11]{Nak87} it follows that there is a $f$-nef line bundle $L$ on $Y$ such that $L-(K_Y+B_Y)$ is $f$-ample and $L^\bot\cap\NE(X/Y)=R$. Write $L=K_Y+B_Y+H$ for some $f$-ample class $H$; then $L=K_Y+(1-\ve B_Y)+(H+\ve B_Y)$  such that $H+\ve B_Y$ is $f$-ample for $\ve\in\mbQ^{\geq 0}$ sufficiently small. Note that $(Y, (1-\ve)B_Y)$ is klt and thus by the Base-Point Free Theorem \cite[Theorem 4.8]{Nak87}, 
  there is a projective bimeromorphic morphism $\phi:Y\to Z$ over $X$ contracting the ray $R$. If $\phi$ is a small morphism, then the existence of the flip follows from Corollary \ref{c-fg1}. Note that from our construction above it follows that at each step $(Y_i, B_{Y_i})$, the contracted locus is contained in $\lrd B_{Y_i}\rrd$. Since the log minimal model program is known in dimension $\<3$ due to \cite{DH20}, special termination holds and the above MMP terminates. Let $g:(X', B')\to (X, B)$ be the end result of this MMP. Then from the negativity lemma it follows that $K_{X'}+B'=g^*(K_X+B)$, and $(X', B')$ is a $\mbQ$-factorial compact K\"ahler dlt pair of dimension $4$. 
\end{proof}~\\

\begin{theorem}[Local dlt-model]\label{thm:local-dlt-model}
Let $(X, B)$ be a log canonical pair, where $X$ is a relatively compact Stein open subset of a K\"ahler variety and there is a compact subset $W\subset X$. Then shrinking $X$ around $W$ if necessary, there exists a projective bimeromorphic morphism $f:Y\to X$ such that $K_Y+B_Y=f^*(K_X+B)$ and $(Y, B_Y)$ is a $\mbQ$-factorial dlt pair, where $B_Y:=f^{-1}_*B+\Ex(f)$.  
\end{theorem}

\begin{proof}
The proof of \cite[Theorem 3.1]{KK10} works here with few changes. Their proof uses ample divisors on $X$ and Bertini's theorem on a resolution of singularities of $X$. The ampleness assumption is replaced by assuming that $X$ is Stein, and the required Bertini theorem follows from Theorem \ref{t-bertini+}. Additionally, \cite[Theorem 3.1]{KK10} uses \cite{BCHM10} to obtain a log terminal model of a klt pair by running a relative MMP for projective morphism. We achieve the same thing here by Theorem \ref{t-mmpscale}.
\end{proof}

\begin{definition}\label{def:lct}
Let $(X, B)$ be a relatively compact log canonical pair and $M\>0$ an $\mbR$-Cartier divisor. Then we define 
\[ 
\lct(X, B; M):=\sup\;\{t\>0\;:\; (X, B+tM) \mbox{ is lc}\}.
\]
Now fix two sets $I\subset [0, 1]$ and $J\subset [0, \infty)$. Let $\mathfrak{I}_n(I)$ be the set of all log canonical pairs $(X, B)$, where $X$ is a relatively compact K\"ahler variety of dimension $n$, and the coefficients of $B$ belong to the set $I$. We define
\[
\LCT_n(I, J):=\{\lct(X, B; M)\; :\; (X, B)\in\mathfrak{I}_n(I)\},
\]
where the coefficients of $M$ belong to the set $J$.
\end{definition}

\begin{theorem}\label{thm:acc-for-lct}
Fix a positive integer $n$, and sets $I\subset [0, 1]$ and $J\subset [0, \infty)$. If $I$ and $J$ are DCC sets, then $\LCT_n(I, J)$ satisfies the ACC.
\end{theorem}

\begin{proof}
By contradiction assume that there is a strictly increasing sequence $\{c_i\}$, where $c_i=\lct(X_i, B_i; M_i)$. Now first assume that there is a component $S_i$ of $M_i$ which is a lc center of $(X_i, B_i+c_iM_i)$ for infinitely many $i$. Let the coefficient of $S_i$ in $B_i$ and $M_i$ be $b_{i}$ and $m_{i}$, respectively. Then we have $b_{i}+c_im_{i}=1$. Since $b_{i}$ and $m_{i}$ are both contained in DCC sets, by passing to a common subsequence we may assume that $b_{i}$ and $m_{i}$ are both monotonically increasing sequences. Then from $c_{i}m_{i}=1-b_{i}$ we see that the LHS is a strictly increasing sequence (since $\{c_i\}$ is strictly increasing) while the RHS is a monotonically decreasing sequence, a contradiction.\\
Thus passing to a tail of the sequence $\{c_i\}$ we may assume that all lc centers of $(X_i, B_i+c_iM_i)$ are contained in the support of $M_i$ are of codimension at least $2$. Let $Z_i$ be a maximal lc center of $(X_i, B_i+c_iM_i)$ contained in $\Supp M_i$ for all $i$.
Next choose a relatively compact Stein open subset  $U_i\subset X_i$ such that  $U_i\cap Z_i\neq\emptyset$ and $Z_i|_{U_i}$ is still a maximal lc center of $(U_i, (B_i+c_iM_i)|_{U_i})$. Replacing $(X_i, B_i+c_iM_i)$ by $(U_i, (B_i+c_iM_i)|_{U_i})$ we may assume that $X_i$ is relatively compact Stein space. Note that shrinking $X_i$ further we can pick a small open subset $V_i\subset X_i$ such that $V_i\cap Z_i\neq\emptyset$ is still a maximal lc center of $(V_i, (B_i+c_iM_i)|_{V_i})$, and additionally $\overline{V}_i\subset X_i$ holds. \\  
Now let $f_i:Y_i\to U_i$ be a dlt model of $(U_i, (B_i+c_iM_i)|_{U_i})$ as in Theorem \ref{thm:local-dlt-model}. Then there is an exceptional divisor $E_i$ intersecting the strict transform of $M_i$ such that $f_i(E_i)=Z_i$. Now write
\[ 
K_{Y_i}+E_i+\Gamma_i=f^*(K_{X_i}+B_i+c_iM_i)
\]
so that $f_{i*}\Gamma_i=B_i+c_iM_i$.\\
Then by adjunction, $(E_i, \Theta_i)$ is a dlt pair, where $K_{E_i}+\Theta_i=(K_X+E_i+\Gamma_i)|_{E_i}=f^*_i((K_{X_i}+B_i+c_iM_i)|_{V_i})$. Note that $\Theta_i$ has a component whose coefficient in $\Theta_i$ is of the form
\begin{equation}\label{eqn:adjunction-coefficient}
\frac{m-1+f+kc_i}{m},
\end{equation}
where $k, m\>1$ and $f\in D(I)$.\\
Now let $F_i$ be a general fiber of the induced morphism $f_{E_i}:=f_i|_{E_i}: E_i\to f_i(E_i)$. Then $F_i$ is projective, since $f_i$ is projective, and by adjunction we have $K_{F_i}+\Theta_{F_i}=(K_{E_i}+\Theta_i)|_{F_i}\num 0$. Note that $\Theta_{F_i}$ has a coefficient of the form \eqref{eqn:adjunction-coefficient}, and thus we arrive at a contradiction by Theorem 1.5 and Lemma 5.2 of \cite{HMX14}.

\end{proof}

\begin{remark}
    We note that Fujino independently proved a slightly more general result on ACC for LCT in \cite[Theorem 1.6]{Fuj22b}.
\end{remark}

\begin{theorem}\cite[Theorem 1.3]{Bir07}\label{thm:effective-termination}
	Let $(X,B)$ be a relatively compact dlt $4$-fold pair such that $(K_{X}+B)\sim_{\mbQ} D\>0$. Then any sequence $\{(X_i, B_i)\}$ of $(K_{X}+B)$-flips where $X_i$ is  K\"ahler for all $i$, terminates. 
\end{theorem}

\begin{proof}
	The same proof as in \cite{Bir07} works here using Theorems \ref{thm:local-dlt-model} and \ref{thm:acc-for-lct}.
\end{proof}

\section{MMP for $\kappa(X, K_X+B)\>0$}
In this section we will prove Theorem \ref{thm:effective-dlt-mm}. First we prove the following easy lemma which will allow us to perturb a nef (but not K\"ahler) class of the form $K_X+B+\omega$, where $\omega$ is a K\"ahler class, such that its null locus intersects $\NA(X)$ precisely along an extremal ray.

\begin{lemma}[General and very general K\"ahler class]\label{lem:very-general-kahler} Let $V$ be a  vector space (resp. a finite dimensional vector space) over $\mbR$ and $\mcC\subset V$ a cone in $V$ which is not contained in any hyperplane.
Let $V^*$ denote the dual space of $V$, and fix a finite collection (resp. a countable collection) of dual vectors $\{C_i\}_{\\i\in I}$ in $V^*$ such that $\omega\cdot C_i:=C_i(\omega)>0$ for all $\omega\in \mcC$. Additionally, assume that $C_i\neq \lambda C_j$ for any $i\neq j\in I$ and $\lambda \in\mbR$. Fix an element $D\in V$. Then there is a finite (resp. countable) union of hyperplanes $\mcH\subset V$ such that if $\omega\in \mcC\setminus \mcH$, then for any $t\in \mbR$, $(D+t\omega)\cdot C_i=0$ for at most one $i\in I$. 
\end{lemma}
\begin{proof}
Since $C_i\neq \lambda C_j$ for any $i\neq j$ and $\lambda\in \mbR$, $\langle C_i, C_j\rangle^\bot$ is a codimension $2$ linear subspace of $V$. Define for $i\neq j$
\[
\mcH(C_i, C_j):=\{\omega\in \mcC\;|\; (D+t\omega)\in \langle C_i, C_j\rangle^\bot \mbox{ for some } t\in\mbR\}.
\]
Then $\mcH(C_i, C_j)$ is contained in some hyperplane. Indeed, if $D+t\omega\in \langle C_i, C_j\rangle^\bot$, then $t\omega\in \langle C_i, C_j\rangle^\bot-D$, and hence $\omega$ is contained in the linear subspace spanned by $\langle C_i, C_j\rangle^\bot-D$, which is contained in a hyperplane.\\
Define $\mcH:=\cup_{i\neq j} \mcH(C_i, C_j)$. Thus $\mcH$ is contained in a finite (resp. countable) union of hyperplanes, and for any $\omega\in\mcC\setminus\mcH$ it follows from our construction above that $D+t\omega\not\in \langle C_i, C_j\rangle^\bot$ for all $i\neq j\in I$ and any $t\in \mbR$; in particular, for any $t\in \mbR$, $(D+t\omega)\cdot C_i\neq 0$ for at most one $i\in I$.  
\end{proof}~\\

In the following we will show that if $(X, B)$ is a dlt pair such that $K_X+B\sim_{\mbQ} M\>0$ and all $(K_X+B)$-negative extremal contractions are contained in the support of $\lrd B\rrd$, then we have a minimal model. 
\begin{theorem}\label{thm:special-effective-dlt-mmp}
Let $(X, S+B)$ be a $\mbQ$-factorial compact K\"ahler dlt pair of dimension $4$ such that $\lrd S+B\rrd=S$ and $(K_X+S+B)\sim_{\mbQ} D\>0$ for some effective $\mbQ$-divisor $D\>0$. Assume that $\Supp (D)\subset S$. Then there exists a finite sequence of flips and divisorial contractions
\[
\xymatrixcolsep{3pc}\xymatrix{\phi:X=X_0\ar@{-->}[r] & X_1\ar@{-->}[r] &\cdots \ar@{-->}[r] & X_n}
\]
such that $K_{X_n}+S_n+B_n$ is nef, where $S_n+B_n=\phi _*(S+B)$.
\end{theorem}

\begin{proof}
If $K_X+S+B$ is nef, then we are done, and so we will assume that $K_X+S+B$ is not nef.  
Note that the set of all curves in $X$ corresponds to countably many classes of curves $\{C_i\}_{i\in I}$ in $N_1(X)$, by \cite[Lemma 4.4]{Toma16}. So we can choose a very general K\"ahler class $\omega\in H^{1,1}_{\BC}(X)$ as in Lemma \ref{lem:very-general-kahler}.
Define 
\[\lambda:=\inf\{t\>0\;|\;K_X+S+B+t\omega \mbox { is K\"ahler}\}.\] 
Replacing $\omega $ by $\lambda\omega$, we may assume that $K_X+S+B+\omega$ is nef but not K\"ahler. Then $K_X+S+B+\omega\num D+\omega$ is a nef and big class but not K\"ahler. We make the following claim.

\begin{claim}\label{clm:passing-to-a-component}
There exists an irreducible component $T$ of $S$ and a curve $\Gamma =C_i$ for some $i\in I$, such that $(K_X+S+B+\omega )\cdot \Gamma=0$ and $T\cdot \Gamma<0$ and
$(K_X+S+B+\omega )\cdot C_{j}>0$ for any $i\ne j\in I$. 
In particular, $K_T+B_T+\omega_T:=(K_X+S+B)|_T+\omega|_T$ is nef but not K\"ahler.
\end{claim}

\begin{proof}[Proof of Claim \ref{clm:passing-to-a-component}]
Since $K_X+S+B+\omega$ is nef and big but not K\"ahler, by Theorem \ref{thm:nef-big-to-kahler} there exists a subvariety $V\subset X$ such that $((K_X+S+B+\omega)|_V)^{\dim V}=0$, i.e. $((D+\omega)|_V)^{\dim V}=0$. By Lemma \ref{lem:arbitrary-nef-and-big}, it follows that $(D+\omega)|_{V^\nu}$ is not a big class on $V^\nu$, where $V^\nu\to V$ is the normalization of $V$. In particular, $V$ is contained in the support of $D$, and hence there is an irreducible component, say $T$ of $S$ such that $V\subset T$. Clearly, $K_T+B_T+\omega_T:=(K_X+S+B)|_T+\omega|_T$ is nef but not K\"ahler. Then by Corollary \ref{cor:nQ-unified-cone}, the $(K_T+B_T)$-negative extremal face $F:=(K_T+B_T+\omega_T)^\bot\cap \NA(T)$ is generated by finitely many curve classes, say $[\Sigma_1],\ldots, [\Sigma_r]$, i.e. $F=\langle\Sigma _{1},\ldots , \Sigma _r\rangle$.

Then $\mathbb R\cdot [\Sigma _{1}]=\mathbb R\cdot [C_i]\subset N_1(X)$ for some $i\in I$.
Since $\omega$ is very general and $K_X+S+B+\omega$ is nef, from Lemma \ref{lem:very-general-kahler} it follows that $(K_X+S+B+\omega)\cdot C_i=0$ and $(K_X+S+B+\omega )\cdot C_{j}>0$ for all $j\ne i\in I$. Therefore $\mathbb R^{\geq 0}\cdot[\Sigma _k] =\mathbb R^{\geq 0}\cdot [C_i]$ in $N_1(X)$ for all $1\leq k\leq r$. Let $\Gamma =C_{i}$. Observe that we have $D\cdot \Gamma<0$, in particular, there is an irreducible component $T'$ of $S$ such that $T'\cdot \Gamma<0$. It then follows that \[(K_{T'}+B_{T'}+\omega|_{T'})\cdot \Gamma=(K_X+S+B+\omega)\cdot \Gamma=(D+\omega)\cdot\Gamma=0,\] and thus $K_{T'}+B_{T'}+\omega|_{T'}$ is nef but not K\"ahler. Thus replacing $T$ by $T'$ we may assume that $(K_T+B_T+\omega|_T)\cdot\Gamma=0$ and $T\cdot\Gamma<0$. In particular, $T\cdot \Sigma_k<0$ and thus $\Sigma_{k}\subset T$ for all $1\<k\<r$.

\end{proof}
By Corollary \ref{cor:contraction-non-q-factorial}, there exists a projective morphism $\vphi:T\to W$ contracting the $(K_T+B_T)$-negative extremal face $F=(K_T+B_T+\omega_T)^\bot\cap\NA(T)$. 
Note that $W$ is a normal compact K\"ahler variety and $\omega_W$ a K\"ahler class on $W$ such that $K_T+B_T+\omega|_T=\vphi^*\omega_W$. 
Also, recall that the face $F$ is generated by the classes of finitely many curves $\Sigma_1,\ldots,\Sigma_r\subset T$ such that $T\cdot\Sigma_i<0$ for all $i=1,\ldots, r$, and a curve $C\subset T$ is contracted by $\vphi$ if and only if its class $[C]\in F$. Thus 
\[
{\rm NE}(T/W)=\NE(T/W)=\left\{\sum_{i=1}^ra_i\Sigma_i\;|\; a_i\>0 \mbox{ for all } i=1, 2,\ldots, r \right\},
\]
and hence from \cite[Proposition 4.7(3)]{Nak87} it follows that $\mcO_T(-mT)$ is $\vphi$-ample, where $m>0$ is the Cartier index of $T$ in $X$.

By \cite[Proposition 7.4]{HP16} or \cite[Theorem 1.1]{DH23}, there exists a proper bimeromorphic morphism $f:X\to Y$ to a normal compact analytic variety $Y$ such that $f|_T=\vphi$ and $f|_{X\setminus T}$ is an isomorphism. From the discussion above it follows that the face $F$ of $\NA(T)$ corresponds to a $(K_X+S+B)$-negative extremal ray $R=\mbR_{\>0}\cdot[\Gamma]$, where $\Gamma=C_i$. Moreover, we know that  $(K_X+S+B+\omega)\cdot \Gamma=0$, and thus $-(K_X+S+B)$ is $f$-nef-big. Then $Y$ has rational singularities by Lemma \ref{lem:rational-singularities}. From \cite[Lemma 3.3]{HP16} it follows that $\rho(X/Y):=\dim_{\mbR}H^{1,1}_{\BC}(X)-\dim_{\mbR}H^{1,1}_{\BC}(Y)=1$. An immediate consequence of this is that $K_X+S+B+\omega=f^*\omega_Y$ for some $(1, 1)$ class $\omega_Y$ on $Y$. 
 Clearly $\omega_Y$ is nef and big. If $V$ is a subvariety of $Y$ of positive dimension, then we claim that $(\omega _Y|_V)^{\dim V}>0$. If $V\subset W$, then let $\lambda=\int _F \omega ^{d}>0$, where {$F$ is a general fiber of $f^{-1}(V)\to V$} and $d=\dim F$. Then by the projection formula
(see eg. \cite[Corollary 4.5]{Nicol})
\begin{equation*}
  \begin{split}
      \lambda\cdot \int _V (\omega _Y)^{\dim V}= \int _{f^{-1}(V)}(f^* \omega _Y)^{\dim V}\wedge\; \omega ^d &= \int _{\varphi ^{-1}(V)}(\varphi ^* \omega _W )^{\dim V}\wedge\; \omega ^d\\
      &= \lambda\cdot \int _V\omega _W ^{\dim V}>0.
  \end{split}  
\end{equation*}

If $V\not \subset W$, then let $V'$ be the strict transform of $V$. If 
 $V'$ is not contained in the support of $D$, then clearly $(K_X+S+B+\omega)|_{V'}=(D+\omega)|_{V'}$ is big (and nef), and so $(\omega _Y|_V)^{\dim V}=(D+\omega)|_{V'}^{\dim V}>0$ by Lemma \ref{lem:arbitrary-nef-and-big}. On the other hand, if $V'$ is contained in a component, say $T'\ne T$, of the support of $D$, then $(K_X+S+B+\omega)|_{T'}=K_{T'}+B_{T'}+\omega _{T'}$, where $(T',B_{T'})$ is dlt, $\omega _{T'}=\omega|_{T'}$
 and   \[K_{T'}+B_{T'}+\omega _{T'}=(K_X+S+B+\omega)|_{T'}=(f^*\omega _Y)|_{T'}=
 (f|_{T'})^*\omega _{W'},\] where $\omega _{W'}=\omega _Y|_{W'}$. By Corollary \ref{cor:contraction-non-q-factorial}, there is a contraction $g:T'\to \bar W$ such that
 $K_{T'}+B_{T'}+\omega _{T'}\equiv g^* \omega _{\bar W}$, where $\omega _{\bar W}$ is a K\"ahler class on $\bar W$. The curves $\Gamma$ contracted by $g$ are precisely the curves in $T'$ such that $(K_X+S+B+\omega) \cdot \Gamma=(K_{T'}+B_{T'}+\omega _{T'}) \cdot \Gamma =0$.
But these are also the curves contracted by $f$ and so by the rigidity lemma (see \cite[Lemma 4.1.13]{BS95}) it follows that $W'=\bar W$. Thus
 \[(\omega _Y|_V)^{\dim V}=((K_X+S+B+\omega)|_{V'})^{\dim V}=((K_{T'}+B_{T'}+\omega _{T'})|_{V'})^{\dim V}=(\omega _{\bar W}|_V)^{\dim V}>0.\]
 Then from Theorem \ref{thm:nef-big-to-kahler} it follows that $\omega_Y$ is a K\"ahler class, and hence $Y$ is a K\"ahler variety.\\
Now if $f$ is a divisorial contraction, then by a similar argument as in the projective case one can show that $Y$ is $\mbQ$-factorial and $(Y, S_Y+B_Y)$ has dlt singularities, where $S_Y+B_Y:=f_*(S+B)$.\\
If $f:X\to Y$ is a flipping contraction, then by Corollary \ref{c-fg1} the flip $f':X'\to Y$ exists, and again as in the algebraic case it follows that $X'$ is $\mbQ$-factorial and $(X', S'+B')$ has dlt singularities, where $S'+B':=\phi_*(S+B)$ and $\phi:X\bir X'$ is the induced bimeromorphic map.\\
Finally, the termination of flips follows from special termination in this case, since all the contracted curves are contained in $\Supp(D)$ and $\Supp(D)\subset S=\lrd S+B\rrd$. Note that the special termination holds here, since MMP in dimension $\<3$ is known due to \cite{DH20}.

\end{proof}

\begin{remark}\label{rmk:cone-theorem}
    The above proof essentially gives a cone theorem in dimension 4 under the given hypothesis. More specifically, with the same hypothesis as in Theorem \ref{thm:special-effective-dlt-mmp}, if $K_X+S+B$ is not nef, then there exists countably many rational curves $\{C_i\}_{i\in I}$ in $X$ such that 
    $0<-(K_X+S+B)\cdot C_i \<6$ and $\NA(X)=\NA(X)_{(K_X+B)\>0}+\sum_{i\in I}\mbR^{\geq 0}\cdot[C_i]$.
\end{remark}

\begin{remark}\label{rmk:choice-of-epsilon}
	Let $(X, \Delta)$ be a $\mbQ$-factorial compact K\"ahler $4$-fold dlt pair and $C$ an effective $\mbQ$-divisor. Fix a positive real number $t>0$ and let $\Lambda$ be the countable set indexing  \textit{all} $(K_X+\Delta)$-negative curve classes $[\Gamma_i]$ on $X$ such that $-(K_X+\Delta)\cdot \Gamma_i\<6$, $\Gamma_i\cdot C>0$ and $(K_X+\Delta+tC)\cdot\Gamma_i>0$. Let $m>0$ be the smallest positive integer such that $m(K_X+\Delta)$ and $mC$ are both Cartier. Then the intersection numbers $(K_X+\Delta)\cdot \Gamma_i$ and $C\cdot\Gamma_i$ are all contained in the set $\frac{1}{m}\mbZ$ for all $i\in \Lambda$. Moreover, since $0<-(K_X+\Delta)\cdot\Gamma_i\leq 6$ for all $i\in \Lambda$, the numbers $(K_X+\Delta)\cdot \Gamma_i$ are contained in a finite set, say $\mcK$. Then $(\mcK+\frac{t}{m}\mbN)\cap\mbR_{>0}$ is a DCC set and hence it has a non-zero minimum, say $\gamma>0$. Thus we can choose a sufficiently small rational number $\eps\in\mbQ^{\geq 0}$ such that
	\begin{equation}\label{eqn:fixing-epsilon}
		0<\eps<\frac{t\gamma}{\gamma+6}.
	\end{equation}  
\end{remark}~\\

The following theorem allows us to run the MMP with scaling in certain cases. This result is in the technical heart of the proof of Theorem \ref{thm:effective-dlt-mm} below.
   \begin{theorem}\label{t-scale}
    Let $(X,\Delta=S+B)$ be a $\mbQ$-factorial compact K\"ahler $4$-fold dlt pair. Assume that there is an effective $\mbQ$-divisor $C\geq 0$ and effective $\mbR$-divisors $D ,D'\geq 0$, and a positive real number $\alpha >0$ such that
    \begin{enumerate}
        \item $K_X + \Delta+C$ is nef,
        \item $K_X + \Delta \sim _{\mathbb R}D$,
        \item $D=\alpha C +D'$, and  ${\rm Supp}(D')\subset S$.
    \end{enumerate}
    Then we can run a $(K_X+\Delta)$-MMP with scaling of $C$ and it terminates with a log terminal model $\phi:X\bir Y$ (see Definition \ref{def:log-terminal-and-log-minimal-model}) such that $K_Y + \phi_*\Delta$ is nef.
    \end{theorem}
    \begin{proof}
    
    Let \[t:={\rm inf}\{s\geq 0\;|\;K_X+\Delta  +sC\ {\rm is \ nef}\}.\] 
    Then $0\leq t\< 1$. Note that if $t=0$, then we are done, otherwise, by Theorem \ref{thm:nef-restricts-to-pseff}, there is a subvariety $V\subset X$ such that $(K_X+\Delta  +(t-\epsilon)C)|_V$ is not pseudo-effective for any $t\geq \epsilon >0$. We fix $\eps$ satisfying the following 
    \begin{equation}\label{eqn:choosing-epsilon}
        0<\eps<\min\left\{t+\alpha, \frac{t\gamma}{\gamma+6}, \frac{1}{6m^2+1}\right\},
    \end{equation}
  where $\gamma \in\mbR^{\geq 0}$ and $m\in\mbZ^{\geq 0}$ are defined in the Remark \ref{rmk:choice-of-epsilon} above.   
     Since 
    \[ K_X+\Delta +(t-\epsilon )C=\frac {t+\alpha-\epsilon }{t+\alpha}( K_X+\Delta  +tC)+\frac {\epsilon }{t+\alpha}(K_X+\Delta -\alpha C)\]
     and $\frac {t+\alpha-\epsilon }{t+\alpha}>0$, it follows that $(K_X+\Delta -\alpha C)|_V\num D'|_V$ is not pseudo-effective. Since the support of $D'$ is contained in $S$, $V$ is contained in an irreducible component, say $T$, of $S$.

Also, note that $t(K_X+\Delta+(t-\epsilon)C)=(t-\epsilon)(K_X+\Delta+tC)+\epsilon(K_X+\Delta)$. Thus $(K_X+\Delta)|_V$ is not pseudo-effective; in particular,  $(K_T+\Delta _T )|_V$ is not pseudo-effective, and hence not nef, where $K_T+\Delta _T:=(K_X+\Delta )|_T$. Let $I$ be the countable set of all $(K_T+\Delta_T)$-negative extremal rays generated by the rational curves $\{\Gamma_i\}_{i\in I}$ as in Corollary \ref{cor:nQ-unified-cone}. We make the following claim.

\begin{claim}\label{clm:zero-at-a-curve}
	$(K_X+\Delta+tC)\cdot \Gamma_i=0$ for some $i\in I$. 
\end{claim}

\begin{proof}[Proof of Claim \ref{clm:zero-at-a-curve}]
To the contrary assume that $(K_X+\Delta +tC)\cdot \Gamma_i>0$ for all $i\in I$. Then we claim that there is a $\delta>0 $ such that $(K_X+\Delta +tC)\cdot \Gamma _i\geq \delta$ for all $i\in I$. To see this, let $m\>1$ be the smallest positive integer such that $m(K_X+\Delta)$ and $mC$ are both Cartier. Then $(K_X+\Delta)\cdot \Gamma_i$ and $C\cdot \Gamma_i$ belong to $\frac 1m \mathbb Z$ for all $i\in I$.
Since $0>(K_X+\Delta )\cdot \Gamma_i=(K_T+\Delta _T)\cdot \Gamma_i\geq -6$ by Corollary \ref{cor:nQ-unified-cone}, the intersection numbers 
$(K_X+\Delta)\cdot \Gamma_i$ are contained in a finite set $\mathcal K\subset\frac{1}{m}\mbZ$.
But then, since $t>0$ is a fixed number, the set $( \mathcal K+\frac tm\cdot \mathbb N)\cap \mathbb R ^{>0}$ is a DCC set and hence has a positive minimum $\delta>0$, i.e. $(K_X+\Delta+tC)\cdot\Gamma_i\>\delta$ for all $i\in I$.\\
By Remark \ref{rmk:choice-of-epsilon} we see that $I\subset \Lambda$, and hence $\delta\>\gamma$. Then from our choice of $\eps>0$ in equation \eqref{eqn:choosing-epsilon}, it follows that 
\[
	0<\eps<\frac{t\gamma}{\gamma+6}\<\frac{t\delta}{\delta+6}.
\]
Thus we have \[(K_T+\Delta _T+(t-\epsilon )C|_T)\cdot \Gamma_i\geq (t-\epsilon)\delta +\epsilon (K_T+\Delta _T)\cdot \Gamma_i\geq (t-\epsilon)\delta -6\epsilon>0\]
for all $i\in I$, and if $\eta\in \overline {\rm NA}(T)_{K_T+\Delta _T\geq 0}$, then 
\[(K_T+\Delta _T+(t-\epsilon )C|_T)\cdot \eta = (t-\epsilon)(K_T+\Delta _T+tC|_T)\cdot \eta +\epsilon (K_T+\Delta _T)\cdot \eta \geq 0.\]
Therefore, by the cone theorem on $T$ (see Corollary \ref{cor:nQ-unified-cone}), $K_T+\Delta _T+(t-\epsilon )C|_T$ intersects every class in $\eta \in \overline {\rm NA}(T)$ non-negatively. Then by \cite[Proposition 3.6]{HP16}, $K_T+\Delta _T+(t-\epsilon )C|_T$ is nef, which is a contradiction.

\end{proof}~\\

Now let $\{R_j\}_{j\in J_T}$ be the countable set of all $(K_T+\Delta_T)$-negative extremal rays which are spanned by rational curves $\Gamma_j\subset T$ as in Corollary \ref{cor:nQ-unified-cone}, for some component $T$ of $S=\lfloor \Delta \rfloor$. Let $J'_T\subset J_T$ be the subset of $(K_T+\Delta_T+tC|_T)$-trivial curves, and $J':=\cup _{T\in S}J'_T$ and $J:=\cup _{T\in S}J_T$, where $T\in S$ means $T$ is a component of $S$. By the above claim, $J'\ne \emptyset$. Let $\bar R_j\in N_1(X)$ be the image of $R_j$ for $j\in J$, and $\mathcal C\subset N_1(X)$ be the cone corresponding to the image of $\sum_{T\in S}{\NA}(T)$. Note that  since ${\NA}(T)={\NA}(T)_{K_T+\Delta _T\geq 0}+\sum _{j\in J_T}R_j$ for each component $T\subset S$, we have $\mathcal C=\mathcal C _{K_X+\Delta \geq  0}+\sum _{j\in J}\bar R_j$. Moreover, $\mathcal C \subset {\NA }(X)$ and 
\[\{ \bar R_j\; |\; j\in J'\}\subset (K_X+\Delta +tC)^\perp\cap {\NA }(X).\]
Let $\omega \in N^1(X)$ be a very general K\"ahler class as in Lemma \ref{lem:very-general-kahler}, and 
\[\lambda:={\rm inf}\{l>0\;|\;(-tC+l\omega ) \cdot \bar R_j\geq 0\ \mbox{ for all } j\in J'\}.\] 
\begin{claim}\label{clm:unique-ray} There is a unique ray $ \bar R_{j'}$ for some $j'\in J'$ such that $(-tC+\lambda\omega)\cdot \bar R_{j'}=0$.\end{claim}
\begin{proof} Since $\omega $ is very general in $N^1(X)$ as in Lemma \ref{lem:very-general-kahler}, it suffices to show that there is one such ray.
By definition of $\lambda$, for each $n\>1$, there is a $j'_n \in J'$ such that $(-tC+(\lambda-1/n)\omega)\cdot\Gamma_{j'_n}<0$, where $\bar{R}_{j'_n}=\mbR^{\geq 0}\cdot[\Gamma_{j'_n}]$. Then we have
\[(K_T+\Delta_T +(\lambda/2)\omega )\cdot \Gamma_{j'_n}= (K_X+\Delta+(\lambda/2)\omega)\cdot\Gamma_{j'_n}=
((\lambda/2)\omega-tC)\cdot \Gamma_{j'_n}<0\]
for all $n>\frac{2}{\lambda}$. By the cone theorem (Corollary \ref{cor:nQ-unified-cone}) there are only finitely many $(K_T+\Delta_T +(\lambda/2)\omega)$-negative extremal rays and so the $\Gamma _{j'_n}$ correspond to only finitely many distinct numerical equivalence classes in $N_1(T)$, and hence in  $N_1(X)$.
Thus, there is a ray  $j'\in J'$ such that  $(-tC+(\lambda-1/n)\omega)\cdot\Gamma_{j'}<0$  for infinitely many $n>0$, and hence $(-tC+\lambda\omega)\cdot\Gamma_{j'}\leq 0$. Then from our construction of $\lambda$ above it follows that $(-tC+\lambda\omega)\cdot\Gamma_{j'} =0$.

\end{proof}
Re-scaling $\omega$, we may assume that $(-tC+\omega )\cdot  \bar R_{j'}=0$ and  $(-tC+\omega )\cdot  \bar R_{j}>0$ for all $R_j\ne R_{j'}$, $j\in J'$. Now recall that $m\>1$ is the smallest positive integer such that $m(K_X+\Delta)$ and $mC$ are both Cartier.
\begin{claim}\label{clm:nef-but-not-kahler} For any $0< \ve \ll 1$, the class
$\alpha_{\ve} :=K_X+\Delta +(1-\ve) tC+\ve \omega \in N^1(X)$ is nef but not K\"ahler.
 \end{claim} 

\begin{proof}

We begin by showing that $ \mathcal C \subset (\alpha_\ve)_{\geq 0}$ or equivalently that $\alpha _\ve |_T$ is nef for all components $T$ of $S$.
Write 
\[\alpha_\ve=(1-\ve)(K_X+\Delta + tC)+\ve (K_X+\Delta +\omega ).\]
It then follows that, if  $ \mathcal C $ is not contained in $ (\alpha_\ve)_{\geq 0}$, then there is a $(K_X+\Delta +\omega)$-negative extremal ray $\bar R_j$ for some $j\in J$ such that $\alpha_\ve\cdot  \bar R_j<0$. 
Note that this set of rays, say indexed by the set $\Lambda$, is a finite set, by Corollary \ref{cor:nQ-unified-cone} (applied on each component $T$ of $S$). So, in particular we may assume that there exists a $\gamma >0$ such that if $j\in\Lambda$ and $(K_X+\Delta + tC)\cdot \bar R_j>0$, then $(K_X+\Delta + tC)\cdot  \Gamma _j>\gamma$, where $\bar R_j=\mathbb R _{\geq 0}[\Gamma _j]$.
But then $\alpha _\ve \cdot  \Gamma _j\geq (1-\ve)\gamma-6\ve>0$ for $\ve <\gamma /(6+\gamma)$, which is a contradiction.
Therefore, we may assume that $(K_X+\Delta + tC)\cdot \bar R_j=0$ for all $j\in\Lambda$, i.e. $\Lambda\subset J'_T$.
But then, by Claim \ref{clm:unique-ray}, $(-tC+\omega)\cdot \bar R_j\geq 0$ for all $j\in\Lambda$, and so
\[\alpha _\ve \cdot \bar R_j=(K_X+\Delta + tC)\cdot \bar R_j+\ve(-tC+\omega)\cdot \bar R_j\geq 0,\] 
this is a contradiction to the fact that $\alpha_{\ve}\cdot\bar R_j<0$ for all $j\in\Lambda$.
Thus $\alpha _\ve |_T$ is nef for all component $T$ of $S$.

Now if $\alpha _\ve$  is not nef on $X$, then by Theorem \ref{thm:nef-restricts-to-pseff} there is a subvariety $V\subset X$ such that $\alpha _\ve |_V$ is not pseudo-effective. Since  $\alpha _\ve=K_X+\Delta +(1-\epsilon)t C+\ve \omega$ and $\omega $ is Kahler,  $(K_X+\Delta +(1-\ve)t C) |_V$ is not pseudo-effective. Observe that
\[
K_X+\Delta+(1-\ve)tC=\frac{(1-\ve)t+\alpha}{t+\alpha}(K_X+\Delta+tC)+\frac{\ve t}{t+\alpha}(K_X+\Delta-\alpha C)
\]
and thus $(K_X+\Delta-\alpha C)|_V\num D'|_V$ is not pseudo-effective. Since $\Supp D'\subset S$, it follows that there is a component $T$ of $S$ such that $V\subset T$. In particular, $\alpha_{\ve}|_T$ is not pseudo-effective; this is a contradiction to the fact that $\alpha_{\ve}|_T$ is nef for all $T$ of $S$ as proved above. 

\end{proof}~\\

From what we have proved above it follows that $\mathcal C \cap (\alpha_\ve)^\bot=\bar R_{j'}$ for a unique $j'\in J'$ as in Claim \ref{clm:unique-ray}. Thus $\bar R_{j'}\subset \alpha_\ve^\bot\cap\NA(X)$. Note that a priori we don't know whether this inclusion is an equality or not. However, we have the following:
\[\alpha:=\frac 1 \ve \alpha _\ve =K_X+\Delta +\omega+\frac{1-\ve}\ve(K_X+\Delta +tC)=K_X+\Delta +\omega_\ve,\]
 where 
\begin{enumerate}
\item $\omega _\ve:=\omega+\frac{1-\ve}\ve(K_X+\Delta +tC) $ is K\"ahler, 
\item $\alpha$ is nef, and 
\item $\alpha ^\perp \cap \mathcal C =\bar R _{j'}\subset \alpha ^\perp \cap  \NA(X)$.
\end{enumerate}

Then we have
\[R_{j'}\subset F:=(\alpha  |_T)^\perp \cap \NA(T)\]
for some component $T$ of $S$.\\
Note that this inclusion could be strict, never the less, from Corollary \ref{cor:nQ-unified-cone} it follows that $F$ is spanned by a finite collection 
of $(K_T+\Delta _T)$-negative extremal rays $\{R_j\}_{j\in J''}$ such that $(K_T+\Delta _T+tC_T)\cdot R_j=0$, i.e. $J''\subset J'$. Note that $R_{j'}$ is one of these extremal rays. By Corollary \ref{cor:contraction-non-q-factorial}, there exists a projective contraction $\varphi:T\to W$ to a normal compact K\"ahler variety $W$ contracting the face $F$ such $\alpha  |_T =\varphi ^*\alpha _W$, where $\alpha _W$ is a K\"ahler class on $W$. Let $R_j$ be generated by the curve $\Sigma_j\subset T$ and $J''=\{1,2,\ldots, r\}$, i.e. $R_j=\mbR^{\geq 0}\cdot[\Sigma_j]$ for all $j=1,2,\ldots, r$. Then by our construction $(K_X+\Delta+\omega_\epsilon)\cdot\Sigma_j=0$ for all $j=1,2,\ldots, r$. Note that $R_{j'}=\mbR^{\geq 0}\cdot[\Sigma_{j'}]$, where  $\Sigma_{j'}=\Sigma_j$ for some $j\in\{1,2,\ldots, r\}$. Let $\bar R_{j'}$ be the image of $R_{j'}$ in $N_1(X)$ and $\bar R_{j'}=\mbR^{\geq 0}\cdot[\Gamma_{j'}]\subset \NA(X)$. Then $\mbR^{\geq 0}\cdot[\Sigma_{j'}]=\mbR^{\geq 0}\cdot[\Gamma_{j'}]$ in $N_1(X)$. Now recall that, since $\omega$ is very general (and hence so is $\omega_\ve$), $(K_X+\Delta+\omega_\ve)\cdot \Gamma_{j'}=0$ and $(K_X+\Delta+\omega_\ve)\cdot\Gamma_j>0$ for all $j\neq j'\in J'$. Therefore 
 \begin{equation}\label{eqn:num-equivalent-on-x}
 	 \mbR^{\geq 0}\cdot [\Sigma_j]=\mbR^{\geq 0}\cdot[\Sigma_{j'}]=\mbR^{\geq 0}\cdot[\Gamma_{j'}] \mbox{ in } N_1(X)\mbox{ for all } j=1,2,\ldots, r. 
 \end{equation}
Next we claim that $\mcO_T(-mT)$ is $\vphi$-ample. First observe that 
\[\mbox{NE}(T/W)=\NE(T/W)=\left\{\sum_{j=1}^r a_j[\Sigma_j]\; |\; a_j\>0 \mbox{ for all }j\right\}.  \]
Therefore by \cite[Proposition 4.7(3)]{Nak87} it is enough to show that $-T\cdot \Sigma_j>0$ for all $j\in J''$. 

Now let $\Sigma\subset T$ be a curve in a fiber of $\vphi$ such that $\Sigma$ is not contained in $\Supp(S-T)$. Then there are real numbers $a_j\geq 0$ for all $j\in J''$  such that $[\Sigma]=\sum _{j\in J''}a_j[\Sigma_j]$ in $N_1(T)$. Now recall that $tC\cdot\Sigma_j=-(K_X+\Delta)\cdot\Sigma_j=-(K_T+\Delta_T)\cdot\Sigma_j>0$, and thus $D'\cdot\Sigma_j<0$  for all $j\in J''$. Write $D'=bT+D''$ such that $b>0$ and $D''$ doesn't contain $T$ as a component. Then $(bT+D'')\cdot \Sigma=\sum _{j\in J''}a_j(D'\cdot \Sigma_j)<0$, and hence $T\cdot \Sigma<0$, since $D''\cdot \Sigma\>0$ by construction of $\Sigma$. But from equation \eqref{eqn:num-equivalent-on-x} it follows that $\mbR^{\geq 0}\cdot[\Sigma_j]=\mbR^{\geq 0}\cdot[\Sigma]$ for all $j=1, 2,\ldots, r$. Hence $T\cdot\Sigma_j<0$ for all $j=1,2,\ldots, r$.\\

Then by \cite[Proposition 7.4]{HP16}  or \cite[Theorem 1.1]{DH23}, $\vphi$ extends to a projective bimeromorphic morphism $\phi:X\to Y$ to a normal compact analytic variety $Y$ such that $\phi|_T=\vphi$. Note that by construction $-(K_X+\Delta)$ is $\phi$-ample. Then from Lemma \ref{lem:rational-singularities} it follows that $Y$ has rational singularities. Consequently, by Lemma \ref{l-HP16} we have $\alpha =\phi^*\omega _Y$ for some $(1, 1)$ class $\omega _Y$ on $Y$. 

 Clearly $\omega_Y$ is nef and big. Following the arguments of Theorem \ref{thm:special-effective-dlt-mmp}, it follows that if $V$ is a subvariety of $Y$ of positive dimension, then $(\omega _Y|_V)^{\dim V}>0$ as long as $V$ is contained in $W$ or in the image of the support of $D'$ or not contained in the image of the support of $D$. 

 Thus, we may assume that $V'$, the strict transform of $V$, is contained in the support of $D$ but not in the support of $D'$. Then we write
 \[\alpha _\epsilon =K_X+\Delta+(1-\epsilon)tC+\epsilon \omega =(1-\lambda)(K_X+\Delta-\alpha C) +\lambda (K_X+\Delta +tC)+\epsilon \omega, \]
 where $\lambda =\frac {(1-\epsilon)t+\alpha}{\alpha +t}$ so that $0<\lambda < 1$.
 Since $(K_X+\Delta-\alpha C)|_V\equiv D'|_{V'}\geq 0$, $(K_X+\Delta +tC)|_{V'}$ is nef and $\omega|_{V'}$ is K\"ahler, then $\alpha _\epsilon |_{V'}$ is big and so $\omega _Y|_V$ is also big.
 
  Then from Theorem \ref{thm:nef-big-to-kahler} it follows that $\omega_Y$ is a K\"ahler class, and hence $Y$ is a K\"ahler variety.
 In particular, $\Null(\alpha)=\Ex(\phi)$. Also, observe that from the discussion above it follows that a curve $C\subset X$ contracted by $\phi$ if and only if $\mbR^{\geq 0}\cdot[C]=\mbR^{\geq 0}\cdot[\Gamma_{j'}]=\bar R_{j'}$ in $N_1(X)$. Thus it follows that $\alpha^\bot\cap\NA(X)=\bar R_{j'}$, and hence from Lemma \ref{l-HP16} again it follows that $\rho(X/Y)=\dim_{\mbR}H^{1,1}_{\BC}(X)-\dim_{\mbR}H^{1,1}_{\BC}(Y)=1$.\\

Now if $\phi:X\to Y$ is a divisorial contraction, then we replace $(X, \Delta)$ by $(Y, \phi_*\Delta)$. Note that $K_Y+\phi_*\Delta+t\phi_*C$ is nef on $Y$. If $\phi$ is flipping contraction, then the flip  $\phi':X'\to Y$ exists by Corollary \ref{c-fg1}. Let $\psi:X\dasharrow X'$ be the induced bimeromorphic map. Then from a standard argument it follows that $(X',\psi_*\Delta)$ is a $\mbQ$-factorial dlt pair, $K_{X'}+\psi_*(\Delta+tC)$ is nef (as $(K_X+\Delta +tC)\cdot R_{j'}=0$), $K_{X'}+\psi_*\Delta\num\psi_*D$ and $\psi_* D=(\alpha/t)\psi_*(tC)+\psi_*D'$, where the support of $\psi_*D'$ is contained in the support of $\psi_* S$.
Therefore, replacing 
\[X,\Delta , S,B,C,D,D',\alpha\qquad {\rm by}\qquad X',\psi_*\Delta ,\psi_*S,\psi_*B,\psi_*(tC),\psi_*D,\psi_*D',\frac \alpha t,\] 
the hypothesis still hold and we may repeat the procedure. In this way we obtain a sequence of $(K_X+\Delta)$-flips and divisorial contractions for the $(K_X+\Delta)$-MMP with scaling of $C$. Since $K_X+\Delta \sim_\mbQ D\geq 0$, this procedure terminates after finitely many steps by Theorem \ref{thm:effective-termination}.\\

    \end{proof}~\\

\begin{lemma}\label{lem:extracting-divisor}
Let $(X, B)$ be a compact K\"ahler lc pair of dimension $4$ and $\{E_i\}_{i\in I}$ a finite set of exceptional divisors over $X$ with $a(E_i, X, B)\<0$ for all $i\in I$. Then there exists a $\mbQ$-factorial dlt pair $(X', B')$ and projective bimeromorphic morphism $f:X'\to X$ such that the following holds:
\begin{enumerate}
    \item $K_{X'}+B'=f^*(K_X+B)$.
    \item Every $E_i$ is an $f$-exceptional divisor, and for an arbitrary $f$-exceptional divisor $F$ either $F=E_i$ for some $i\in I$ or $a(F, X, B)=-1$ holds.
\end{enumerate}
\end{lemma}

\begin{proof}
Let $g:Y\to X$ be a log resolution of $(X, B)$ which extracts all exceptional divisors $\{E_i\}_{i\in I}$. Let $\{F_j\}_{j\in J}$ be the set of all $g$-exceptional divisors. Let $J'\subset J$ such that $\{F_j\}_{j\in J'}=\{E_i\}_{i\in I}$. We define $B_Y:=f^{-1}_*B-\sum_{j\in J'}a(F_j, X, B)F_j+\sum_{j\in J\setminus J'}F_j$. Observe that $B_Y\>0$ is an effective divisor and 
\[
K_Y+B_Y=g^*(K_X+B)+\sum_{j\in J\setminus J'}(1+a(F_j, X, B))F_j.
\]
Now we run a $(K_Y+B_Y)$-MMP over $X$ as in the proof of Theorem \ref{thm:global-dlt-model} and obtain a $\mbQ$-factorial dlt pair $(X', B')$ such that $K_{X'}+B'$ is nef over $X$. Let $f:X'\to X$ be the induced bimeromorphic morphism. Then from the negativity lemma it follows that $K_{X'}+B'=f^*(K_X+B)$.
\end{proof}~\\

\begin{definition}
Let $X$ be a normal variety and $D=\sum a_iD_i$ an $\mbR$-divisor. Then we define $D^{\<1}:=\sum a'_iD_i$, where $a'_i=\min\{a_i, 1\}$. \\
\end{definition}

\begin{proof}[Proof of Theorem \ref{thm:effective-dlt-mm}] We closely follow the proof of \cite[Proposition 3.4]{Bir10} using Theorem \ref{t-scale} as our main technical tool for running the MMP with scaling.\\
Let $(W, \Delta)$ be a log pair, i.e. $\Delta\>0$ is a $\mbQ$-divisor such that $K_W+\Delta$ is $\mbQ$-Cartier. We will call $(W, \Delta)$ an effective pair if there exists an effective $\mbQ$-Cartier divisor $D\>0$ such that $K_W+\Delta\sim_\mbQ D$. We will denote such a pair by the triple $(W, \Delta, D)$. Let $\mcM$ be the collection of all $4$-dimensional triples $(X, B, M)$ such that $(X, B)$ is a $\mbQ$-factorial dlt pair with $(K_X+B)\sim_\mbQ M\>0$ and $(X, B)$ does not admit a log minimal model. Let $\theta(X, B, M)$ be the number of components $P$ of $M$ such that $\mult_P(B)<1$. Pick $(X, B, M)\in\mcM$ such that $\theta(X, B, M)$ is minimal. If $\theta(X, B, M)=0$, then $\Supp M\subset \lrd B\rrd$ and thus by Theorem \ref{thm:special-effective-dlt-mmp}, $(X, B)$ has a log minimal model in fact a log terminal model); hence $(X, B, M)\not\in \mcM$. So assume that $\theta(X, B, M)>0$. Let $f:Y\to X$ be a log  resolution of the pair $(X, B+M)$. Let $E$ be the reduced sum of all exceptional divisors of $f$. Then $(Y, B_Y:=f^{-1}_*B+E)$ is a log smooth dlt pair and 
\[ 
M_Y:=(K_Y+B_Y)-f^*(K_X+B)+f^*M\sim_\mbQ K_Y+B_Y.
\]
Note that $M_Y\>0$ is an effective divisor, since $(X, B)$ is dlt. Moreover, the components of $M_Y$ are either the components of $f^{-1}_*M$ or $f$-exceptional divisors, and 
\begin{equation}\label{eqn:theta}
    \theta(Y, B_Y, M_Y)=\theta(X, B, M).
\end{equation}
Observe that, if $(Y, B_Y)$ has a log minimal model, then $(X, B)$ also has log minimal model (see \cite[Remark 2.6(i)]{Bir10}). Therefore replacing $(X, B, M)$ by $(Y, B_Y, M_Y)$ we may assume that $(X, B+M)$ is a log smooth pair. Define $\alpha>0$ as follows:
\[ 
\alpha:=\min\{t>0\;:\; \lrd (B+tM)^{\<1}\rrd\neq \lrd B\rrd\}.
\]
Note that $\alpha$ is a rational number, since $B$ and $M$ are $\mbQ$-divisors. We can write $(B+\alpha M)^{\<1}=B+C$, where $C$ is an effective $\mbQ$-divisor such that $\Supp\; C\subset \Supp M$. Moreover, we can write $\alpha M=C+M'$ such that $\Supp M'\subset \Supp \lrd B\rrd$, and $C=\alpha M$ outside of $\Supp \lrd B\rrd$. In particular, $\Supp M\subset \Supp (B+C)$.\\
Now observe that we have $(K_X+B+C)\sim_{\mbQ} M+C$ such that $(X, B+C)$ is a log smooth dlt pair and $\theta(X, B+C, M+C)<\theta(X, B, M)$. Therefore by the minimality of $\theta$, it follows that $(X, B+C)$ has a log minimal model, say $(Y, B_Y+C_Y+E)$, where $\phi:X\bir Y$ is the induced bimeromorphic map and $E$ is the sum of all exceptional divisors of $\phi^{-1}$. If $D$ is divisor on $X$, we will denote $\phi_*D$ by $D_Y$ from now on. 
Observe that $(K_Y+B_Y+E)\sim_{\mbQ} M_Y+E$, where $M_Y:=\phi_*M$, since $(K_X+B)\sim_{\mbQ} M$. Moreover, since $\alpha M=C+M'$ on $X$ for some $\mbQ$-divisor $M'\>0$ such that $\Supp M'\subset \lrd B\rrd$, it follows that $M_Y+E=(\frac{1}{\alpha}M'_Y+E)+\frac{1}{\alpha}C_Y$ such that $\Supp (M'_Y+E)\subset \lrd B_Y+E\rrd$. Then the hypothesis of Theorem \ref{t-scale} are satisfied and we can run a $(K_Y+B_Y+E)$-MMP with scaling of $C_Y$. Assume that this MMP terminates with $Y\bir Y'$ such that $K_{Y'}+B_{Y'}+E_{Y'}$ is nef.

Note that this is a nef model of $(X,B)$; however, it is not clear whether it is a log minimal model of $(X, B)$ or not, since the strict inequality $a(P,X,B)< a(P,Y',B_{Y'}+E_{Y'})$ does not necessarily hold for every divisor $P$ on $X$ exceptional over $Y'$. 
Let \[\mathcal T=\{t\in [0,1]\; |\; K_X+B+tC  \mbox{ has a log minimal model}\}.\]
Note that using the minimality of $\theta(X, B, M)$ we have already shown above that $(X, B+C)$ has a log minimal model, i.e. $1\in \mathcal T$. Now our goal is to show that $0\in \mathcal T$. 
For any $0<t\in \mathcal T$, let $\phi_t:X\dasharrow Y_t$ be a log minimal model for $K_X+B+tC$ such that $K_{Y_t}+B_t+E_t+tC_t$ is nef. Proceeding as above, we run a $(K_{Y_t}+B_t+E_t)$-MMP with  scaling of $tC_t$ as in Theorem \ref{t-scale}.
Since  $a(P,X,B+tC)< a(P,Y_{t},B_{t}+tC+E_{t})$ for any divisor $P$ on $X$ exceptional over $Y_{t}$, we also have that $a(P,X,B+t'C)< a(P,Y_{t},B_{t}+t'C+E_{t})$ for any divisor $P$ on $X$ exceptional over $Y_{t}$ and $0\leq t-t'\ll 1$.
But then, this MMP with the scaling of $tC_t$ also yields a log minimal model for $K_X+B+t'C$ for $0\leq t-t'\ll 1$. Thus $[t', t]\subset\mathcal T$.

Let $\tau ={\rm inf}\{t\in \mathcal T \}$. By what we have seen above, if $\tau \in \mathcal T$, then $\tau =0$ and we are done. Suppose therefore that $\tau \not \in \mathcal T$ and $t_k\in \mathcal T$ is a strictly decreasing sequence with $\lim t_k=\tau$; we will derive a contradiction.
For each $k\>1$, let $(Y_{t_k}, B_{t_k}+t_kC_{t_k}+E)$ be a log minimal model of $(X, B+t_kC)$ whose existence is guaranteed by the definition of $\mathcal T$. Then we get a nef model $(Y'_{t_k}, B'_{t_k}+E'_{t_k}+\tau C'_{t_k})$ of $(X, B+\tau C)$ by running a $(K_{Y_{t_k}}+B_{t_k}+E+\tau C_{t_k})$-MMP with the scaling of $(\tau-t_k)C_{t_k}$ as in Theorem \ref{t-scale}.   

Let $D\subset X$ be a divisor contracted by $X\dasharrow Y'_{t_k}$, then by the arguments in Step 5 of the proof of \cite[Proposition 3.4]{Bir10}, we have \[a(D,X,B+t_kC)<a(D,Y'_{t_k},B'_{t_k}+\tau C'_{t_k}+E'_{t_k}).\]
Passing to a subsequence of the $t_k$, we may assume that $X\dasharrow Y'_{t_k}$ contracts a fixed set of components of the support of $B+C$.
By \cite[Claim 3.5]{Bir10} we have that
\[a(D,Y'_{t_k},B'_{t_k}+\tau C'_{t_k}+E'_{t_k})=a(D,Y'_{t_{k+1}},B'_{t_{k+1}}+\tau C'_{t_{k+1}}+E'_{t_{k+1}})\] 
for every divisor $D$ over $Y'_{t_k}$ and for all $k\>1$.
It then follows that 
\[a(D,X,B+\tau C)=\lim a(D,X,B+t_kC) \leq a(D,Y'_{t_k},B'_{t_k}+\tau C'_{t_k}+E'_{t_k}).\] 
This is not yet a log minimal model because we need the inequality to be strict for every divisor $D$ on $X$ exceptional over $Y'_{t_k}$. To remedy this, it suffices to construct a bimeromorphic model $\nu:Y^\sharp\to Y'_k$ which extracts exactly the divisors $D$ on $X$ exceptional over $Y'_{t_k}$ such that 
$a(D,X,B+\tau C)=a(D,Y'_{t_k},B'_{t_k}+\tau C'_{t_k}+E'_{t_k})$ holds. Note that $a(D,X,B+\tau C)\leq 0$ and $(Y'_{t_k},B'_{t_k}+\tau C'_{t_k}+E'_{t_k})$ is lc, so this can be done by Lemma \ref{lem:extracting-divisor}. Let $K_{Y^\sharp}+B_{Y^\sharp}+\tau C_{Y^\sharp}=\nu ^*(K_{Y'_{t_k}}+B'_{t_k}+\tau C'_{t_k}+E'_{t_k})$ such that $\nu_*B_{Y^\sharp}=B'_{t_k}+E'_{t_k}$; then $({Y^\sharp},B_{Y^\sharp}+\tau C_{Y^\sharp})$ is a $\mathbb Q$-factorial dlt pair and $a(D,X,B+\tau C)<a(D,{Y^\sharp},B_{Y^\sharp}+\tau C_{Y^\sharp})$ for every divisor $D$ on $X$ exceptional over $Y^\sharp$. Therefore $X\dasharrow Y^\sharp$ is a log minimal model of $(X,B+\tau C)$. Thus, we have shown that $\tau \in \mathcal T$, which is a contradiction.

\end{proof}~\\

\begin{corollary}\label{cor:klt-ltm}
Let $(X, B)$ be a $\mbQ$-factorial compact K\"ahler plt pair of dimension $4$ such that $\kappa (X, K_X+B)\geq 0$. Then $(X, B)$ has log terminal model.
\end{corollary}

\begin{proof}
This follows from Theorem \ref{thm:effective-dlt-mm} and Lemma \ref{lem:lmm-to-ltm}.
\end{proof}

\section{MMP for Semi-stable pairs}
The main result of this section is Theorem \ref{thm:ss-mmp}. We start with various definitions and establish necessary results first.  
\begin{definition}\label{def:klt-semi-stable-pair}
Let $f:X\to T$ be a proper surjective morphism from a normal K\"ahler variety $X$ to a smooth curve $T$ and $W\subset T$ a compact subset. Let $B\>0$ be an effective $\mbQ$-divisor on $X$. We say that $(X, B/T;W)$ is a \textit{semi-stable klt pair} if $(X,X_w+B)$ is plt  for any $w\in W$. It is well known that this implies (and is in fact equivalent to) the following conditions:
\begin{enumerate}
    \item the fibers $X_w$ of $f$ are all reduced, irreducible and normal,
    \item $\Supp B$ does not contain any fiber $X_w$, and 
    \item $K_X+B$ is $\mbQ$-Cartier and $(X_w, B_w)$ is klt, where $B_w:=B|_{X_w}$.
\end{enumerate}
\end{definition}
By abuse of notation, we will occasionally omit $W$ and simply say that $f:(X, B)\to T$ is a semi-stable klt pair to mean that $(X, B/T; W)$ is a semi-stable klt pair. 
We wish to run a relative MMP for $K_X+B$ over $T$ in a neighborhood of $W$ (so we will repeatedly replace $T$ by an appropriate neighborhood of $W$). We will say that $K_X+B$ is nef over $W$ if $K_{X_w}+B_w=(K_X+B)|_{X_w}$ is nef for every $w\in W$.\\

\begin{definition}\label{def:Neron-Severi-group} 
Let $f:X\to T$ be a proper morphism from a normal analytic variety  $X$ to a smooth curve $T$ such that every fiber of $f$ is an irreducible and reduced normal complex space. Let $W\subset T$ be a fixed compact subset and $U\subset T$ an open neighborhood of $W$.\\ 
If $\tau$ is a real closed bi-dimension $(1, 1)$ current on $X_u$ for some $u\in U$, then for any real closed $(1, 1)$ form $\eta$ on $f^{-1}U $ with local potentials, we define 
\[
\tau (\eta):=(\iota_{u,*}\tau)(\eta)=\tau (\eta|_{X_u}),
\]
where $\iota_u:X_u\injective X$ is the closed embedding.\\
We define $N_1(X/T;W)$ to be the vector space generated by the real closed bi-dimension $(1, 1)$ currents $\tau$ on $X_w$ as $w$ varies in $W$, modulo the following equivalence relation: \[ \tau_1\num \tau_2 \mbox{ if and only if } \tau_1(\alpha)=\tau_2(\alpha) \] for all classes $\alpha\in H^{1,1}_{\rm BC}(X_{U})$, for some open neighborhood $U\subset T$ of $W$ such that $X_U=f^{-1}U\supset f^{-1}W$. 
We define  $\NA(X/T, W)\subset N_1(X/T, W)$ to be the closed cone generated by the classes of closed positive currents. 

We also define $N^1(X_U/U, W)$ as the vector space generated by the classes $\alpha\in H^{1,1}_{\BC}(X_U)$ modulo the following equivalence relation: 
 \[
 \alpha_1\num\alpha_2 \mbox{ if and only if } [\tau](\alpha_1)=[\tau](\alpha_2)
 \]
 for $\tau$  real closed bi-dimension $(1, 1)$ currents on $X_w$ for all $w\in W$.
 Note that if $U\supset U'$ are open subsets containing $W$, then there is a natural restriction map $N^1(X_U/U, W)\to N^1(X_{U'}/U', W)$. Finally let $N^1(X/T, W):=\varinjlim_{W\subset U} N^1(X_U/U, W)$.\\

 We also define $\Pic(X/T, W)$ as the direct limit of $\Pic(f^{-1}U)$, where $W\subset U\subset T$ is an open neighborhood of $W$, i.e.
\[
\Pic(X/T, W):=\varinjlim_{W\subset U} \Pic(f^{-1}U).
\]

\end{definition}

\begin{remark}\label{rmk:infinite-dim}
    We note that $N^1(X/T, W)$ and $N_1(X/T, W)$ could be infinite dimensional vector spaces over $\mbR$, since $X$ and $T$ are not assumed to be compact here.
\end{remark}

\subsection{Relative cone theorem for $4$-folds} We now prove a weak form of the relative cone theorem for proper morphisms $f:X\to T$ from a K\"ahler variety to a curve. We say that a form $\omega$ or a class $\omega\in N^1(X/T; W)$ is relatively nef (resp. relatively K\"ahler) if $\omega _t:=\omega |_{X_t}$ is nef (resp. K\"ahler) for any $t\in T$. 

\begin{lemma}\label{l-douady} 
Let $f:X\to T$ be as above, $\omega$ a relatively K\"ahler  form and $W\subset T$ a compact subset. Fix $M>0$ and let $\{C_i\}_{i\in I}$ be the set of $f$-vertical curves such that $f(C_i)\subset W$ and $\omega \cdot C_i \leq M$,
then the $C_i$ belong to finitely many families of curves.
\end{lemma}
\begin{proof}
Let $\eta$ be a K\"ahler form on $X$. Then for each $t\in W$ there exists an $\epsilon_t>0$ such that $(\omega -\epsilon_tf^*\eta)|_{X_t}$  is a K\"ahler form on $X_t$. It follows that $(\omega -\epsilon_t f^*\eta )|_{X_s}$ is K\"ahler for any $s$ in a neighborhood of $t$. Since $W$ is compact, we may pick an $\epsilon >0$ such that $(\omega -\epsilon f^*\eta )|_{X_t}$ is K\"ahler for every $t$ is a neighborhood of $W$.
Then
\[ 
\eta\cdot C_i<\frac 1 \epsilon \omega \cdot C_i\leq \frac{M}{\epsilon}. 
\]
Since the relative cycle space has only finitely many components of bounded degree with respect to $\eta$, it follows that the curves $C_i$ belong to finitely many families (see for example \cite[Theorem 5.5]{Tom21} for further details).

\end{proof}

The following result gives a weak form of relative cone theorem for semi-stable klt pairs.
\begin{theorem}\label{thm:weak-cone0} Let $f:X\to T$ be a proper surjective morphism from a K\"ahler $4$-fold $X$ to a  curve $T$ such that $f_*\mcO_X=\mcO_T$. Let $W\subset T$ be a compact subset and $(X,B/T; W)$ is a semi-stable klt pair. Fix a K\"ahler form $\omega$ on $X$.
Then there are finitely many classes of curves  $\{C_i\}_{i\in J}$ ($J$ is a finite set) over $W$ such that $0> (K_X+B)\cdot C_i\geq -6$ and for each $t\in W$
\[\overline{\rm NA}(X_t)=\overline{\rm NA}(X_t)_{(K_{X_t}+B_t+\omega_t)\geq 0}+\sum _{i\in J}\mathbb R ^{\geq 0}[C_i].\]
Suppose now that $K_{X_t}+B_t+\omega _t$ is nef for all $t\in W$, where $\omega _t:=\omega |_{X_t}$ is K\"ahler for all $t\in W$. Let \[\lambda :=\inf\{s\geq 0\;|\; K_{X_t}+B_t+s\omega _t \mbox{ is nef for all } t\in W\}.\]
If $\lambda >0$, then there are finitely many classes of curves  $\{C_i\}_{i\in I}$ ($I\subset J$) over $W$ which satisfy the following properties:
\begin{enumerate}
    \item $C_i\subset X_t$ for some $t\in W$, and $\mathbb R ^{\geq 0}[C_i]$ is a $(K_{X_{t}}+B_{t})$-negative extremal ray of $\NA(X_t)$ such that $(K_{X_{t}}+B_{t}+\lambda \omega _{t})\cdot C_i=0$,  
    \item  if $C\subset X_t$ is a curve such that $(K_{X_t}+B_t+\lambda \omega _t)\cdot C=0$ for some $t\in W$, then $[C]\equiv \sum_{i\in I} c_i[C_i]$ in $N_1(X/T,W)$ for some $c_i\in \mathbb R^{\geq 0}$,
    
    \item if $\omega \in N^1(X/T,W)$ is general, then $|I|=1$ (i.e. we may assume that there is a unique such class $[C_i]\in N_1(X/T,W)$). 
\end{enumerate}

\end{theorem}
\begin{proof}
By Corollary \ref{cor:nQ-unified-cone}, for any $t\in T$ there are finitely many $(K_{X_t}+B_t+\omega_t)$-negative extremal rays $C_i$ where $i\in J_t$ and $0> (K_{X_t}+B_t) \cdot C_i=(K_X+B)\cdot C_i \geq -6$.
Let $J=\cup _{t\in T}J_t$. Since $\omega \cdot C_i=\omega _t\cdot C_i<- (K_{X_t}+B_t)\cdot C_i\leq  6$, it follows from  Lemma \ref{l-douady} that $J$ is finite. The first statement is proven.

Suppose now that $K_{X_t}+B_t+\omega _t$ is nef for all $t\in W$. Define the set \[\Lambda:=\{t\in W\; |\; K_{X_t}+B_t+\lambda\omega_t \mbox{ is nef but not K\"ahler} \}\subset T.\]
Then $\Lambda\ne \emptyset$, as otherwise arguing as in the proof of Lemma \ref{l-douady} above, one sees that $K_X+B+\lambda \omega $ is relatively K\"ahler over a neighborhood of $W$, which contradicts the definition of $\lambda$.
For any $t\in \Lambda$, we have $F_t:=(K_{X_{t}}+B_{t}+\lambda\omega_{t})^\bot\cap\NA(X_{t})\neq \{0\}$ by \cite[Corollary 3.16]{HP16}. Moreover, from Corollary \ref{cor:nQ-unified-cone} it follows that $F_t$ is generated by finitely many classes of curves, each of which generates a $(K_{X_{t}}+B_{t})$-negative extremal ray. 
Let $\Gamma:=\{C\subset X_t\;|\; t\in \Lambda,\ C \mbox{ generates a } (K_{X_t}+B_t)\mbox{-negative extremal ray such that } (K_{X_t}+B_t+\lambda\omega_t)\cdot C=0 \}$.
Then for a curve $C\in \Gamma$ we have $C\subset X_t$ for some $t\in\Lambda$, and 
\[ 
\omega\cdot C=\omega_{t}\cdot C=\frac{-1}{\lambda}(K_{X_{t}}+B_{t})\cdot C\<\frac{6}{\lambda}. 
\]
By Lemma \ref{l-douady} the curves in $\Gamma$ belong to finitely many families, and hence correspond to finitely many numerical classes.  This proves (1).\\

For (2), let $C\subset X_t$ be a curve such that $(K_{X_t}+B_t+\lambda\omega_t)\cdot C=0$. Then $[C]\in F_t$, and by Corollary \ref{cor:nQ-unified-cone} and Part (1) above it follows that there is a subset $J\subset I$ such that $F_t$ is generated by the curves $C_j$ for $j\in J$. In particular, $[C]=\sum c_i[C_i]$ in $H^{1,1}_{\BC}(X_t)$ for some $c_i\in\mbR_{\geq 0}$, and hence also in $H^{1,1}_{\BC}(X)$.\\

(3) now follows from Lemma \ref{lem:very-general-kahler}.

\end{proof}

\begin{definition}
We say that $(X,B)$ is a minimal model over $W$ if $K_X+B$ is nef over $W$. If, possibly replacing $T$ by an appropriate neighborhood of $W$, there is a morphism $g:X\to Z$ over $T$ such that $\dim X>\dim Z$ and $-(K_X+B)$ is ample on each fiber of $g$, then we say that $g$ is a Mori fiber space over $W$. 
We say that $(X/T;W)$ is $\mathbb Q$-factorial if: (i) every Weil divisor $D$ defined over a neighborhood of $W$ is $\mbQ$-Cartier over a (possibly smaller) neighborhood of $W$, and (ii) $(\omega_X^{\otimes m})^{**}$ is a line bundle over a neighborhood of $W$ for some $m\>1$.
\end{definition}~\\

We will use the following variant of \cite[Lemma 3.3]{HP16}. The main point here is that $X$ and $Y$ are not assumed to be compact. The proof is similar to that of \cite{HP16}, however, we reproduce it here for the convenience of the reader.   
\begin{lemma}\cite[Lemma 3.3]{HP16}\label{l-HP16}
Let $f:X\to Y$ be a proper birational map between normal complex spaces in Fujiki's class $\mathcal C$ with rational singularities. 
Then we have an injection
\[f^*:H^{1,1}_{\rm BC}(Y)=H^1(Y,\mathcal H_Y)\hookrightarrow H^1(X,\mathcal H_X) =H^{1,1}_{\rm BC}(X)\] 
such that ${\rm Im}(f^*)=\{\alpha \in H^1(X,\mathcal H_X)\;|\; \alpha\cdot C=0 \mbox{ for all curves } C\subset X \mbox{ s.t. } f(C)=\pt\}$.
\end{lemma}
\begin{proof} Note that we are not assuming that $X,Y$ are compact and so it is not clear that $H^1(X,\mathcal H_X)\to H^2(X,\mathbb R)$ and $H^1(Y,\mathcal H_Y)\to H^2(Y,\mathbb R)$ are injective. However, we still have a commutative diagram similar to \cite[Eqn. (5), page 224]{HP16}: 
\begin{equation}\label{eqn:vertical-class}
    \xymatrixcolsep{3pc}\xymatrixrowsep{3pc}\xymatrix{
    0\ar[r] & H^1(Y, \mcH_Y)\ar[r]\ar[d] & H^1(X, \mcH_X)\ar[r]^{\varphi}\ar[d]^{\psi} & H^0(Y, R^1f_*\mcH_X)\ar[d]^{\cong}\\
    0\ar[r] & H^1(Y, \mbR)\ar[r] & H^1(X, \mbR)\ar[r]^{\varphi'} & H^0(Y, R^2f_*\mbR)
    }
\end{equation}
Suppose now that $\alpha \in H^1(X,\mathcal H_X)$ such that $\alpha \cdot C=0$ for all curves $C\subset X$ such that $f(C)=\pt$. Then from the claim $(\star)$ in the proof of \cite[Thm. 12.1.3, page 649]{KM92} it follows that $(\varphi'\circ\psi)(\alpha)=0$. Therefore from the diagram above it follows that there exists a $\beta\in H^1(Y, \mcH_Y)$ such that $\alpha=f^*\beta$.

\end{proof}~\\

\begin{lemma}\label{lem:rational-singularities}
Let $f:X\to Y$ be a proper morphism of normal analytic varieties and $f_*\mcO_X=\mcO_Y$. Let $B\>0$ be an effective $\mbQ$-divisor such that $K_X+B$ is $\mbQ$-Cartier. Assume that one of the following conditions hold:
\begin{enumerate}
    \item[(i)] $(X, B)$ is klt and $-(K_X+B)$ is $f$-nef-big.
    \item[(ii)] $(X, B)$ is dlt, $K_X$ is $\mbQ$-Cartier and $-(K_X+B)$ is $f$-ample. 
\end{enumerate}
Then $Y$ has rational singularities.
\end{lemma}

\begin{proof}
It follows from \cite[Lemma 2.44]{DH20} and its proof.
\end{proof}~\\

\begin{proposition}\label{pro:extremal-ray-contraction}
Let $(X, B/T; W)$ be a $\mbQ$-factorial semi-stable klt pair of dimension $4$. Let $R=\mbR^{\geq 0}\cdot[\Gamma]$ be a $(K_X+B)$-negative extremal ray of $\NA(X,B/T; W)$ generated by a curve $\Gamma\subset X$. Assume that contraction of $R$ exists, i.e. there is an open neighborhood $U$ of $W$ and a projective morphsim $g:f^{-1}U\to Z$ over $U$ such that a (compact) curve $C\subset f^{-1}U$ which maps to a point $f(C)\in W$ is contracted by $g$ if and only if $[C]\in R$. Let $h:Z\to U$ be the induced morphism. Then the following hold:
\begin{enumerate}
    \item We have the following exact sequences:
    \begin{equation}\label{eqn:N^1}
    \xymatrixcolsep{3pc}\xymatrix{0\ar[r] & N^1(Z/U, W)\ar[r] & N^1(f^{-1}U/U, W)\ar[r]^-{\alpha\mapsto \alpha\cdot\Gamma} & \mbR\ar[r] & 0}
    \end{equation}
    \begin{center}
        and
    \end{center}
    \begin{equation}\label{eqn:Pic}
        \xymatrixcolsep{3pc}\xymatrix{0\ar[r] & \Pic(Z/U, W)\ar[r] & \Pic(f^{-1}U/U, W)\ar[r]^-{L\mapsto L\cdot\Gamma} & \mbZ\ar[r] & 0. }
    \end{equation}
    
    \item If $g$ is a divisorial contraction, then $(Z/U, W)$ is $\mbQ$-factorial semi-stable klt pair.
    \item If $g$ is a flipping contraction with flip $g':V\to Z$, then $(V/U, W)$ is $\mbQ$-factorial semi-stable klt pair.  
\end{enumerate}
\end{proposition}

\begin{proof}
First note that using \cite[Proposition 1.4]{Nak87} we may replace $U$ by a smaller open neighborhood of $W$ and assume that $-(K_X+B)|_{f^{-1}U}$ is $g$-ample. Thus by Lemma \ref{lem:rational-singularities}, $Z$ has rational singularities. The exactness of the sequence \eqref{eqn:N^1} follows from Lemma \ref{l-HP16}. Next, let $L$ be a line bundle on $f^{-1}U$ such that $L\cdot \Gamma=0$.  Then $L\cdot C=0$ for all curves in the fibers of $g$; in particular, $L|_{g^{-1}(z)}$ is nef for all $z\in h^{-1}W$. Then $(L-(K_X+B))|_{g^{-1}(z)}$ is ample for all $z\in h^{-1}W$. Then again from \cite[Proposition 1.4]{Nak87} it follows that $L-(K_X+B)$ is $g$-ample over a neighborhood of $h^{-1}W$. Since $h:Z\to U$ is proper and flat (as $U$ is a smooth curve), and hence is both an open and closed morphism, shrinking $U$ suitably near $W$ we may assume that $L-(K_X+B)$ is $g$-ample.\\
Next, for the exactness of the sequence \eqref{eqn:Pic}, observe that if $L\cdot\Gamma=0$, then we need to show that $L\cong g^*M$ for some line bundle $M$ on $Z$. Since $g_*L$ is unique, it is enough to show locally on $Z$ that $g_*L$ is a line bundle and $L\cong g^*M_Z$ locally over $Z$, where $M_Z$ is a line bundle on an appropriate open subset of $Z$. So we may assume that $Z$ is Stein. Then $L$ is given by a Cartier divisor (since $g$ is projective), and hence by the base-point free theorem as in \cite[Theorem 4.8]{Nak87} and the rigidity lemma \cite[Lemma 4.1.13]{BS95} it follows that $L\cong g^*M_Z$ for some line bundle $M_Z$ on $Z$. Then by the projection formula, $g_*L\cong M$ is a line bundle, as required. This shows the exactness of the sequence \eqref{eqn:Pic}.\\
Now assume that $g$ is a divisorial contraction. Then from a standard argument using \eqref{eqn:Pic} it follows that $(Z/U, W)$ is $\mbQ$-factorial. 
Since $(X, X_w+B)$ has plt singularities for any $w\in W$ and $R=\mbR^{\geq 0}\cdot[\Gamma]$ is also a $(K_X+X_w+B)$-negative extremal ray, {then by the proof of} \cite[Corollary 3.44]{KM98} $(Z, Z_w+B_Z)$ has plt singularities, where $Z_w+B_Z=g_*(X_w+B)$. Thus $(Z, B_Z/U; W)$ is a semi-stable klt pair.\\
If $g$ is a flipping contraction, let $g': V\to Z$ be the flip. Then again from a standard argument it follows that $(V, B'/U; W)$ is a $\mbQ$-factorial semi-stable klt pair, where $B':=\phi_*B$ and $\phi:f^{-1}U\to V$ is the induced bimeromorphic morphism.
\end{proof}~\\

\begin{lemma}\label{lem:n1-pushforward}
Let $(X, B/T; W)$ be a $\mbQ$-factorial semi-stable klt pair of dimension $4$, and $\phi:X\dasharrow X'$ is either a $(K_X+B)$-flip or a divisorial contraction over $T$. Then $\phi_*:N^1(X/T;W)\to N^1(X'/T;W)$ is well defined and surjective.
\end{lemma}

\begin{proof} We omit this proof as it follows from the standard well known arguments of the corresponding compact and projective results.

\end{proof}~\\

\begin{lemma}\label{l-psef}
Let $(X,B/T;W)$ be a semistable klt pair. Then the following are equivalent.
\begin{enumerate}
\item $\kappa (K_{X_t}+B_t)\geq 0$ for all $t\in W$. 
\item $\kappa (K_{X_t}+B_t)\geq 0$ for very general $t\in W$. 
\item $W\subset {\rm Supp}f_*\mathcal O _X(m(K_X+B))$ for some $m>0$.
\item For every positive constant $\mu >0$, $K_{X_t}+B_t+\mu \omega _t$ is pseudo-effective for very general $t\in W$.
\end{enumerate}
\end{lemma}
\begin{proof}
(1) clearly implies (2).

(2) implies (3). Since the supports of $f_*\mathcal O _X(m(K_X+B))$ are closed subsets of $W$ it suffices to show that for a very general point $w\in W$ there is an integer $m>0$ such that 
$w\in {\rm Supp}f_*\mathcal O _X(m(K_X+B))$. Assume that $\kappa (K_{X_t}+B_t)\geq 0$ for very general $t\in W$. Let $T'\subset W$ be the set of $t\in W$ for which $\kappa (K_{X_t}+B_t)\geq 0$ and for each $t\in T'$, let $m(t)>0$ be the smallest positive integer such that $m(t)(K_{X_t}+B_t)$ is Cartier and $H^0(X_t, m(t)(K_{X_t}+B_t))\neq 0$. Let $m(t)(K_{X_t}+B_t)\sim M(t)\>0$ for some effective Cartier divisor $M(t)$ for any $t\in T'$. Since $T'$ is a complement of countably many analytic subsets, it follows that for any $w\in W$ there is a subset $T''\subset T'$ such that $w$ is an accumulation point of $T''$ and  $m(K_{X_t}+B_t)\sim M(t)$ for all $t\in T''$ for some positive integer $m$ independent of $t\in T''$. Therefore, from Grauert's theorem (see \cite[Theorem III.4.7]{GPR94}) it follows that $w\in {\rm Supp}f_* \mathcal O _X(m(K_X+B))\ne 0$. This concludes the proof that (2) implies (3).

(3) implies (1). Suppose that $W\subset {\rm Supp}f_*\mathcal O _X(m(K_X+B))$, then for any $t\in W$, there is an open subset  $t\in V\subset T$ and an effective divisor $D_V$ on $X_V$ such that $m(K_V+B_V)\sim _V D_V$. Discarding vertical components of $B_V$ we may assume that $D_V$ contains no fibers and hence we have $m(K_{X_t}+B_t)\sim D_t:=D_V|_{X_t}$ for every $t\in V$ and $\kappa (K_{X_t}+B_t)\geq 0$ for all $t\in W$.\\

(2) clearly implies (4) and hence it suffices to show that (4) implies (2).

So, suppose that for every $\mu >0$, $K_{X_t}+B_t+\mu \omega _t$ is pseudo-effective for very general $t\in W$.
Let $W_k=\{t\in W\;|\; K_{X_t}+B_t+\frac 1k \omega _t\ {\rm is\ pseudo-effective}\}$, then $W_k$ contains the complement of countably many points and hence so does $W_\infty =\cap _{k\geq 0}W_k$.
But then $K_{X_t}+B_t$ is pseudo-effective for any $t\in W_\infty$. By \cite[Theorem 1.1]{DO23}, $\kappa (K_{X_t}+B_t)\geq 0$.
\end{proof}~\\

Now we are ready to prove the existence of minimal models for a semi-stable klt pairs $(X, B/T;W)$ when $K_X+B$ is effective over $W$ i.e. when any of the equivalent conditions of Lemma \ref{l-psef} hold.
\begin{theorem}\label{thm:ss-pseff-mmp}
Let $f:(X,B)\to T$ be a  
semi-stable klt pair of dimension $4$ and $W\subset T$ a compact subset. If $(X/T;W)$ is $\mbQ$-factorial and $K_X+B$ is effective over $W$, then we can run the $(K_X+B)$-MMP over a neighborhood of $W$ in $T$ which ends with a minimal model over $W$.
\end{theorem}
\begin{proof}
Suppose that $K_X+B$ is not nef over $W$. Choose a K\"ahler class $\omega$ on $X$ such that $K_{X_t}+B_t+\omega_t$ is nef for all $t\in W$, where $\omega_t:=\omega|_{X_t}$. 
We may assume that $\omega$ is {general} in $N^1(X/T;W)$. 
Let \[\lambda:=\inf\{s\>0\;|\; K_{X}+B+s\omega \mbox{ is nef over } W\},\] then we have the following. By Theorem \ref{thm:weak-cone0}, there exists a $(K_{X_t}+B_t)$-negative extremal ray $R_t=\mathbb R ^{\geq 0}[C]\subset N_1(X_t)$ on $X_t$ for some $t\in W$ such that $(K_{X_t}+B_t+\lambda \omega _t)\cdot C=0$ and if $C'\subset X_{t'}$ is a  $(K_{X}+B+\lambda\omega)$-trivial curve for some $t'\in W$, then $[C']\in R:=\mathbb R ^{\geq 0}[C]\subset \NA(X/T;W)$.

 Replacing $\lambda \omega $ by $\omega$, we may assume that $K_{X_t}+B_t+ \omega_t$ is nef for all $t\in W$ and $K_X+B+\omega$ supports the extremal ray $R\subset N_1(X/T;W)$.  
 Note that $K_X+B+ \omega$ may cut out $(K_{X_t}+B_t)$-negative faces $F_t$ from multiple or even all fibers $X_t$ with $t\in W$. 
By Theorem \ref{thm:contraction-non-q-factorial}, there is an extremal contraction $g_t:X_t\to Z_t$ for the face $F_t\subset \NA(X_t)$.
By \cite[Proposition 11.4]{KM92}, this extends to a contraction $g:X_U\to Z_U$ over a neighborhood $U$ of $t\in T$, where $X_U=X\times _TU$ (we note that $X_t,Z_t$ are compact, $g_{t,*}\OO _{X_t}=\OO _{Z_t}$ and $R^1g_{t,*}\OO _{X_t}=0$  by the relative Kawamata-Viehweg vanishing Theorem \ref{thm:relative-kvv}, as $-(K_{X_t}+B_t)$ is $g_t$-ample). Note that $X_U\to Z_U$ is a surjective morphism of normal varieties with connected fibers which contracts precisely the set of curves $C\subset X_t$ for some $t\in U$ such that $[C]\in R\subset N_1(X/T;W)$. Suppose that $U,U'\subset T$ are two such open subsets, then over $U\cap U'$, $X_U\to Z_U$ and $X_{U'}\to Z_{U'}$ are isomorphic, since they are both surjective morphisms of normal varieties with connected fibers which contract identical subsets (see the rigidity lemma in \cite[Lemma 4.1.13]{BS95}). 
Thus these contractions glue together to give a projective contraction $g:X\to Z$ over $T$.
Note that if $\dim Z_t<\dim X_t$ for some $t\in T$, then from the flatness over $T$ it follows that $\dim Z<\dim X$, which is impossible as $K_X+B$ is pseudo-effective. In particular, $g$ is bimeromorphic. 
If $g$ is a divisorial contraction, then we replace $X$ with $Z$ and $B$ with $g_*B$. If $g$ is a flipping contraction, then flip $g^+:X^+\to Z$ exists by Corollary \ref{c-fg1}. Then we replace $X$ by the flip $X^+$.

Note that by construction, for every $t\in W$ we have that $K_{X_t}+B_t+\omega _t$ is nef, and by Corollary \ref{cor:contraction-non-q-factorial}, $K_{X_t}+B_t+\omega _t=g_t^*\omega _{Z_t}$ for some K\"ahler form $\omega _{Z_t}$ on $Z_t$. Since $-(K_X+B)$ is $g$-nef-big, by Proposition \ref{lem:rational-singularities}, $Z$ has rational singularities. By Lemma \ref{l-HP16}, $K_X+B+\omega =g^*\alpha$ for some form $\alpha$ on $Z$. 
Since $(g^*\alpha)|_{X_t}=g_t^*\omega _{Z_t}$ for every $t\in W$, it follows that $\alpha $ is K\"ahler over $W$ (see eg. the proof of Theorem \ref{thm:special-effective-dlt-mmp}).\\
If $X\to Z$ is a flipping contraction, then since $K_{X^+}+B^+$ is ample over $Z$, it follows that $(g^+)^*\alpha +\epsilon(K_{X^+}+B^+)$ is K\"ahler over $W$.\\

Termination of flips follows from Theorem \ref{thm:effective-termination}, however termination of divisorial contractions is not immediately clear as $N^1(X/T;W)$ may be infinite dimensional. But observe that if $X\to Z$ is a divisorial contraction, then the exceptional divisor $E$ dominates $T$ and so $\rho (X_t)>\rho (Z_t)$ for general $t\in T$. Therefore there are no infinite sequences of divisorial contractions.
\end{proof}~\\

Next we prove the existence of Mori fiber space when $K_X+B$ is not effective over $W$.
\begin{theorem}\label{thm:ss-non-pseff-mmp}
Let $(X,B/T;W)$ be a $\mbQ$-factorial semi-stable klt pair of dimension $4$, where $W\subset T$ is a compact subset. If $K_X+B$ is not effective over $W$ (see Lemma \ref{l-psef}), then we can run a $(K_X+B)$-MMP over a neighborhood of $W$ which ends with a Mori fiber space.
\end{theorem}
\begin{proof}
Throughout the proof we will repeatedly shrink $T$ in a neighborhood of $W$ without further mention.
	
The existence of flips and divisorial contractions here works exactly as in Theorem \ref{thm:ss-pseff-mmp}, and so we will only discuss the termination of flips below.

To see termination, we proceed as follows. First by inversion of adjunction, $(X,X_t+B)$ is dlt for any $t\in T$. Moreover, it is easy to see that any $(K_X+B)$-MMP over $T$ is also a $(K_X+X_t+B)$-MMP over $T$ for a fixed $t\in T$, and thus by special termination the flipping locus is disjoint from $X_t$ after a finitely many steps. Note also that any divisorial contraction must induce a nontrivial morphism on $X_t$ for general $t\in T$, and hence decreases its Picard number $\rho (X_t)$. Therefore, we may assume that there are no divisorial contractions after finitely many steps of this minimal model program.
We fix a point $t_0\in T$, and from now on we will assume that any $(K_X+B)$-MMP over $T$ is disjoint from the fixed fiber $X_{t_0}$; regardless of what MMP we run. In particular, the flipping loci do not dominate the base curve $T$, and hence the flipping curves for any given flip are contained in finitely many fibers of $f$. 
 Since there are at most countably many flips for any given $(K_X+B)$-MMP over $T$, it follows that, for very general $t\in T$, any finite sequence of steps of a $(K_X+B)$-MMP over $T$ will induce an isomorphism on a neighborhood of $X_t$.

By contradiction assume that flips do not terminate for any $(K_X+B)$-MMP over $T$. Let $\omega$ be a  K\"ahler class on $X$ such that $K_X+B+\omega$ is K\"ahler over $W$. Now we will discuss the strategy our proof first without full technical details. The idea is as follows. We run a minimal model program with the scaling of $\omega$: $X=X^1\dasharrow X^2\dasharrow \ldots \dasharrow X^n$. As we have observed above, this MMP is disjoint from a very general fiber $X_s$ and from any fiber $X_t$ for $n\gg 0$. It follows that there is a sequence of fibers $X_{t_i}\cong X^i_{t_i}$ containing a flipping curve for $X^i\dasharrow X^{i+1}$. Let $C_i\subset X_{t_i}$ be a curve whose isomorphic image in $X^i_{t_i}$ is a flipping curve of $X^i\bir X^{i+1}$; we will identify $C_i$ with its image in $X^i_{t_i}$. Suppose that $(K_{X^i}+B^i+\lambda _i\omega ^i)\cdot C_i=0$, where $\lambda _1\geq \lambda _2\geq \ldots $ are the nef thresholds. By Lemma \ref{l-psef} $\lim \lambda _i=\mu >0$, as $K_X+B$ is not effective over $W$, and so
\[\omega \cdot C_i=\omega ^i\cdot C_i=\frac{-1}{\lambda _i}(K_{X^i_{t_i}}+B^i_{t_i})\cdot C_i\leq \frac{6}{\mu}.\] By Lemma \ref{l-douady}, these $C_i\subset X$ belong to finitely many families and so must be contained in finitely many fibers. This is a contradiction, and hence the sequence of flips terminates. Unluckily, there are several technical issues that arise in the proof. Since we do not have a cone theorem here, it is not clear whether for each $i$ there is a \textit{unique} $(K_{X^i}+B^i)$-negative extremal ray $R_i$ of $\NA(X/T; W)$ such that $(K_{X^i}+B^i+\lambda_i\omega^i)\cdot R_i=0$. 

However, this can be achieved as long as each $\omega ^i$ is general in $N^1(X/T; W)$, and so at each step it suffices to perturb the given K\"ahler class. Thus we end up with a sequence of K\"ahler classes $\omega_{i+1} =\omega _i+\epsilon _i\alpha _i$ such that $\alpha _i$ is general in $N^1(X/T; W)$ and $0<\epsilon _i\ll 1$. This is discussed in detail below.\\

As mentioned above, we will run a $(K_X+B)$-MMP over $W$ with scaling of a sequence of general K\"ahler classes $\omega_i$. This means that:
{\it  There exists a sequence $X=X^1\dasharrow X^2\dasharrow \cdots \dasharrow X^n$ of $(K_X+B)$-flips and divisorial contractions over $W$ and real numbers $\lambda _1> \lambda _2> \cdots > \lambda _n>0$ satisfying the following properties:
\begin{enumerate}
    \item $\omega_i:=\omega_{i-1}+\eps_i \alpha_i, \omega_1=\omega$, where $\alpha_i\in N^1(X/T; W)$ is a general class and $0<\eps_i\ll 1$ for all $i\>1$. In particular, we may assume that $\omega+2(\omega _i-\omega)$ and $K_X+B+\omega_i$ are both K\"ahler over $W$ for $i\>1$. 
    \item $\lambda_i:=\inf\{s\>0\;:\; K_{X^i}+B^i+s\omega_i^i\mbox{ is nef over }W\}$.
    \item For each $i\>1$, $(K_{X^i}+B^i+\lambda_i\omega^i_i)^\bot\cap \NA(X^i/T; W)=R$ is an extremal ray. Moreover,  
    there is a point $w_i\in W$ and a curve $C_i\subset X^i_{w_i}$ spanning the ray $R$.
    \item $K_{X^i}+B^i+t\omega ^i_i$ is K\"ahler over $W$ for $0<t-\lambda_i \ll 1$. 
    \item There is a positive integer $n\>1$ such that 
    there is a morphism $X^n\to Z^n$ over $W$ such that $\dim X^n>\dim  Z^n$, $-(K_{X^n}+B^n)$ is relatively ample over $Z^n$ and $K_{X^n}+B^n+\lambda _n\omega ^n_n$ is relatively trivial over $Z^n$.
\end{enumerate}}
Note that this MMP is still disjoint from the fiber $X_{t_0}$.
 We explain the details of running this MMP below.
 Let $X:=X^1$ and $\lambda _0=1$.
     Suppose that $\phi ^{i-1}:X^1\dasharrow X^{i-1}$ have already been constructed so that properties (1-4)$^{i-1}$ are satisfied. In particular, by  (3-4)$^{i-1}$ we have that
 $K_{X^{i-1}}+B^{i-1}+t\omega^{i-1}_{i-1} = \phi ^{i-1}_*(K_X+B+t\omega_{i-1})$ is K\"ahler 
 for $0< t- \lambda _{i-1}\ll 1$ and $(K_{X^{i-1}}+B^{i-1}+\lambda_{i-1}\omega^{i-1}_{i-1})^\bot\cap \NA(X^{i-1}/T; W)=R_{i-1}$ is an extremal ray spanned by a curve $C_{i-1}$. If $R_{i-1}$ defines a Mori fiber space, then we are done.
 Otherwise, by what we argued above, we may assume that we have a flip, say $\psi ^{i-1}:X^{i-1}\dasharrow X^i$. 
 If $g^{i-1}:X^{i-1}\to Z^{i-1}$ and $h^i:X^i\to Z^{i-1}$ are the corresponding flipping and flipped contraction, then arguing as in the proof of Theorem \ref{thm:ss-pseff-mmp}, $\eta_{Z^{i-1}}:=g^{i-1}_*(K_{X^{i-1}}+B^{i-1}+\lambda_{i-1}\omega^{i-1}_{i-1})$ is K\"ahler over $W$. Since $\rho (X^i/Z^{i-1})=1$ and $K_{X^i}+B^i$ is ample over $Z^{i-1}$, it follows that $-\omega^{i}_{i-1}$ is K\"ahler over $Z^{i-1}$. Then for $0<\delta \ll 1$ we have
 \[K_{X^{i}}+B^{i}+(\lambda_{i-1}-\delta)\omega^{i}_{i-1}=\psi ^{i-1}_*(K_{X^{i-1}}+B^{i-1}+(\lambda_{i-1}-\delta)\omega^{i-1}_{i-1})=(h^i)^*\eta-\delta \omega^{i}_{i-1}\] is K\"ahler over $W$.
 Note that since $N^1(X/T;W)\to N^1(X^i/T;W)$ is surjective by Lemma \ref{lem:n1-pushforward}, and since $\alpha_i \in N^1(X/T;W)$ is a general class, then so is its pushforward  $\alpha^i_i \in N^1(X^i/T;W)$. In particular, $\omega^i_i=\omega^i_{i-1}+\epsilon _i\alpha^i_i$ is a general class in $N^1(X^i/T; W)$.
 Since $0<\epsilon _i\ll 1$, we may assume that \[K_{X^i}+B^i+(\lambda_{i-1}-\delta)\omega^{i}_{i}=K_{X^i}+B^i+(\lambda_{i-1}-\delta)\omega^{i}_{i-1}+\epsilon _i(\lambda_{i-1}-\delta)\alpha^{i}_{i}\] is K\"ahler over $W$.
 Let $\lambda _i:=\inf\{s\>0\;:\; K_{X^i}+B^i+s\omega_i^i\mbox{ is nef over }W\}$. Clearly  property (2)$^i$ is satisfied. Since $0<\epsilon _i\ll 1$, \[K_X+B+\omega _i=K_X+B+\omega _{i-1}+\epsilon _i \alpha _i\qquad {\rm and}\qquad \omega +2(\omega_i -\omega )=\omega+2(\omega_{i-1} -\omega  )+2\epsilon _i \alpha _i,\] property (1)$^{i-1}$ implies property (1)$^{i}$.
 
 To see (3)$^i$ we proceed as follows.
 We write  \begin{equation}\label{eqn:kahler-scaling}
     K_{X^i}+B^i+\lambda _i\omega _i^i=\frac 1{m+1}\left(K_{X^i}+B^i+m\left(K_{X^i}+B^i+\left(\frac {m+1} m\right)\lambda _i\omega _i^i \right)\right).
 \end{equation}
  For $m\gg 0$, $\lambda _i<\lambda _i\left(\frac {m+1} m\right)\leq \lambda _{i-1}-\delta$, and hence $K_{X^i}+B^i+\lambda _i\left(\frac {m+1} m\right)\omega ^i_i$ is K\"ahler over $W$. 
  From Theorem \ref{thm:weak-cone0} it easily follows that the face $F=(K_{X^i}+B^i+\lambda_i\omega^i_i)\cap\NA(X/T;W)$ is generated finitely many classes of curves.
 
  Since $\omega _i^i$ is general in $N^1(X^i/T; W)$, it follows that $(K_{X^i}+B^i+\lambda_i\omega^i_i)^\bot\cap \NA(X^i/T; W)=R_i$ is an extremal ray spanned by a curve $C_i\subset X^i_{w_i}$ for some $w_i\in T$ and so (3)$^i$ holds.\\
 
To see (4)$^i$, simply note that  the sum of a nef class and a K\"ahler class is K\"ahler, and hence $K_{X^i}+B^i+t\omega ^i_i$ is K\"ahler over $W$ for $\lambda _{i-1}-\delta \geq t>\lambda _i$ and $\delta>0$.

Finally, we must show that the process terminates after finitely many steps.
 
We claim that $\lim\lambda_i>0$. By contradiction assume that $\lim \lambda _i=0$.
For a very general $t\in T$, we have $X_t\cong X^i_t$ for all $i\>1$ (as discussed above).
By Lemma \ref{l-psef}, there exists a $\mu >0$ such that $K_{X_t}+B_t+\mu \omega _t$ is not pseudo-effective for very general $t\in T$. 
Since \[K_{X_t}+B_t+\mu \omega _t = K_{X_t}+B_t+\lambda _i(\omega_i)_t+(\mu \omega _t-\lambda _i(\omega_i)_t) \] and $\mu \omega _t-\lambda _i(\omega_i)_t$ is K\"ahler for $i\gg 0$ (as $\lim\lambda_i=0$), it follows that $K_{X_t}+B_t+\lambda _i(\omega_i)_t$ is not pseudo-effective for $i\gg 0$. 
Since
\[K_{X_t}+B_t+\lambda _i(\omega_i)_t = K_{X^i_t}+B^i_t+\lambda _i(\omega^i_i)_t\] 
is nef (for $t\in T$ very general), this is the required contradiction. So $\lim \lambda _i=\lambda>0$.  

Now for a fixed point $w_0\in W$, let $C_{w_0}\subset X_{w_0}$ be a flipping curve of the above MMP. Note that every step of the above MMP is also a step of the $(K_X+B+X_{w_0})$-MMP over $W$. Thus by special termination, after finitely many steps the flipping locus of the above MMP is disjoint from the fiber $X_{w_0}$. So after passing to a subsequence we may assume that for each $i\>1$, $t_i\in W$ is a point such that the fiber $X_{t_i}$ contains a flipping curve of the above MMP for the very first time. Consequently, we have that $X=X^1\bir X^i$ is an isomorphism over a neighborhood of $t_i$; in particular, $X_{t_i}\cong X^i_{t_i}$. Let $C_i\subset X^i_{t_i}$ be a flipping curve of the above MMP as in Theorem \ref{thm:weak-cone0}. Then identifying $C_i$ with its image in $X_{t_i}$ we get
 \[ (K_{X}+B+\lambda _i\omega _i)\cdot C_i=(K_{X^i}+B^i+\lambda _i\omega^i_i )\cdot C_i=(K_{X^i_{t_i}}+B^i_{t_i}+\lambda _i(\omega ^i_i)_{t_i})\cdot C_i=0.\]
 Since $\lambda _i\geq \lambda >0$, and $2\omega_i-\omega =\omega +2(\omega_i-\omega )$ is K\"ahler, it follows that
 \[\omega \cdot C_i\leq 2\omega^i \cdot C_i=2(\omega^i_i)_{t_i}\cdot C_i=\frac{-2}{\lambda _i}(K_{X^i_{t_i}}+B^i_{t_i})\cdot C_i\leq \frac {12}\lambda,\] and so by Lemma \ref{l-douady}, the curves $\{C_i\}_{i}$, belong to finitely many families of curves on $X$ (over $W$). Consequently, the curves $\{C_i\}_i$ are contained in finitely many fibers $X_{t_1},\ldots , X_{t_k}$,  where $t_i\in W$, and hence by special termination this sequence of flips must terminate, this is a contradiction. Therefore, we may assume that $K_{X^m}+B^m+\lambda_m \omega^m$ is nef for some $m\>1$, and there is a Mori fiber space $X^m\to Z$ over $T$.

\end{proof}

\begin{proof}[Proof of Theorem \ref{thm:ss-mmp}]
This follows from Theorem \ref{thm:ss-pseff-mmp} and \ref{thm:ss-non-pseff-mmp}.
\end{proof}

\bibliographystyle{hep}
\bibliography{4foldreferences} 

\end{document}